\documentclass{amsart}
\usepackage{amsmath}
\usepackage{amsthm}
\usepackage{amssymb}

\usepackage[utf8]{inputenc}
\usepackage[T1]{fontenc}
\usepackage{lmodern}
\usepackage{eucal}
\pdfoutput=1 
\usepackage{microtype}
\usepackage{url}
\usepackage{enumitem}
\usepackage{tikz-cd}
\usepackage{bbm}

\usepackage[letterpaper, left = 1.5in, right = 1.5in, top = 1.5in, bottom = 1.5in]{geometry}

\swapnumbers

\usepackage{hyperref}

\usepackage{xcolor}
\hypersetup{
    colorlinks,
    linkcolor={red!50!black},
    citecolor={blue!50!black},
    urlcolor={blue!80!black}
}

\usepackage[capitalise,nameinlink]{cleveref}
\crefname{subsection}{Subsection}{Subsections}
\crefname{subsubsection}{Subsection}{Subsections}

\theoremstyle{definition}
\newtheorem{theorem}{Theorem}[subsection]
\newtheorem{defn}[theorem]{Definition}
\newtheorem{ex}[theorem]{Example}

\newtheorem{lemma}[theorem]{Lemma}
\newtheorem{prop}[theorem]{Proposition}
\newtheorem{rmk}[theorem]{Remark}
\newtheorem{warning}[theorem]{Warning}
\newtheorem*{rmk*}{Remark}
\newtheorem*{ex*}{Example}
\newtheorem*{theorem*}{Theorem}
\newtheorem*{defn*}{Definition}

\newcommand{\bbZ}{\mathbb{Z}}
\newcommand{\bbG}{\mathbb{G}}
\newcommand{\bbH}{\mathbb{H}}
\newcommand{\bbF}{\mathbb{F}}
\newcommand{\bbN}{\mathbb{N}}
\newcommand{\bbP}{\mathbb{P}}
\newcommand{\bbE}{\mathbb{E}}
\newcommand{\bbQ}{\mathbb{Q}}
\newcommand{\bbR}{\mathbb{R}}
\newcommand{\bbS}{\mathbb{S}}

\newcommand{\bbA}{\mathbb{A}}
\newcommand{\bbT}{\mathbb{T}}
\newcommand{\bbC}{\mathbb{C}}

\newcommand{\bbL}{\mathbb{L}}

\newcommand{\bbone}{\mathbbm{1}}

\newcommand{\calM}{\mathcal{M}}

\newcommand{\calO}{\mathcal{O}}
\newcommand{\calC}{\mathcal{C}}
\newcommand{\calD}{\mathcal{D}}

\newcommand{\calF}{\mathcal{F}}

\newcommand{\calU}{\mathcal{U}}
\newcommand{\calA}{\mathcal{A}}
\newcommand{\calP}{\mathcal{P}}

\newcommand{\calB}{\mathcal{B}}
\newcommand{\calX}{\mathcal{X}}

\newcommand{\calH}{\mathcal{H}}

\newcommand{\calR}{\mathcal{R}}

\newcommand{\frakm}{\mathfrak{m}}

\newcommand{\Cat}{\mathcal{C}\mathrm{at}}
\newcommand{\Sp}{\mathcal{S}\mathrm{p}}
\newcommand{\Mod}{\mathcal{M}\mathrm{od}}
\newcommand{\Alg}{\mathcal{A}\mathrm{lg}}
\newcommand{\Gpd}{\mathcal{G}\mathrm{pd}}
\newcommand{\CAlg}{\mathcal{C}\mathrm{Alg}}
\newcommand{\Ab}{\mathcal{A}\mathrm{b}}
\newcommand{\Fun}{\mathrm{Fun}}

\newcommand{\Set}{\mathcal{S}\mathrm{et}}

\newcommand{\LMod}{\mathrm{L}\mathcal{M}\mathrm{od}}
\newcommand{\Fin}{\mathcal{F}\mathrm{in}}

\newcommand{\Model}{\mathcal{M}\mathrm{odel}}
\newcommand{\Ch}{\mathrm{Ch}}
\newcommand{\DL}{\mathrm{DL}}
\newcommand{\Ring}{\mathcal{R}\mathrm{ing}}
\newcommand{\CRing}{\mathrm{C}\mathcal{R}\mathrm{ing}}
\newcommand{\CoAlg}{\mathrm{Co}\mathcal{A}\mathrm{lg}}
\newcommand{\CMon}{\mathrm{C}\mathcal{M}\mathrm{on}}
\newcommand{\Grp}{\mathcal{G}\mathrm{rp}}
\newcommand{\coBiAlg}{\mathrm{co}\mathcal{B}\mathrm{i}\mathcal{A}\mathrm{lg}}
\newcommand{\Pleth}{\mathcal{P}\mathrm{leth}}
\newcommand{\Tamb}{\mathcal{T}\mathrm{amb}}
\newcommand{\Ind}{\operatorname{Ind}}
\newcommand{\Pro}{\operatorname{Pro}}
\newcommand{\QCoh}{\mathcal{QC}\mathrm{oh}}

\newcommand{\Map}{\operatorname{Map}}
\newcommand{\Hom}{\operatorname{Hom}}

\newcommand{\Fib}{\operatorname{Fib}}
\newcommand{\Tor}{\operatorname{Tor}}
\newcommand{\Psh}{\operatorname{Psh}}
\newcommand{\Ext}{\operatorname{Ext}}

\newcommand{\coker}{\operatorname{coker}}
\newcommand{\colim}{\operatorname*{colim}}

\newcommand{\im}{\operatorname{Im}}
\renewcommand{\ker}{\operatorname{Ker}}

\newcommand{\EXT}{\mathcal{E}\mathrm{xt}}
\newcommand{\ev}{\mathrm{ev}}
\newcommand{\gr}{\operatorname{gr}}
\newcommand{\Sq}{\mathrm{Sq}}
\newcommand{\tins}[2]{\mathop{{}_{#1}\otimes_{#2}}}
\newcommand{\tims}[2]{\mathop{{}_{#1}\times_{#2}}}
\newcommand{\TAQ}{\operatorname{TAQ}}
\newcommand{\cotimes}{\mathbin{\widehat{\otimes}}}
\newcommand{\Spec}{\operatorname{Spec}}
\newcommand{\Spf}{\operatorname{Spf}}
\newcommand{\sch}{\operatorname{Sp}^\vee}
\newcommand{\wTAQ}{\widehat{\TAQ}{}}
\newcommand{\wAQ}{\widehat{\bbT\mathrm{AQ}}{}}
\newcommand{\AQ}{\operatorname{\bbT AQ}}
\newcommand{\Def}{\operatorname{Def}}
\newcommand{\Sub}{\operatorname{Sub}}
\newcommand{\nul}{\mathrm{nul}}
\newcommand{\Coord}{\mathrm{Coord}}
\newcommand{\Level}{\operatorname{Level}}

\newcommand{\loc}{\mathrm{loc}}
\newcommand{\heart}{\heartsuit}
\newcommand{\h}{\mathrm{h}}
\newcommand{\op}{\mathrm{op}}
\newcommand{\cn}{\mathrm{cn}}
\newcommand{\free}{\mathrm{free}}

\newcommand{\Cpl}{\mathrm{Cpl}}
\newcommand{\aug}{\mathrm{aug}}
\newcommand{\un}{\mathrm{un}}

\newcommand{\bs}{{-}}
\newcommand{\bbs}{{=}} 

\newcommand{\ol}{\overline}
\def\hyp{{\hbox{-}}}

\newcommand\xqed[1]{%
  \leavevmode\unskip\penalty9999 \hbox{}\nobreak\hfill
  \quad\hbox{#1}}
\newcommand\tqed{\xqed{$\triangleleft$}}

\makeatletter
\DeclareRobustCommand{\tvdots}{%
  \vbox{\baselineskip4\p@\lineskiplimit\z@\kern0\p@\hbox{.}\hbox{.}\hbox{.}}}
\makeatother

\makeatletter
\@namedef{subjclassname@2020}{\textup{2020} Mathematics Subject Classification}
\makeatother

\begin{document}

\title{Algebraic theories of power operations}

\author{William Balderrama}


\begin{abstract}
We develop and exposit some general algebra useful for working with certain algebraic structures that arise in stable homotopy theory, such as those encoding well-behaved theories of power operations for $\bbE_\infty$ ring spectra.
In particular, we consider Quillen cohomology in the context of algebras over algebraic theories, plethories, and Koszul resolutions for algebras over additive theories.
By combining this general algebra with obstruction-theoretic machinery, we obtain tools for computing with $\bbE_\infty$ algebras over $\bbF_p$ and over Lubin-Tate spectra. 
As an application, we demonstrate the existence of $\bbE_\infty$ periodic complex orientations at heights $h\leq 2$.
\end{abstract}

\maketitle

\section{Introduction}

Let $R$ be an $\bbE_\infty$ ring spectrum. Then there is a theory of \textit{$R$-power operations}: for any $\bbE_\infty$ algebra $A$ over $R$, there are natural maps
\[
R_b(S^a)^{\wedge n}_{\h\Sigma_n}\times\pi_a A\rightarrow\pi_b A
\]
refining the $n$'th power map. As with all natural operations, these are immediately applicable to nonexistence theorems. For example, a map $\phi\colon \pi_\ast A\rightarrow \pi_\ast B$ that fails to be compatible with these operations must also fail to refine to a commutative $R$-algebra map $A\rightarrow B$. The converse is false, making existence theorems more subtle. As a general heuristic, by understanding the global structure of these operations, one can start to quantify the failure of the converse by introducing a suitable obstruction theory.

A collection of methods for realizing this heuristic was studied in \cite{balderrama2021deformations}, with obstruction groups given in terms of abstractly defined algebraic theories of operations. In some cases, these algebraic theories can be identified, and are sufficiently well-behaved that the corresponding obstruction groups can be described. Two important examples of this are the theories of power operations for $\bbE_\infty$ algebras over $\bbF_p$, and for $K(h)$-local $\bbE_\infty$ algebras over a Lubin-Tate spectrum of height $h$. These two examples turn out to share a number of properties, and the purpose of this paper is twofold: first, to study the general algebraic context in which they may be put on the same footing; second, to elaborate on the particulars of the resulting obstruction-theoretic tools available for these objects.

The tools we obtain for working with $K(h)$-local $\bbE_\infty$ algebras over a Lubin-Tate spectrum $E$ of height $h$ are new, and turn out to be very pleasant at heights $h\leq 2$: many obstruction groups vanish, and those which do not are fairly computable. A special case of this leads to the following theorem, which motivated much of this work.

\begin{theorem*}[\cref{thm:orientations}]
Let $E$ be a Lubin--Tate spectrum of height $h\leq 2$, and let $MUP$ be the $2$-periodic complex cobordism spectrum, defined as the Thom spectrum of the tautological bundle over $\bbZ\times BU$. Then every $H_\infty$ orientation $MUP\rightarrow E$ refines to an $E_\infty$ orientation.
\tqed
\end{theorem*}


\subsection*{Acknowledgements}

This paper was originally the second half of the author's thesis at the University of Illinois Urbana-Champaign, and benefited from discussions with many people over this period. Particular thanks are due to Charles Rezk for his support over the past several years, without which this work would not have been carried out.

\subsection{Preliminaries}

The rest of this section gives an overview of the main topics of the paper. The paper splits into two halves. The first half, consisting of \cref{sec:theories}, \cref{sec:koszul}, and \cref{sec:plethories}, consists of general theory, and is entirely algebraic in nature. The second half, consisting of \cref{sec:p} and \cref{sec:e}, combines the first half with work of \cite{balderrama2021deformations} to produce some tools for working with $\bbE_\infty$ algebras over $\bbF_p$ and with $K(h)$-local $\bbE_\infty$ algebras over Lubin-Tate spectra of height $h$.

The general algebraic context in which we work is that of algebraic theories, which we shall refer to just as \textit{theories}. We describe the particular variant that will be used in \cref{ssec:malcev}, recalling what we need from \cite{balderrama2021deformations}. In short, for us a theory will be a category $\calP$ admitting all small coproducts, satisfying some additional conditions which serve to make the infinitary aspects of $\calP$ well-behaved. The category $\Model_\calP^\heart$ of set-valued models of $\calP$ is then the category of presheaves of sets on $\calP$ that send coproducts in $\calP$ to products of sets.

We are interested in topics such as the Quillen cohomology of models of a theory $\calP$, and for this we find it most convenient to work $\infty$-categorically. Thus we shall freely use the language and theory of $\infty$-categories, as developed by Lurie \cite{lurie2017highertopos}, and by default all of our constructions should be interpreted in this manner. For example, we shall write $\Model_\calP$ for the $\infty$-category---we shall just say category---of $\infty$-groupoid-valued models of $\calP$. When $\calP$ is a discrete theory, i.e.\ a $1$-category, $\Model_\calP$ is the underlying $\infty$-category of Quillen's model category of simplicial set-valued models of $\calP$, and much of what we do could be easily translated to this language.

\begin{rmk}
The central notion studied in \cite{balderrama2021deformations} was a certain homotopical variant of an algebraic theory, called a \textit{loop theory}. For the most part, these play no role in the material developed in this paper. They do appear, at least implicitly, in \cref{sec:p} and \cref{sec:e}: in short, they constitute the technical machinery being used in the background to construct the various spectral sequences and obstruction theories to which we apply the algebra developed in this paper. See especially the proof of \cref{thm:pmaps} for further discussion. However, we have done our best to arrange things so that a deep knowledge of \cite{balderrama2021deformations} is not needed to understand these applications.
\tqed
\end{rmk}

\subsection{Algebra}

We now give an overview of the general algebraic topics we will cover. In all the following, $\calP$ will refer to a discrete theory.

\subsubsection{Algebras over theories and other topics}(\cref{sec:theories}).\label{sssec:algebrasintro}

There are two fundamental facts which serve to make the theory of power operations for $\bbE_\infty$ algebras in positive characteristic, or for $K(h)$-local $\bbE_\infty$ algebras over a Lubin-Tate spectrum of height $h$, well-behaved from our perspective. By way of example, let $\DL$ be the theory of power operations for $\bbE_\infty$ algebras over $\bbF_p$; abstractly, this theory may be defined as $\DL = \h\CAlg_{\bbF_p}^\free$, the homotopy category of the category of free $\bbE_\infty$ algebras over $\bbF_p$. In addition, write $\Ring_\DL = \Model_\DL$. The two key properties of this theory are the following:
\begin{enumerate}
\item The forgetful functor $\Ring_\DL^\heart\rightarrow\CRing_{\bbF_{p\ast}}^\heart$ is both monadic and comonadic;
\item The free models of $\DL$ have discrete and projective Andr\'e-Quillen homology.
\end{enumerate}
The first fact implies that $\Ring_\DL^\heart$ is built in a simple way from the more familiar category $\CRing_{\bbF_{p\ast}}^\heart$ of graded commutative $\bbF_p$-rings, and the second fact implies that this continues to hold at the level of derived categories. Using these one may reduce the Quillen cohomology of $\DL$-rings into a purely classical part, determined by the ordinary Andr\'e-Quillen homology of $\bbF_{p\ast}$-rings, and a purely linear part, governed by a suitable algebra of Dyer-Lashof operations with instability conditions. The use of commutative rings in particular is not necessary here, and instead we are led to the more general notion of an \textit{algebra over a theory}, due to Freyd \cite{freyd1966algebra} and Wraith \cite{wraith1971algebras}. 

\begin{defn*}[\cref{def:algebras}]
Let $\calP$ be a theory. A (discrete) \textit{$\calP$-algebra} $F$ consists of either of the following equivalent pieces of data:
\begin{enumerate}
\item A colimit-preserving monad $F$ on $\Model_\calP^\heart$;
\item A limit-preserving comonad $F^\vee$ on $\Model_\calP^\heart$.
\end{enumerate}
In this context, there is an adjunction $F\dashv F^\vee$, and $\Alg_F\simeq\CoAlg_{F^\vee}$.
\tqed
\end{defn*}

This definition may be understood as follows. Let $F$ be a monad on $\Model_\calP^\heart$, and let $\Model_F^\heart$ be the category of models for the theory $F\calP\subset\Alg_F$ obtained as the essential image of $\calP$ under $F$. Suppose moreover that $F$ preserves reflexive coequalizers; this is necessary to ensure $\Model_F^\heart\simeq\Alg_F$. Given $P,P'\in\calP$, the sets $F(P)(P')$ describe \textit{natural operations} on $F$-models. Explicitly, where $\ev_P$ denotes evaluation at $P$, there is a natural isomorphism
\[
F(P)(P')\cong\Hom_{\Fun(\Model_F^\heart,\Set)}(\ev_P,\ev_{P'});
\]
this is a consequence of the Yoneda lemma (cf.\ \cref{prop:yoneda}).

As $F$ preserves reflexive coequalizers, for $F$ to be an algebra it is sufficient that $F$ preserves coproducts. As
\[
\Hom_{\Fun(\Model_F^\heart,\Set)}\left(\prod_{i\in I}\ev_{P_i},\ev_{P'}\right)\cong \left(\coprod_{i\in I}F(P_i)\right)(P),
\]
we find that if $F$ preserves coproducts then natural operations $\prod_{i\in I}\ev_{P_i}\rightarrow\ev_{P'}$ on $F$-models are generated by operations $\ev_{P_i}\rightarrow\ev_{P'}$ together with operations already defined for models of $\calP$. In general, one may heuristically view $\calP$-algebras as those theories obtained from $\calP$ by adjoining only additional \textit{unary} operations and relations.

All of this is useful for understanding the ordinary algebra of $\Model_F^\heart$. However, our primary interest is in derived invariants; for instance, we would like to compute the Quillen cohomology of $F$-models. The first step in carrying out these computations is to understand abelianization for $F$-models, and here it is the right adjoint $F^\vee$ that makes algebras useful. This right adjoint may be identified explicitly as $F^\vee(X)(P) = \Hom_\calP(F(P),X)$; informally, $F^\vee$ is representable with representing object $F$. The key observation is that because $F^\vee$ preserves limits, it preserves all kinds of algebraic structure.

This plays out as follows. Write $D$ for abelianization, so that $\Ab(\Model_\calP^\heart)\simeq\LMod_{D\calP}^\heart$ and $\Ab(\Model_F^\heart)\simeq\LMod_{DF}^\heart$. Then $DF$ is an algebra over $D\calP$, informally obtained by linearizing the unary operations used to form $F$. If $B$ is an $F$-model, then we may consider the abelianizaton $DB$ of its underlying $\calP$-model; the action of $F$ on $B$ then linearizes to an action of $DF$ on $DB$, and this provides a model for the abelianization of $B$ as an $F$-model (\cref{prop:abelianizecomonad}). If $F$ satisfies an additional smoothness condition, then all of this may be made to work for derived abelianization as well (\cref{prop:smoothab}). In this case, Quillen cohomology in $\Model_F$ is built in a simple way from Quillen homology in $\Model_\calP$ and the ordinary homological algebra of $\LMod_{DF}$. See \cref{sssec:pex} below for an example of this general recipe.

\cref{sec:theories} also covers some additional topics useful for working with theories of operations. Here we will just highlight one: in \cref{ssec:compositions}, we recall the concept of a \textit{distributive law}, as discovered by Beck \cite{beck1969distributive}. These turn out to be the key to answering a number of basic questions that come up when working with theories. As one simple example, if $\calP$ is a theory, $F$ is a $\calP$-algebra, and $A$ is an $F$-model, then distributive laws allow one to understand the manner in which $\Ab(A/\Model_F^\heart)$ is built from $A$ and $\Ab(\Model_F^\heart)$.

\subsubsection{Plethories}(\cref{sec:plethories}).

\cref{sec:koszul} and \cref{sec:plethories} are essentially independent of each other, and we will begin by describing the latter, where we consider the special case of algebras over a theory of commutative rings. Let $\calP$ be an additive symmetric monoidal theory (cf.\ \cref{ssec:monoidalt}), and let $S$ be the free commutative monoid monad on $\LMod_\calP^\heart$, so that $\Model_{S\calP}^\heart\simeq\CMon(\LMod_\calP^\heart)$. A \textit{$\calP$-plethory} is equivalent to the data of an $S\calP$-algebra (\cref{def:pleth}). These generalize the biring triples of Tall-Wraith \cite{tallwraith1968representable} and plethories of Borger-Wieland \cite{borgerwieland2005plethystic}. Abelianization for rings over a plethory fits into the general story of abelianization for models over an algebra over a theory, and we unravel this more explicitly in \cref{ssec:plethaq}.

One of the features that distinguishes rings over a plethory from more familiar algebraic structures is the presence of \textit{nonlinear} structure. To successfully work with such plethories is to avoid dealing with this nonlinear structure by any means possible. It turns out that many plethories of interest are determined by their additive operations; here the classic example is the plethory of $\theta$-rings, also known as $\delta$-rings \cite{joyal1985delta}. In short, a $\theta$-ring is an ordinary commutative ring $R$ equipped with an operation $\theta\colon R\rightarrow R$ satisfying the identities that ensure $\psi(x) = x^p+p\theta(x)$ is generically a ring homomorphism. The operation $\theta$ is nonlinear, but it is determined by $\psi$ if $R$ is $p$-torsion free; as free $\theta$-rings are $p$-torsion free, it follows that the entire concept of a $\theta$-ring is encoded by the operation $\psi$ together with knowledge that $\psi(x)\equiv x^p\pmod{p}$. This is a $p$-typical analogue of the Wilkerson criterion for lambda rings \cite[Proposition 1.2]{wilkerson1982lambda}. Work of Rezk \cite{rezk2009congruence} shows that the theory of power operations for $K(h)$-local $\bbE_\infty$ algebras over a Lubin-Tate spectrum of height $h$ is similar in form to the theory of $\theta$-rings, and this motivates studying the general situation.

Under a minor flatness assumption, it is possible to package together the additive operations of a plethory into a single algebraic object, using a generalization of the additive bialgebras considered in \cite[Section 10]{borgerwieland2005plethystic}, as follows. Let $\Lambda$ be a $\calP$-plethory, and write $\Ring_\Lambda^\heart$ for the category of $\Lambda$-models, which we call $\Lambda$-rings. Given $P,P'\in\calP$, write
\[
\Gamma_{P,P'} = \Hom_{\Fun(\Ring_\Lambda^\heart,\Ab)}(\ev_P,\ev_{P'}).
\]
This is the set of natural additive operations $\ev_P\rightarrow\ev_{P'}$ on $\Lambda$-rings, and defines the morphisms for a theory $\Gamma = \Gamma(\Lambda)$ with the same objects as $\calP$. There is more structure present on $\Gamma$, encoding how these additive operations interact with the underlying $\calP$-ring structure of $\Lambda$-rings. This extra structure may be summarized by saying that the forgetful functor $\LMod_\Gamma^\heart\rightarrow\LMod_\calP^\heart$ is equipped with the structure of a strong symmetric monoidal functor (\cref{thm:plethcobi}), and we say that $\Gamma$ is a \textit{$\calP$-cobialgebroid} (\cref{def:cobialgebroid}). When $\calP$ is the theory of ordinary commutative rings over a commutative ring $R$, then $\calP$-cobialgebroids, i.e.\ strong symmetric monoidal functors to $\Mod_R^\heart$ which are both monadic and comonadic, are equivalent to the twisted $R$-bialgebras used in \cite{borgerwieland2005plethystic} (\cref{ex:twistedbialgebra}). The abstract definition in terms of symmetric monoidal functors is more convenient when working in multisorted contexts.

\subsubsection{Koszul algebras}(\cref{sec:koszul}).

The material described above serves to reduce various calculations to purely linear calculations. In \cref{sec:koszul}, we turn to the topic of \textit{additive theories}. The primary goal of this section is to develop a robust theory of \textit{Koszul algebras}, in particular to gain access to their associated Koszul resolutions and Koszul complexes. Koszul algebras were first introduced by Priddy \cite{priddy1970koszul}, motivated by the example of the Steenrod algebra. A more general notion of Koszul algebra is necessary to incorporate examples of interest in homotopy theory. For example, many of the algebras derived from homotopy operations fail to contain their coefficient rings centrally; others fail to admit an augmentation; others may not even be algebras in the classical sense, perhaps due to the presence of instability conditions; and still others may be built out of richer algebraic objects, such as Mackey functors. In fact, all of these cases are encountered even when just considering analogues of the mod $2$ Steenrod algebra.

These can all be encoded as algebras over an additive theory, and so we are left with developing a theory of Koszul algebras over additive theories. Fortunately, once the correct definition is known, the development is not significantly more difficult than the classical case. Let $\calP$ be an additive theory, and let $F = \colim_{m\rightarrow\infty}F_{\leq m}$ be a filtered $\calP$-algebra, in the sense that $F_{\leq 0}$ is the initial algebra and multiplication on $F$ restricts to $F_{\leq n}\circ F_{\leq m}\rightarrow F_{\leq n+m}$. To any filtered algebra $F$ we may associate its associated graded algebra $\gr F$, and for any $F$-module $M$ the monadic bar resolution $C(F,F,M)$ may be filtered in such a way that $\gr C(F,F,M)\cong C(\gr F,\gr F,\ol{M})$; here $\ol{M}$ is $M$ considered as a $\gr F$-module via the augmentation on $\gr F$. 

\begin{defn*}[\cref{def:koszul}]
A filtered $\calP$-algebra $F$ is \textit{Koszul} if:
\begin{enumerate}
\item $F$ and $\gr F$ are projective, in the sense that they restrict to functors $\calP\rightarrow\calP$;
\item $H_n C(\gr F,\gr F,\ol{P})[m] = 0$ for $P\in\calP$ and $n\neq m$.
\tqed
\end{enumerate}
\end{defn*}

This definiton may be understood as follows. As $\gr F$ is augmented, one can define its homology $H_\ast(\gr F)$ and cohomology $H^\ast(\gr F)$ (\cref{ssec:homologycohomology}). Any time that $F$ is filtered and $M$ is projective over $\calP$, we may identify a subcomplex of $C(F,F,M)$ of the form $K(F,F,M) = F H_\ast(\gr F)(M)\subset C(F,F,M)$; the condition that $F$ is Koszul implies that this inclusion is a quasiisomorphism (\cref{thm:koszulres}). Thus $K(F,F,M)$ is a small projective resolution of $M$. In particular, given any $F$-model $N$, one may form the Koszul complexes $K_F(M,N) = \LMod_F(K(F,F,M),N) \cong\LMod_\calP(M,H^\ast(\gr F)(N))$, and these provide small models for $\EXT_F(M,N)$.

We describe these Koszul complexes explicitly in \cref{ssec:quadraticduality} and \cref{ssec:koszulcomplexes}. As in the classical story, $\gr F$ is a quadratic algebra, and $H^\ast(\gr F)$ is its quadratic dual (\cref{thm:quadraticduality}). This describes the graded objects $K_F(M,N)$ and pairings between these, and the differential admits a simple description in terms of this (\cref{thm:homogkoszuldifferential}, \cref{thm:nonhomogkoszul}).

\subsection{Applications}\label{ssec:examplesintro}

\cref{sec:p} and \cref{sec:e} serve both as applications and extended examples of all the above machinery. 

\subsubsection{\texorpdfstring{$\bbE_\infty$}{E-infinity} algebras over \texorpdfstring{$\bbF_p$}{F_p}}(\cref{sec:p}).\label{sssec:pex}

Let $R$ be an $\bbE_\infty$ ring with $p = 0$ in $\pi_0 R$, and let $\bbP\calR = \CAlg_R^\free$ be the essential image of the category $\Mod_R^\free$ of free $R$-modules, including suspensions and desuspensions, under the free functor $\bbP_R\colon\Mod_R\rightarrow\CAlg_R$. Then $\h\bbP\calR$ is a theory of power operations for $\bbE_\infty$ algebras over $R$. There is a forgetful functor $\Model_{\h\bbP\calR}^\heart\rightarrow\CRing_{R_\ast}^\heart$; this preserves colimits, and so realizes $\Model_{\h\bbP\calR}^\heart = \Ring_{\h\bbP\calR}^\heart$ as the category of rings over an $R_\ast$-plethory. This is a consequence of the K\"unneth isomorphisms $\pi_\ast\bbP_R(F\oplus F')\cong\pi_\ast\bbP_R F\otimes_{R_\ast}\pi_\ast\bbP_R F'$, which hold for $F,F'\in\Mod_R^\free$ due to a stronger fact: each $\pi_\ast\bbP_R\Sigma^a R$ is free as an $R_\ast$-ring. This is well-known when $R = \bbF_p$, and in general the Hopkins-Mahowald  Thom spectrum theorem may be used to produce an unstructured isomorphism $\pi_\ast\bbP_R\Sigma^a R\cong R_\ast\otimes\pi_\ast\bbP_{\bbF_p}\Sigma^a\bbF_p$ of rings.

Thus we are in a position to apply our machinery. By way of example, the recipe discussed above in \cref{sssec:algebrasintro} amounts to the following (cf.\ \cref{ssec:plethaq}). First, there are composition maps
\[
\pi_c \bbP_R\Sigma^b R\times\pi_b \bbP_R\Sigma^a R\rightarrow \pi_c\bbP_R\Sigma^a R;
\]
given $\alpha\colon \Sigma^c R\rightarrow\bbP_R\Sigma^b R$ and $\beta\colon\Sigma^b R\rightarrow\bbP_R\Sigma^a R$ in $\Mod_R$, their composition is the composite
\[
m\circ\bbP_R\beta\circ\alpha\colon \Sigma^cR\rightarrow\bbP_R\Sigma^b R\rightarrow\bbP_R\bbP_R\Sigma^a R\rightarrow\bbP_R\Sigma^a R.
\]
Each $\pi_\ast\bbP_R\Sigma^a R$ is augmented over $R_\ast$, and on indecomposables these compositions pass to
\[
Q(\pi_\ast\bbP_R\Sigma^b R)_c\otimes Q(\pi_\ast \bbP_R\Sigma^a R)_b\rightarrow Q(\pi_\ast\bbP_R \Sigma^a R)_c.
\]
Let $\LMod_{\Delta(R)}^\heart$ be the category of $R_\ast$-modules $M_\ast$ equipped with maps
\[
Q(\pi_\ast\bbP_R\Sigma^a R)_b\otimes M_a\rightarrow M_b
\]
satisfying the evident associativity and unitality conditions. Then there is an equivalence of categories $\Ab(\Ring_{\h\bbP\calR}^{\heart,\aug})\simeq\LMod_{\Delta(R)}^\heart$. More generally, given $C_\ast\in\Ring_{\h\bbP\calR}^\heart$, there is an equivalence $\Ab(\Ring_{\h\bbP\calR}^\heart/C_\ast)\simeq\LMod_{C_\ast\otimes_{R_\ast}\Delta(R)}^\heart$, where $C_\ast\otimes_{R_\ast}\Delta(R)$ is an object obtained from $C_\ast$ and $\Delta(R)$ via a certain distributive law. This extends to an explicit description of the Quillen cohomology of $\h\bbP\calR$-rings: given $A_\ast\rightarrow C_\ast$ in $\Ring_{\h\bbP\calR}^\heart$ and $M_\ast\in\LMod_{C_\ast\otimes_{R_\ast}\Delta(R)}^\heart$, there is an equivalence
\[
\calH_{\h\bbP\calR/C_\ast}(A_\ast;M_\ast)\simeq\EXT_{C_\ast\otimes_{R_\ast}\Delta(R)}(C_\ast\otimes_{A_\ast}^\bbL\bbL\Omega_{A_\ast|R_\ast},M_\ast),
\]
where the left side is a Quillen cohomology space and the right side an $\Ext$ space (cf.\ \cref{ssec:additivenotation}, \cref{ssec:qcohomology}).

Except for the homotopical input that $\pi_\ast\bbP_R\Sigma^a R$ is free as an $R_\ast$-ring, all of this is just a specialization of the general algebraic material that has already been discussed.

In \cref{sec:p}, we consider the particular case $R = \bbF_p$; here the structure of power operations is well-understood, and so there is more that can be said.

Write $\DL = \h\bbP\bbF_p$ for the plethory of power operations for $\bbE_\infty$ algebras over $\bbF_p$. We recall the structure of these power operations in \cref{ssec:dl}; for instance, when $p=2$, a $\DL$-ring amounts to an $\bbF_{2\ast}$-ring $A_\ast$ equipped with additive maps $Q^s\colon A_n\rightarrow A_{n+s}$ for $s\in\bbZ$, which are subject to certain explicitly describable Adem relations, Cartan formulas, and instability conditions. The category of $\DL$-rings is closely related to the category of unstable rings over the Steenrod algebra, and we give the precise relation in \cref{ssec:unstablea}.

As above, one may form an object $\Delta$ with $\Ab(\Ring_\DL^\heart/C_\ast)\simeq\LMod_{C_\ast\otimes\Delta}^\heart$, and in this case $\Delta$ admits an explicit description in terms of generators and relations. The presence of instability conditions implies that $\Delta$ is not merely a $\bbZ$-graded $\bbF_p$-algebra, but it is an algebra for the theory of $\bbF_{p\ast}$-modules. The algebra $\Delta$ turns out to be Koszul, and we compute its cohomology in \cref{ssec:cohomologydl}. This gives access to Koszul complexes for computing with $\Delta$-modules, which can be regarded as nonconnective analogues of the unstable Koszul complexes considered by Miller \cite{miller1978spectral}.

In \cref{ssec:pmaps}, we elaborate on the mapping space obstruction theory of \cite{balderrama2021deformations} in this context, and give some examples. In \cref{ssec:paq}, we similarly describe a generalization of the Basterra spectral sequence for computing the Andr\'e-Quillen-Goodwillie towers of $\bbE_\infty$ rings in this context.

\subsubsection{Lubin-Tate spectra}(\cref{sec:e}).

We finally return to our original motivation in \cref{sec:e}, where we give applications to $K(h)$-local $\bbE_\infty$ algebras over a Lubin-Tate spectrum $E$ of height $h$. This ultimately takes the same form as for $\bbE_\infty$ algebras in positive characteristic: there is an $E_\ast$-plethory $\bbT$ satisfying all the niceness properties one could hope for, and there are obstruction theories for working with $K(h)$-local $\bbE_\infty$ algebras over $E$ with obstruction groups built from derived invariants of $\bbT$-rings. This turns out to not be obvious, due to some subtleties that arise in this context, but the framework has been set up to handle this.

In the end, this works out as follows. Building on work of Strickland \cite{strickland1998morava}, which in turn builds on calculations of Kashiwabara \cite{kashiwabara2001brownpeterson}, Rezk \cite{rezk2009congruence} has produced what is in our language an $E_\ast$-plethory $\bbT$ encoding the structure of $E$-power operations. In short, a $\bbT$-ring is very much like a $\theta$-ring, only with more additive operations, which satisfy more complicated relations; moreover, $\bbT$-rings admit a conceptual interpretation in terms of the deformation theory of isogenies of formal groups. We recall all of this in \cref{ssec:epow}. There are some additional properties special to $\bbT$; most notably, $\bbT$-rings are very close to being rings over an ungraded plethory. This is captured in \cite{rezk2009congruence} using the notion of a twisted $\bbZ/(2)$-graded category, and in \cref{ssec:evenperiodic} we give a different packaging in terms of even-periodic plethories. All of these facts together reduce many computations with $\bbT$-rings to computations over an ungraded cobialgebroid $\Gamma$ associated to $\bbT$; ungraded cobialgebroids are the coalgebraic analogue of formal category schemes, and we review this in \cref{ssec:quasicoherent}.

Now let $\calP = \CAlg_E^{\loc,\free}$ be the category of $K(h)$-localizations of free $\bbE_\infty$ algebras over $E$. Then $\h\calP$ is a theory of power operations for such $E$-algebras, but $\Model_{\h\calP}^\heart$ is not equivalent to the category of $\bbT$-rings: the $K(h)$-local condition enforces a completeness condition on homotopy groups, and this is not reflected in $\bbT$. This issue was first noted by Rezk \cite[Section 1.6]{rezk2009congruence}, and approaches to handling it have been studied by Barthel-Frankland \cite{barthelfrankland2015completed} and Brantner \cite{brantner2017lubin}, the latter also making use of algebraic theories. We give another approach to completing $\bbT$ in \cref{prop:dfcpl}, and describe how this interacts with Quillen cohomology in \cref{ssec:completions}; in the end, this goes as follows. Write $\frakm\subset E_0$ for the maximal ideal of $E_0$. Then there is a category $\Mod_{E_\ast}^{\Cpl(\frakm)}$ of $\frakm$-complete objects in the derived category of $E_\ast$-modules, and this is a localization of the dervied category of $E_\ast$-modules. Let $\Ring_\bbT^{\Cpl(\frakm)}$ be the homotopy theory of simplicial $\bbT$-rings whose underlying object of $\Mod_{E_\ast}$ is $\frakm$-complete; this is a localization of $\Ring_\bbT$. Then there is an equivalence of homotopy theories $\Model_{\h\calP}\simeq\Ring_\bbT^{\Cpl(\frakm)}$ (\cref{lem:derivedcompleterefl}). In particular, the Quillen cohomology of a model of $\h\calP$ agrees with the Quillen cohomology of its underlying $\bbT$-ring.

From this it follows that the work of \cite{balderrama2021deformations} provides obstruction theories for working with $\CAlg_E^\loc$ with obstruction groups built from derived invariants of $\bbT$-rings. These turn out to be very pleasant at heights $h\leq 2$, as a consequence of the following. Let $\Delta$ be the $E_\ast$-algebra such that $\Ab(\Ring_\bbT^{\aug,\heart})\simeq\LMod_\Delta$. A theorem of Rezk \cite{rezk2017rings} shows that $\Delta$ is a Koszul algebra, and moreover $H^n(\Delta) = 0$ for $n>h$. The theory of Koszul resolutions then implies that $\Ext^n_\Delta(M,N) = 0$ for $n>h$ whenever $M$ is a $\Delta$-module which is projective as an $E_\ast$-module. With a bit more work, all of this may be made to work for other slices of $\Ring_\bbT$, and to play well with completions; as a sample of the sort of subtlety that must be dealt with, observe that a $\Delta$-module whose underlying $E_\ast$-module is the completion of a projective $E_\ast$-module need not be the completion of a $\Delta$-module whose underlying $E_\ast$-module is projective.

Once all the details are dealt with, we are placed in a good position to give applications. In \cref{ssec:emaps}, we unravel what the mapping space obstruction theory of \cite{balderrama2021deformations} says in this context, and in \cref{ssec:etaq}, we discuss the topological Andr\'e-Quillen homology and cohomology of $K(h)$-local $\bbE_\infty$ algebras over $E$. One easily stated special case of the mapping space obstruction theory is the following.

\begin{theorem*}[\cref{thm:emapss}]
Let $E$ be a Lubin-Tate spectrum of height $h\leq 2$, and fix $A,B\in\CAlg_E^\loc$. Suppose:
\begin{enumerate}
\item $A_\ast$ and $B_\ast$ are concentrated in even degrees;
\item $A_0$ is the completion of a localization of a polynomial ring over $E_0$.
\end{enumerate}
Then every $\bbT$-ring map $A_\ast\rightarrow B_\ast$ lifts to a map $A\rightarrow B$ in $\CAlg_E$. Moreover, this lift is unique when $h=1$.
\tqed
\end{theorem*}

This follows quickly from our general machinery combined with bounds on obstruction groups implied by Rezk's theorem on $\Delta$, and the existence of $\bbE_\infty$ orientations at heights $h\leq 2$ is an immediate corollary (\cref{thm:orientations}).

\section{Algebraic theories}\label{sec:theories}

This section is largely a compilation of a number of definitions which are useful for understanding the algebraic structures we are interested in, together with a discussion of Quillen cohomology in certain special contexts.

\subsection{Review}\label{ssec:malcev}

We begin by reviewing some basic definitions and facts regarding algebraic theories that we will need.

\begin{defn}[{cf.\ \cite[Section 2]{balderrama2021deformations}}]
\hphantom{blank}
\begin{enumerate}
\item An \textit{algebraic theory} is a category $\calP$ which admits all small coproducts, and we say that $\calP$ is a \textit{discrete theory} if $\calP$ is a $1$-category.
\item A theory $\calP$ is \textit{$\kappa$-bounded} for a regular cardinal $\kappa$ if there exists a small full subcategory $\calP_0\subset\calP$ closed under $\kappa$-small coproducts and satisfying the following $\kappa$-compactness condition: for every $P_0\in\calP_0$ and set of objects $\{P_i:i\in I\}$ of $\calP$, the canonical map $\colim_{F\subset I,|F|<\kappa}\Map_\calP(P_0,\coprod_{i\in F}P_i)\rightarrow\Map_\calP(P_0,\coprod_{i\in I}P_i)$ is an equivalence. We say that $\calP$ is \textit{bounded} if $\calP$ is $\kappa$-bounded for some $\kappa$.
\item The category of \textit{models} of an algebraic theory $\calP$ is the full subcategory $\Model_\calP\subset\Psh(\calP)$ of small presheaves $X$ on $\calP$ such that for any set $\{P_i:i\in I\}$ of objects of $\calP$, the canonical map $X(\coprod_{i\in I}P_i)\rightarrow\prod_{i\in I}X(P_i)$ is an equivalence. The category of \textit{discrete models} of $\calP$ is the full subcategory $\Model_\calP^\heart\subset\Model_\calP$ of models whose underlying presheaf takes values in sets.
\item An algebraic theory $\calP$ is \textit{Mal'cev} provided it satisfies a certain additional condition, equivalent when $\calP$ is a discrete theory to the following: for every simplicial object $X\colon\Delta^\op\rightarrow\Model_\calP^\heart$ and every $P\in\calP$, the simplicial set $X(P)$ is a Kan complex.
\end{enumerate}
By \textit{theory} we will always refer to a bounded Mal'cev theory.
\tqed
\end{defn}

Throughout the paper, $\calP$ will always refer to some theory, possibly satisfying additional assumptions. We will abuse notation by implicitly identifying $\calP$ as a full subcategory of $\Model_\calP$. Except when giving some definitions and basic facts, $\calP$ will be a discrete theory. We will describe some additional notation in \cref{ssec:additivenotation} for models of $\calP$ in the case where $\calP$ is additive. 

The structure of the category of models of $\calP$ can be summarized as follows.

\begin{lemma}[{\cite[Section 2]{balderrama2021deformations}}]\label{prop:models}
\hphantom{blank}
\begin{enumerate}
\item $\Model_\calP$ is the free cocompletion of $\calP$ under geometric realizations, and these are preserved by the embedding $\Model_\calP\subset\Psh(\calP)$. In particular, if $X\in\Model_\calP$, then $\Map_\calP(X,\bs)$ preserves geometric realizations if and only if $X$ is a retract of some object of $\calP$.
\item Say $\calP$ is $\kappa$-bounded, and fix $\calP_0\subset\calP$ realizing this. Then $\Model_\calP$ is equivalent to the category of presheaves on $\calP_0$ which preserve $\kappa$-small coproducts. In particular, it is a $\kappa$-compactly generated presentable category.
\item If $\calP$ is discrete, then $\Model_\calP^\heart$ is the free $1$-categorical cocompletion of $\calP$ under reflexive coequalizers. In this case $\Model_\calP$ is the underlying $\infty$-category of Quillen's homotopy theory of simplicial objects in $\Model_\calP^\heart$, with localization realized by geometric realization.
\qed
\end{enumerate}
\end{lemma}

We think of a theory $\calP$ as encoding, and encoded by, natural operations on its models. This manifests as follows.

For $P\in\calP$, write $\ev_P\colon \Model_\calP\rightarrow\Gpd_\infty$ for the functor $\ev_P(X) = X(P)$.

\begin{prop}\label{prop:yoneda}
For $P,P'\in\calP$, there is a natural isomorphism
\[
\Hom_{\Fun(\Model_\calP,\,\Set)}(\pi_0\ev_{P},\pi_0\ev_{P'})\cong\pi_0\Map_\calP(P',P).
\]
\end{prop}
\begin{proof}
The Yoneda lemma gives a natural isomorphism
\[
\ev_P(X) \simeq \Map_{\calP}(P,X),
\]
and thus
\[
\pi_0\ev_P(X) \cong \pi_0\Map_{\calP}(P,X) \cong \Map_{\h\calP}(P,X).
\]
In other words, $\pi_0\ev_P$ is corepresented by $P$ as a functor on the homotopy category of $\Model_\calP$. We conclude with another application of the Yoneda lemma, yielding
\begin{align*}
\Hom_{\Fun(\Model_\calP,\,\Set)}(\pi_0\ev_P,\pi_0\ev_{P'})&\cong\Hom_{\Fun(\h\Model_\calP,\,\Set)}(\pi_0\ev_P,\pi_0\ev_{P'})\\
&\cong\Hom_{\h\calP}(P',P)\cong\pi_0\Map_\calP(P',P).
\qedhere
\end{align*}
\end{proof}

We are interested in theories primarily as a tool for accessing their models.

\begin{prop}[{\cite[Proposition 3.3.3]{balderrama2021deformations}}]\label{prop:additiveabelian}
\hphantom{blank}
\begin{enumerate}
\item If $\calP$ is a discrete additive theory, then $\Model_\calP^\heart$ is a complete and cocomplete abelian category with enough projectives;
\item If $\calA$ is a cocomplete abelian category and $\calP\subset\calA$ is a full subcategory consisting of projective objects and closed under coproducts such that every $M\in\calA$ is resolved by objects of $\calP$, then $\calA\simeq\Model_\calP^\heart$.
\qed
\end{enumerate}
\end{prop}

\begin{ex}\label{ex:additivetheories}
(1)~~If $R$ is an ordinary associative ring and $\calR$ is the category of free left $R$-modules, then $\calR$ is a theory, $\Model_\calR^\heart\simeq\LMod_R^\heart$ is equivalent to the category of ordinary left $R$-modules, and $\Model_\calR\simeq\LMod_R^\cn$ is equivalent to the category of connective modules over the Eilenberg-MacLane spectrum $HR$. This justifies calling $\calR$ the \textit{theory of (left) $R$-modules}.
  
(2)~~If $G$ is a finite group and $\calB_G$ is the Burnside category of finite $G$-sets, i.e.\ the additive completion of the category of finite $G$-sets and spans thereof, then $\calB_G$ is a finitary theory and $\Model_{\calB_G}^\heart$ is equivalent to the category of $G$-Mackey functors.

(3)~~Let $p$ be a prime and $\calP$ be the category of $p$-completions of free abelian groups. Then $\calP$ is an $\aleph_1$-bounded theory which is not $\omega$-bounded, i.e.\ is not associated to a finitary theory. The category $\Model_\calP^\heart$ is equivalent to the category of ${\Ext}\hyp p$-complete abelian groups in the sense of \cite[Section VI.2.1]{bousfieldkan1972monster}, and $\Model_\calP$ is equivalent to the category of connective $\bbZ$-modules which are $p$-complete in the sense of \cite{greenleesmay1995completions} or \cite[Chapter 7]{lurie2018spectral}. This example has also been observed in \cite[Propositions 4.1.2, 4.2.4]{brantner2017lubin}.
\tqed
\end{ex}

\begin{rmk}
At least up to Morita equivalence, finitary additive theories are equivalent to \textit{ringoids}, i.e.\ small $\Ab$-enriched categories: if $\calC$ is a finitary additive theory and $\calA\subset\calC$ is a subcategory generating $\calC$ under finite sums and retracts, then $\Model_\calC^\heart$ is equivalent to the category of left $\calA$-modules in the sense of \cite{mitchell1972rings}. For example, if $\calC$ is the theory of left modules over a ring $R$, then we may take $\calA\subset\calC$ to be the full subcategory on the single object $R$; here $\calA$ is equivalent to $R^\op$ viewed as an $\Ab$-enriched category with one object and $\LMod_R^\heart$ is equivalent to the category of additive functors $R\rightarrow\Ab$.
\tqed
\end{rmk}

Call a functor $U\colon\calD\rightarrow\calC$ \textit{strongly monadic} if $U$ preserves geometric realizations and is the forgetful functor of a monad adjunction. At least when $\calC$ itself admits geometric realizations, it is equivalent to ask that $\calD\simeq\Alg_T$ for a monad $T$ on $\calC$ which preserves geometric realizations. The monads that we encounter will generally be of this form, as these are the monads which play well with theories.

\begin{lemma}\label{prop:monadictheory}
If $T$ is an accessible monad on $\Model_\calP$ which preserves geometric realizations, and $T\calP\subset\Alg_T$ is the full subcategory spanned by the image of $\calP$ under $T$, then $T\calP$ is a theory and $\Model_{T\calP}\simeq\Alg_T$. Moreover, every conservative accessible functor $U\colon \calD\rightarrow\Model_\calP$ which preserves limits and geometric realizations arises this way.
\end{lemma}
\begin{proof}
That $T\calP$ is a theory is clear, and the equivalence $\Model_{T\calP}\simeq\Alg_T$ follows quickly from \cref{prop:models}. The final statement follows from the crude monadicity theorem \cite[Theorem 4.7.0.3]{lurie2017higheralgebra}.
\end{proof}

\begin{ex}\label{ex:polytheory}
Let $R$ be a commutative ring, and let $S\calR$ the category of polynomial $R$-rings. This is the essential image of the theory of $R$-modules under the free functor $S\colon\Mod_R^\heart\rightarrow\CRing_R^\heart$, so we can identify $\Model_{S\calR}^\heart\simeq\CRing_R^\heart$, and $\Model_{S\calR}\simeq\CRing_R$ is the homotopy theory of simplicial (commutative) $R$-rings.
\tqed
\end{ex}

\subsection{Bimodels}\label{ssec:bimodels}

Fix theories $\calP$ and $\calP'$.

\begin{lemma}\label{lem:bimodels}
The following concepts are equivalent:
\begin{enumerate}
\item Models of $\calP$ in $\Model_{\calP'}^\op$;
\item Left adjoint, or colimit-preserving, functors $H\colon\Model_\calP\rightarrow\Model_{\calP'}$, or equivalently, coproduct-preserving functors $H\colon\calP\rightarrow\Model_{\calP'}$;
\item Right adjoint, or limit-preserving accessible, or pointwise corepresentable, functors $H^\vee\colon\Model_{\calP'}\rightarrow\Model_\calP$.
\end{enumerate}
\end{lemma}
\begin{proof}
These follow directly from either \cref{prop:models} or the adjoint functor theorems for presentable categories \cite[Proposition 5.5.2.2, Corollary 5.5.2.9]{lurie2017higheralgebra}. In addition, we can make the corepresentability condition of (3) explicit: $H^\vee(M)(P) \simeq \Map_\calP(H(P),M)$.
\end{proof}

We call the concept encoded in \cref{lem:bimodels} that of a \textit{$\calP\hyp\calP'$-bimodel}; when $\calP = \calP'$, we will just call these $\calP$-bimodels. We refer the reader to Wraith \cite{wraith1969algebraic} and Freyd \cite{freyd1966algebra} for classical treatments of bimodels, as well as of algebras, defined below. We will consistently adhere to the convention that by bimodel we refer to the underlying left adjoint $H$, and that in this case $H^\vee$ is written for its right adjoint. For $P\in\calP$ and $P'\in\calP'$, we may at times write $H_{P,P'} = H(P)(P')$ and $H^\vee_{P,P'} = H^\vee(P)(P')$; note that these are covariant in the first variable and contravariant in the second. 

\begin{ex}
Let $\calP$ be the theory of groups, $R$ a commutative ring, and $\calP'$ the theory of commutative $R$-rings. To a commutative Hopf algebra $H$ over $R$, one may associate the functor
\[
\CRing_R(H,\bs)\colon\CRing_R\rightarrow\Grp.
\]
This is a $\calP'\hyp\calP$-bimodel, and all discrete $\calP'\hyp\calP$-bimodels arise this way.
\tqed
\end{ex}

Call a $\calP'\hyp\calP$-bimodel $H$ \textit{projective} if $H(P)$ is projective for all $P\in\calP$; equivalently, if $H$ restricts to a functor $H\colon\calP\rightarrow\calP'$, at least up to idempotent completion of $\calP'$.

If $\calP''$ is another theory and $H'$ is a $\calP''\hyp\calP'$-bimodel, then we can compose to obtain the $\calP''\hyp\calP$-bimodel $H'\circ H$. This has right adjoint $(H'\circ H)^\vee\simeq H^\vee\circ H'^\vee$.

\begin{rmk}
Although we are primarily interested in discrete bimodels, we are still interested in derived aspects of these. Even supposing that the theories in question are discrete, there are two possible ambiguities that arise:
\begin{enumerate}
\item Coproduct-preserving functors $H\colon\calP\rightarrow\Model_{\calP'}^\heart$ are not equivalent to coproduct-preserving functors $H\colon\calP\rightarrow \Model_{\calP'}$ such that $H(P)$ is discrete for all $P\in\calP$;
\item Even if $H\colon\Model_{\calP}\rightarrow\Model_{\calP'}$ and $H'\colon \Model_{\calP'}\rightarrow\Model_{\calP''}$ deserve to be called discrete bimodels, the same need not hold for the composite $H'\circ H$.
\end{enumerate}

The second ambiguity is not major, being no different than the ambiguity between a derived tensor product and a non-derived tensor product. The first ambiguity is more subtle, amounting to the observation that discrete models need not be closed under coproducts in the category of all models. When $\calP$ is additive, this amounts to the observation that infinite sums need not be exact in a general abelian category with enough projectives.

Neither of these will be major issues for us. In practice, where they might otherwise cause problems, we will simply assume that our bimodels are projective, at which point both of these ambiguities vanish. However, for the sake of avoiding projectivity assumptions where they are not relevant, we take the following convention. When we are dealing with the purely discrete aspects of bimodels, we take as our discrete bimodels those which correspond to coproduct-preserving functors $H\colon \calP\rightarrow\Model_{\calP'}^\heart$. When we are dealing with homotopical aspects of bimodels, we take as our discrete bimodels those which correspond to coproduct-preserving functors $H\colon\calP\rightarrow\Model_{\calP'}$ that land in $\Model_{\calP'}^\heart$.
\tqed
\end{rmk}

\begin{ex}\label{ex:bimodules}
Let $A$ and $B$ be ordinary associative algebras with theories $\calA$ and $\calB$ of left modules. Then discrete $\calB\hyp\calA$-bimodels are equivalent to $B\hyp A$-bimodules. It is worth spelling out some aspects of this example explicitly to indicate the conventions that arise for working with bimodules. We only consider discrete bimodels in this example, although similar observations hold in the derived setting (where general $\calB\hyp\calA$-bimodels are equivalent to connective modules over the ring spectrum $B\otimes_\bbS A^\op$). To a discrete $B\hyp A$-bimodule $H$, one can associate the bimodel
\begin{align*}
H&\colon\LMod_A^\heart\rightarrow\LMod_B^\heart,\qquad H(M)= H\otimes_A M;\\
H^\vee&\colon\LMod_B^\heart\rightarrow\LMod_A^\heart,\qquad H^\vee(M)=\Hom_B(H,M).
\end{align*}
Here, $B$ acts on $H\otimes_A M$ by
\[
b\cdot (h\otimes m) = (bh)\otimes m,
\]
and $A$ acts on $\Hom_B(H,M)$ by
\[
(a\cdot f)(h) = f(ha).
\]
The bimodel $H$ is projective precisely when the bimodule $H$ is projective as a left $B$-module. The dual functor $H^\vee$ encodes more information than the ordinary dual bimodule $\Hom_B(H,B)$, and the latter can be recovered from the former by considering the restriction of $H^\vee$ to the category of finitely generated free $B$-modules. On the other hand, if $H$ is projective, then we can equip $\Hom_B(H,B)$ with a natural topology as a right $B$-module such that $H^\vee(M)\simeq \Hom_B(H,B)\cotimes_B M$, where $B$ acts on $\Hom_B(H,B)$ on the right by $(f\cdot b)(h) = f(bh)$. 

If $C$ is another ordinary associative algebra, $\calC$ is its theory of left modules, and $H'$ is a discrete $\calC\hyp\calB$-bimodel, then under the correspondence between bimodules and bimodels we identify
\[
H'\circ H \simeq H'\otimes_B H.
\]
The isomorphism $(H'\circ H)^\vee\cong H^\vee\circ H'^\vee$ is given by the maps
\begin{gather*}
\theta\colon \Hom_B(H,\Hom_C(H',M))\rightarrow\Hom_C(H'\otimes_B H,M),\\
\theta(f)(h'\otimes h) = f(h)(h').
\end{gather*}
Taking $A=B=C$, this is an enhancement of the duality pairing
\begin{gather*}
\theta\colon\Hom_A(H,A)\otimes_A\Hom_A(H',A)\rightarrow\Hom_A(H'\otimes_A H,A),\\
\theta(f\otimes f')(h'\otimes h) = f'(h'f(h))
\end{gather*}
of bimodules.
\tqed
\end{ex}

\subsection{Algebras}\label{ssec:algebras}

\begin{defn}\label{def:algebras}
A \textit{$\calP$-algebra} consists of a $\calP$-bimodel $F$ together with the additional structure of a monad on $F$, or equivalently, of a comonad on $F^\vee$. An \textit{$F$-model} is an algebra for the monad $F$, or equivalently, coalgebra for the comonad $F^\vee$.
\tqed
\end{defn}
If $F$ is a $\calP$-algebra, then $\Alg_F\simeq\Model_{F\calP}$ by \cref{prop:monadictheory}; we will abbreviate this to $\Model_F$. The forgetful functor $\Model_F\rightarrow \Model_\calP$ is \textit{plethystic}: it is both monadic and comonadic. Conversely, every category plethystic over $\Model_\calP$ arises from a $\calP$-algebra. Heuristically, $\calP$-algebras are those theories that can be obtained from $\calP$ by adjoining sufficiently well-behaved unary operations and relations.

\begin{ex}\label{ex:algebras}
The following are examples of discrete algebras.

(1)~~Let $R$ be an ordinary associative algebra and $\calR$ be the theory of left $R$-modules. Then a discrete $\calR$-algebra is equivalent to a discrete $R$-bimodule $A$ equipped with the structure of a monoid in the category of $R$-bimodules. Thus discrete $\calR$-algebras are equivalent to ordinary associative algebras equipped with an algebra map from $R$; we will just call these \textit{ordinary $R$-algebras}. In particular, even when $R$ is commutative, it need not be central in its algebras.

(2)~~Let $R$ be a commutative ring and $S\calR$ be the the category of polynomial $R$-algebras, as in \cref{ex:polytheory}. Then discrete $S\calR$-algebras were studied by Tall-Wraith \cite{tallwraith1968representable} under the name of biring triples, and more recently by Borger-Wieland \cite{borgerwieland2005plethystic} under the name of $R$-plethories. Our main examples of algebras over nonadditive theories are essentially of this form. We will study the relevant context in \cref{sec:plethories}, allowing for bases more general than just commutative rings.

(3)~~Let $G$ be a finite group, $\calB_G$ be the Burnside category of finite $G$-sets, as in \cref{ex:additivetheories}, and $S\calB_G$ be the category of commutative green functors free on objects of $\calB_G$, so that $\Model_{S\calB_G}^\heart$ is the category of commutative green functors. Let $\Tamb_G^\heart$ be the category of $G$-Tambara functors. Then $\Tamb_G^\heart\rightarrow \Model_{S\calB_G}^\heart$ preserves limits and colimits, and so realizes $\Tamb_G^\heart$ as the category of models for an $S\calB_G$-algebra. Thus $G$-tambara functors are $\calB_G$-plethories in the sense that we will study in \cref{sec:plethories}, although we do not expect that they satisfy the various niceness properties introduced there. See \cite{blumberghill2018right} for more on this context.
\tqed
\end{ex}

\begin{rmk}\label{ex:salgebras}
Plethystic functors also arise in more homotopical contexts.

(1)~~Let $R$ be a commutative ring, $\calR$ be the theory of $R$-modules, and $\bbP \calR$ be the category of $\bbE_\infty$ algebras over $R$ which are free on a discrete free $R$-module, i.e.\ of the form $\bbP (R^{\oplus I})$ where $\bbP\colon\Mod_R\rightarrow\CAlg_R$ is the free $\bbE_\infty$ algebra functor. Then $\Model_{\bbP R}\simeq\CAlg_R^\cn$ is equivalent to the category of connective $\bbE_\infty$ algebras over $R$. The homotopy category $\h(\bbP\calR)\simeq S\calR$ is equivalent to the category of polynomial $R$-rings, and restriction along the truncation $\bbP\calR\rightarrow S\calR$ gives a forgetful functor $U\colon\CRing_R\rightarrow\CAlg_R^\cn$. The functor $U$ automatically preserves limits and geometric realizations, and it preserves coproducts as these are given by $\otimes_R$ in either category. Thus $U$ is plethystic, and realizes $S\calR$ as a $\bbP\calR$-algebra. We refer the reader to \cite[Chapter 25]{lurie2018spectral} for a more detailed discussion of the relation between $\CRing_R$ and $\CAlg_R^\cn$. We will briefly revisit this example in \cref{ex:affineline}.

(2)~~Let $G$ be a finite group and $\calO'\subset\calO$ be $G$-coefficient systems in the sense of \cite{blumberghill2015operadic}. Then the forgetful functor $\Alg_{\calO}\rightarrow\Alg_{\calO'}$ is plethystic, where $\Alg_\calO$ is the category of algebras over the $N_\infty$-operad associated to $\calO$.

These two examples point to a possible theory of ``spectral plethories'' encoding various refinements of the basic notion of commutative multiplication encoded by the $\bbE_\infty$ operad, but we shall refrain from further speculation on this.
\tqed
\end{rmk}

\subsection{Monoidal products}\label{ssec:monoidalt}

Suppose that $\Model_\calP$ has been equipped with some form of monoidal product $\otimes$ which preserves colimits in each variable. If moreover the monoidal product preserves the full subcategory $\calP\subset\Model_\calP$, then it is determined by its restriction to $\calP$, from which it can be recovered by Day convolution. In this case, one might call $\calP$ a \textit{monoidal theory}. In this case, if $X',X''\in\Model_\calP$, then $X'\otimes X''$ can be identified as the left Kan extension of the functor $(P',P'')\mapsto X'(P')\times X''(P'')$ along the product $\otimes\colon \calP\times\calP\rightarrow\calP$.

Note in particular the following: if $\calP$ is a monoidal theory, then for any $P',P''\in\calP$, there is a natural pairing
\[
\ev_{P'}\times\ev_{P''}\rightarrow\ev_{P'\otimes P''}
\]
satisfying all the coherences one would expect coming from $\otimes$. This is an advantage of working with algebraic theories without specified sorts, as the presence of automorphisms of objects of $\calP$ has not been hidden.

\subsection{Compositions}\label{ssec:compositions}

If $k$ is an ordinary commutative ring, and $A$ and $B$ are ordinary $k$-algebras in which $k$ is central, then the tensor product $A\otimes_k B$ naturally carries the structure of a $k$-algebra, with product
\[
m\otimes m \circ A\otimes \tau\otimes B\colon A\otimes_k B\otimes_k A\otimes_k B\cong A\otimes_k A\otimes_k B\otimes_k B\rightarrow A\otimes_k  B.
\]
This is not true for general $k$-algebras, or for $k$ noncommutative: we have relied on centrality in order to use the switch map $\tau\colon A\otimes_k B\simeq B\otimes_k A$. Axiomatizing this leads to the notion of a distributive law, discovered by Beck \cite{beck1969distributive}. We summarize the relevant definitions here.

\begin{defn}\label{def:dlaw}
Let $\calC$ be a $1$-category and $F$ and $T$ be monads on $\calC$.
\begin{enumerate}
\item A \textit{composition} of $T$ with $F$ is the structure of a monad on the composite functor $TF$ satisfying the following conditions:
\begin{enumerate}
\item Both $T\eta_F\colon T\rightarrow TF$ and $\eta_TF\colon F\rightarrow TF$ are maps of monads;
\item The composite 
\[
m_{TF}\circ T \eta_F\eta_T F\colon TF\rightarrow TFTF\rightarrow TF
\]
is the identity.
\end{enumerate}
\item A \textit{distributive law} of $F$ across $T$ is a natural transformation $c\colon FT\rightarrow TF$ such that the diagrams
\begin{center}\begin{tikzcd}
&T\ar[dl,"\eta_F T"']\ar[dr,"T\eta_F"]\\
FT\ar[rr,"c"]&&TF
\end{tikzcd}
\begin{tikzcd}
&F\ar[dl,"F\eta_T"']\ar[dr,"\eta_TF"]\\
FT\ar[rr,"c"]&&TF
\end{tikzcd}\end{center}
\begin{center}\begin{tikzcd}
FTT\ar[r,"cT"]\ar[d,"Fm_T"]&TFT\ar[r,"Tc"]&TTF\ar[d,"m_TF"]\\
FT\ar[rr,"c"]&&TF
\end{tikzcd}\begin{tikzcd}
FFT\ar[r,"Fc"]\ar[d,"m_FT"]&FTF\ar[r,"cF"]&TFF\ar[d,"Tm_F"]\\
FT\ar[rr,"c"]&&TF
\end{tikzcd}\end{center}
commute.
\item A \textit{distributive square} is a diagram of categories
\begin{center}\begin{tikzcd}
\calD'\ar[d,"U'",shift left=1mm]\ar[r,"V'"]&\calD\ar[d,"U",shift left=1mm]\\
\calC'\ar[r,"V"]\ar[u,"T'",shift left=1mm]&\calC\ar[u,"T",shift left=1mm]
\end{tikzcd}\end{center}
such that
\begin{enumerate}
\item The diagram commutes with $T$ and $T'$ omitted;
\item The pairs $T'\dashv U'$ and $T\dashv U$ are adjoint;
\item The mate $TV\rightarrow V'T'$ is an isomorphism.
\end{enumerate}
\item A \textit{monadic distributive square} is a distributive square as above such that moreover
\begin{enumerate}
\item There are further left adjoints $F'\dashv V'$ and $F\dashv V$;
\item Each of these adjunctions are monadic adjunctions.
\tqed
\end{enumerate}
\end{enumerate}
\end{defn}

We extend the definitions of distributive squares and monadic distributive squares to allow for the categories involved to not necessarily be $1$-categories; these are essentially the left adjointable squares of \cite[Definition 4.7.4.13]{lurie2017higheralgebra}. Heuristically, a monadic distributive square is a square
\begin{center}\begin{tikzcd}
\calD'\ar[d,"U'",shift left=1mm]\ar[r,"V'",shift left=1mm]&\calD\ar[l,"F'",shift left=1mm]\ar[d,"U",shift left=1mm]\\
\calC'\ar[r,"V",shift left=1mm]\ar[u,"T'",shift left=1mm]&\calC\ar[u,"T",shift left=1mm]\ar[l,"F",shift left=1mm]
\end{tikzcd}\end{center}
of monadic adjunctions such that ``$T = T'$''. This is of course dependent on the orientation of the square.

\begin{lemma}\label{lem:distributivelaw}
Let $\calC$ be a $1$-category and $F$ and $T$ be monads on $\calC$. The following concepts are equivalent:
\begin{enumerate}
\item Compositions of $T$ with $F$;
\item Distributive laws of $F$ across $T$;
\item Monadic distributive squares of the form
\begin{center}\begin{tikzcd}
\calD'\ar[d,"U'",shift left=1mm]\ar[r,"V'",shift left=1mm]&\Alg_T\ar[l,"F'",shift left=1mm]\ar[d,"U",shift left=1mm]\\
\Alg_F\ar[r,"V",shift left=1mm]\ar[u,"T'",shift left=1mm]&\calC\ar[u,"T",shift left=1mm]\ar[l,"F",shift left=1mm]
\end{tikzcd}\end{center}
\end{enumerate}
\end{lemma}
\begin{proof}
The equivalence of these notions is proved in \cite{beck1969distributive}; we just recall here the method of translation between the three. Given a composition of $T$ with $F$, we obtain a distributive law of $F$ across $T$ by the composite
\[
c = m_{TF}\circ \eta_T FT\eta_F\colon FT\rightarrow TFTF\rightarrow TF.
\]
Conversely, given a distributive law $c\colon FT\rightarrow TF$, we can construct a composition of $T$ with $F$ via
\begin{align*}
\eta_{TF} = \eta_T\eta_F&\colon I\rightarrow TF,\\
m_{TF}=m_T m_F\circ TcF&\colon TFTF\rightarrow TTFF\rightarrow TF.
\end{align*}
Given a composition of $T$ with $F$, we obtain a diagram
\begin{center}\begin{tikzcd}
\Alg_{TF}\ar[d,"U'",shift left=1mm]\ar[r,"V'",shift left=1mm]&\Alg_T\ar[l,"F'",shift left=1mm]\ar[d,"U",shift left=1mm]\\
\Alg_F\ar[r,"V",shift left=1mm]\ar[u,"T'",shift left=1mm]&\calC\ar[u,"T",shift left=1mm]\ar[l,"F",shift left=1mm]
\end{tikzcd}\end{center}
of monadic functors. We claim this is distributive, i.e.\ $TV\cong V'T'$ where $T'\dashv U'$. Indeed, it is sufficient to verify this on free $F$-algebras and after forgetting to $\calC$, where this is just the identification $UTVF\cong UV'T'F$. Conversely, given a monadic distributive square as in the statement of the lemma, we obtain a distributive law of $F$ across $T$ by the composite
\[
FT = VFUT \rightarrow VU'F'T\cong UTVF = TF,
\]
the arrow being obtained from the mate $FU\rightarrow U'F'$.
\end{proof}

\begin{ex}
Let $k$ be an ordinary associative algebra and $A$ and $B$ be ordinary $k$-algebras. Then
\begin{enumerate}
\item Compositions of $A$ with $B$ are algebra structures on $A\otimes_k B$ such that
\begin{enumerate}
\item $(a'\otimes 1)\cdot (a''\otimes 1) = a'a''\otimes 1$ and $(1\otimes b')\cdot (1\otimes b'') = 1\otimes b'b''$;
\item $(a\otimes 1)\cdot (1\otimes b) = a\otimes b$.
\end{enumerate}
\item Distributive laws of $B$ across $A$ are maps $c\colon B\otimes_k A\rightarrow A\otimes_k B$ of $k$-bimodules such that
\begin{enumerate}
\item $c(1\otimes a) = a\otimes 1$ and $c(b\otimes 1) = 1\otimes b$;
\item If we write $c(b\otimes a) = \sum a_{(i)}\otimes b_{(i)}$ for a placeholder symbol $i$,  then $\sum a_{(1)}'a_{(2)}' \otimes b_{(1)(2)} = \sum (a'a'')_{(3)} \otimes b_{(3)}$ and $\sum a_{(1)(2)}\otimes b_{(2)}'b_{(1)}' = \sum a_{(3)}\otimes (b'b'')_{(3)}$ for any $a,a',a''\in A$ and $b,b',b''\in B$.
\end{enumerate}
\item Given a commutative diagram
\begin{center}\begin{tikzcd}
C&A\ar[l]\\
B\ar[u]&k\ar[u]\ar[l]
\end{tikzcd}\end{center}
of algebra maps, there is an associated commutative diagram
\begin{center}\begin{tikzcd}
\LMod_C^\heart\ar[r]\ar[d]&\LMod_A^\heart\ar[d]\\
\LMod_B^\heart\ar[r]&\LMod_k^\heart
\end{tikzcd}\end{center}
of monadic forgetful functors. The mate of this diagram is given by maps $A\otimes_k M\rightarrow C\otimes_B M$ defined for left $B$-modules $M$, and is a natural isomorphism when it evaluates on $B$ to an isomorphism $A\otimes_k B\cong C$.
\end{enumerate}
Even when each of $k$, $A$, and $B$ are commutative, these notions do not collapse. For example, if $\bbH$ is the ring of quaternions, then
\begin{center}\begin{tikzcd}
\bbH&\bbC\ar[l,"f"']\\
\bbC\ar[u,"g"]&\bbR\ar[l]\ar[u]
\end{tikzcd}\end{center}
satisfies the conditions of (3), where $f(i) = j$ and $g(i) = k$. The distributive law is the map $c\colon \bbC\otimes_\bbR\bbC\rightarrow\bbC\otimes_\bbR\bbC$ given by $c(i\otimes i) = -i\otimes i$, and otherwise by the standard symmetry.
\tqed
\end{ex}

For the most part, we will encounter distributive laws in the form of monadic distributive squares, and the theory of distributive laws then provides a method for understanding the categories involved. In general, this is an instance of the following philosophy: it is often easier to construct the category of algebras over a monad than it is to construct the monad itself. The theory of distributive laws gives a way of accessing the monad associated to categories constructed by such indirect methods; the following are some typical examples.

\begin{ex}\label{ex:semitensor}
Let $k$ be an ordinary commutative ring, $B$ be an ordinary $k$-bialgebra, and $A$ be a monoid in the monoidal category $(\LMod_B^\heart,\otimes_k)$, with resulting category $\LMod_A(\LMod_B^\heart)$ of modules therein. Then
\begin{center}\begin{tikzcd}
\LMod_A(\LMod_B^\heart)\ar[r]\ar[d]&\LMod_A^\heart\ar[d]\\
\LMod_B^\heart\ar[r]&\Mod_k^\heart
\end{tikzcd}\end{center}
is a monadic distributive square, and thus $\LMod_A(\LMod_B^\heart)\simeq\LMod_{A\otimes_k B}^\heart$ for some $k$-algebra structure on $A\otimes_k B$. Using \cref{lem:distributivelaw}, one can compute that this $k$-algebra structure is the ``semi-tensor product'' \cite{masseypeterson1964cohomology}, given by
\[
(a'\otimes b')\cdot (a''\otimes b'') = \sum a'(b'_{(1)}\cdot a'')\otimes b'_{(2)}b'',
\]
where we have written $\Delta(b) = \sum b_{(1)}\otimes b_{(2)}$ for the coproduct on $B$. Equivalently, the distributive law is given by $c(b\otimes a) = \sum (b_{(1)}\cdot a )\otimes b_{(2)}$.
\tqed
\end{ex}

\begin{ex}\label{ex:slicedist}
Let $\calP$ be a discrete theory, $F$ a discrete $\calP$-algebra, and $B\in\Model_F^\heart$. Then
\begin{center}\begin{tikzcd}
B/\Model_F^\heart\ar[r]\ar[d]&B/\Model_\calP^\heart\ar[d]\\
\Model_F^\heart\ar[r]&\Model_\calP^\heart
\end{tikzcd}\end{center}
is a monadic distributive square. The distributive law is just the map
\[
F(B\coprod\bs)\simeq F(B)\coprod F(\bs)\rightarrow B\coprod F(\bs)
\]
obtained from the fact that $F$ preserves coproducts and the $F$-model structure of $B$. 
\tqed
\end{ex}

See \cref{ex:theta} for a more explicit instance of the preceding examples.

\subsection{Left-derived functors}\label{ssec:lder}

Fix two discrete theories $\calP$ and $\calP'$.

\begin{defn}
Fix an arbitrary functor $F\colon \Model_{\calP'}^\heart\rightarrow\Model_\calP^\heart$, and let $f$ denote the composite
\[
f\colon \calP'\subset\Model_{\calP'}^\heart\rightarrow\Model_\calP^\heart\subset\Model_\calP.
\]
The \textit{total left-derived functor} of $F$ is the functor
\[
\bbL F = f_!\colon \Model_{\calP'}\rightarrow\Model_\calP
\]
obtained from $f$ by left Kan extension. When $\calP$ is pointed, we abbreviate $\bbL_n F = \pi_n \bbL F$.
\tqed
\end{defn}

Total left-derived functors can be computed in the usual way, by taking projective resolutions. Their identification with a left Kan extension is a situation where the use of infinitary theories simplifies the story.

\begin{ex}\label{ex:extcompl}
Let $F\colon\Mod_{\bbZ_p}^\heart\rightarrow\Mod_{\bbZ_p}^\heart$ denote the functor of $p$-adic completion. Then $F$ is neither left nor right exact in general. Nonetheless, we may consider the total left-derived functor $\bbL F$. This has the following properties:
\begin{enumerate}
\item $\bbL F$ gives the correct notion of $p$-completion for the homotopy theory $\Mod_{\bbZ_p}^\cn$;
\item If $M\in\Mod_{\bbZ_p}^\heart$, then $\bbL FM$ is $1$-truncated, $\bbL_0 F M$ is the ${\Ext}\hyp p$-completion of $M$, and $\bbL_1 F$ is the ${\Hom}\hyp p$-completion of $M$ in the sense of \cite[Section VI.2.1]{bousfieldkan1972monster}.
\end{enumerate}
This is a purely infinitary construction, as $F$ restricts to the identity on the category of finitely generated $\bbZ_p$-modules.
\tqed
\end{ex}

\subsection{Unbounded derived categories}\label{ssec:additivenotation}

If $\calP$ is an additive theory, then we will write
\[
\LMod_\calP^\heart = \Model_\calP^\heart,\qquad \LMod_\calP^\cn = \Model_\calP,
\]
and further define $\LMod_\calP$ to be the category of $\Sp$-valued models of $\calP$. There are then fully faithful embeddings $\LMod_\calP^\heart\subset\LMod_\calP^\cn\subset\LMod_\calP$, and $\LMod_\calP$ is the stabilization of $\LMod_\calP^\cn$.

In particular, for $X,Y\in\LMod_\calP$, there is a mapping spectrum $\EXT_\calP(X,Y)$ with 
\[
\Omega^{\infty-n}\EXT_\calP(X,Y)\simeq\Map_\calP(X,\Sigma^n Y),
\]
and we write
\[
\Ext^n_\calP(X,Y) = \pi_{-n}\EXT_\calP(X,Y) .
\]
When $\calP$, $X$ and $Y$ are discrete, these are the usual $\Ext$ groups defined for the abelian category $\LMod_\calP^\heart$.

\subsection{Quillen cohomology}\label{ssec:qcohomology}

Fix a discrete theory $\calP$. Then $\Ab(\Model_\calP^\heart)$ is equivalent to the category of $\Ab$-valued models of $\calP$, and this category is strongly monadic over $\Model_\calP^\heart$. Write the left adjoint as $D\colon \Model_\calP^\heart\rightarrow\Ab(\Model_\calP^\heart)$, so that $\Ab(\Model_\calP^\heart)\simeq\LMod_{D\calP}^\heart$. Here, $D\calP$ is an additive theory, so the notation of \cref{ssec:additivenotation} applies.

\begin{defn}\label{def:qcohomology}
For $A\in\Model_\calP$ and $M\in\LMod_{D\calP}$, define
\[
\calH_\calP(A;M) = \EXT_{D\calP}(\bbL DA,M),\quad H^n_\calP(A;M) = \pi_{-n}\calH_\calP(A;M)= \Ext^n_{D\calP}(\bbL DA,M).
\]
Equivalently,
\[
\calH^n_\calP(A;M) = \Omega^{\infty-n}\calH_\calP(A;M) = \Map_\calP(A,B^nM),\qquad H^n_\calP(A;M) = \pi_0 \calH^n_\calP(A;M) .
\]
These are the \textit{Quillen cohomology of $A$ with coefficients in $M$}.
\tqed
\end{defn}

Often the theory at hand is instead of the form $\calP/B$ for some theory $\calP$ and $B\in\Model_\calP^\heart$, as in this case $\Model_{\calP/B} \simeq \Model_\calP/B$. Write $D_B$ for the relevant functor of abelianization. Call $B\in\Model_\calP^\heart$ \textit{smooth} if $\bbL D_B B$ is discrete and projective. When $\calP$ is the theory of $R$-rings for some commutative ring $R$, this is not quite the standard notion of smoothness, as we have imposed no finiteness conditions.

\begin{lemma}
Given $f\colon B\rightarrow C$, there is an equivalence $\bbL D_C B \simeq f_! \bbL D_B B$, where $f_!$ is the total derived functor of the left adjoint to pullback $f^\ast\colon\Ab(\Model_\calP^\heart/C)\rightarrow\Ab(\Model_\calP^\heart/B)$. In particular, if $B$ is smooth, then $\bbL D_C B$ is discrete and projective for any map $f$.
\end{lemma}
\begin{proof}
Observe that the diagram
\begin{center}\begin{tikzcd}
\Ab(\Model_\calP^\heart/C)\ar[r,"f^\ast"]\ar[d]&\Ab(\Model_\calP^\heart/B)\ar[d]\\
\Model_\calP^\heart/C\ar[r,"f^\ast"]&\Model_\calP^\heart/B
\end{tikzcd}\end{center}
of right adjoints commutes, and continues to commute after passing to derived categories. The lemma follows upon taking left adjoints.
\end{proof}

\subsection{Cohomology over an algebra}\label{ssec:abalg}

Fix a discrete theory $\calP$ and $\calP$-algebra $F$. We would like to be able to compute the Quillen cohomology of $F$-models.

\begin{lemma}
Suppose $U\colon\Model_{\calP'}^\heart\rightarrow\Model_\calP^\heart$ is strongly monadic. Then the induced map $V\colon\Model_{D\calP'}^\heart\rightarrow\Model_{D\calP}^\heart$ is strongly monadic, and is plethystic whenever $U$ is.
\end{lemma}
\begin{proof}
It is easily seen that $V$ is strongly monadic. As $V$ is additive, to be plethystic it is sufficient for $V$ to preserve filtered colimits, for which it is sufficient that $U$ preserves filtered colimits, which holds if $U$ is plethystic.
\end{proof}

What makes an algebra $F$ special is the existence of the limit-preserving comonad $F^\vee$. Heuristically, this is because $F^\vee$ preserves algebraic structure. In particular we can identify abelian group objects in $\Model_F^\heart\simeq\CoAlg_{F^\vee}^\heart$ using the following.

\begin{prop}\label{prop:abelianizecomonad}
Let $\calC$ be a $1$-category with finite products, and $G$ be a comonad on $\calC$ which preserves these. Then
\begin{enumerate}
\item $\CoAlg_G\rightarrow\calC$ creates finite products;
\item The resulting forgetful functor $\Ab(\CoAlg_G)\rightarrow\Ab(\calC)$ is comonadic;
\item The diagram
\begin{center}\begin{tikzcd}
\Ab(\CoAlg_G)\ar[r]\ar[d,"U'"]&\Ab(\calC)\ar[d,"U"]\\
\CoAlg_G\ar[r]&\calC
\end{tikzcd}\end{center}
of forgetful functors is Cartesian whenever $U$ is fully faithful;
\item The natural transformation $U'\circ G'\rightarrow G\circ U$ fitting in the diagram
\begin{center}\begin{tikzcd}
\Ab(\CoAlg_G)\ar[d,"U'"]&\Ab(\calC)\ar[d,"U"]\ar[l,"G'"']\\
\CoAlg_G&\calC\ar[l,"G"']
\end{tikzcd}\end{center}
is an isomorphism;
\item If $U$ admits a left adjoint $D$, then $D$ lifts to a left adjoint $D'\colon \CoAlg_G\rightarrow\Ab(\CoAlg_G)$ making the diagram in (3) distributive.
\end{enumerate}
\end{prop}
\begin{proof}
(1)~~ This is clear.

(2, 4)~~ As $G$ preserves finite products, it lifts to a comonad $G'$ on $\Ab(\calC)$. A $G'$-coalgebra consists of some $A\in\Ab(\calC)$ together with a coaction $A\rightarrow GA$ which is a map of abelian group objects, i.e.\ such that the diagram 
\begin{center}\begin{tikzcd}
A\times A\ar[d]\ar[r]&GA\times GA\ar[d]\\
A\ar[r]&GA
\end{tikzcd}\end{center}
commutes. Looking at it a different way, this is the same as asking for $A$ to be an abelian group object in $\CoAlg_G$, so $\CoAlg_{G'}\simeq\Ab(\CoAlg_G)$. 

(3)~~ If $\Ab(\calC)\rightarrow\calC$ is fully faithful, then the above diagram automatically commutes for any choice of multiplication $A\times A\rightarrow A$ and coaction $A\rightarrow GA$, and the claim quickly follows. 

(5)~~ Given the left adjoint $D$, there is a left adjoint $D'\colon\CoAlg_G\rightarrow\Ab(\CoAlg_G)$ sending a $G$-coalgebra $A\rightarrow GA$ to the $G'$-coalgebra with coaction the unique dashed arrow filling in
\begin{center}\begin{tikzcd}
A\ar[d]\ar[r]&GA\ar[d]\\
DA\ar[r,dashed]&GDA
\end{tikzcd}\end{center}
as a map of abelian group objects in $\calC$. This has the desired properties.
\end{proof}

By \cref{prop:abelianizecomonad}, if $B\in\Model_F^\heart$ and we are treating $B$ as a model of $\calP$ equipped with extra structure, then the notation $D(B)$ is unambiguous, for the abelianization of $B$ is the same when computed in $\Model_F^\heart$ or $\Model_\calP^\heart$. However, the notation $\bbL D(B)$ is ambiguous; thus for the moment we write $\bbL D'$ for the derived abelianization of $F$-models.

Call the algebra $F$ \textit{smooth} if $F(P)$ is smooth for all $P\in\calP$.

\begin{prop}\label{prop:smoothab}
If $F$ is smooth, then the diagram
\begin{center}\begin{tikzcd}
\LMod_{DF}^\cn\ar[r,"V'"]\ar[d]&\LMod_{D\calP}^\cn\ar[d]\\
\Model_F\ar[r,"V"]&\Model_\calP
\end{tikzcd}\end{center}
is distributive.
\end{prop}
\begin{proof}
As both $\bbL D\circ V$ and $V'\circ \bbL D'$ preserve geometric realizations, it is sufficient to verify that the map $\bbL D\circ V\rightarrow V'\circ\bbL D'$ is an equivalence when restricted to $F\calP$. Here, it follows from  smoothness and \cref{prop:abelianizecomonad}.
\end{proof}

Thus $\bbL D$ is unambiguous provided $F$ is smooth, for in this case the derived abelianization of an $F$-model is the same as computed with respect to $F$ or $\calP$. In practice we will assume that our algebras are smooth when we consider the Quillen cohomology of their algebras.

We end by noting the following, illustrating the purpose of smooth algebras.

\begin{prop}\label{thm:grothss}
Fix a smooth algebra $F$. For $B\in\Model_F$ and $M\in\LMod_{DF}$, there is a spectral sequence
\[
E_1^{p,q} = \Ext^{q-p}_{DF}(\bbL_p D(B),M)\Rightarrow H^q_F(B;M),\qquad d_r^{p,q}\colon E_r^{p,q}\rightarrow E_r^{p-r,q+1},
\]
which is convergent, for instance, if $\bbL D(B)$ is truncated. In particular, if $B$ is smooth as an object of $\Model_\calP^\heart$, then
\[
H^q_F(B;M) \cong \Ext^q_{DF}(D(B);M).
\]
\end{prop}
\begin{proof}
By definition, $\calH_F(B;M) = \EXT_{DF}(\bbL D(B),M).$ Smoothness of $F$ ensures that the notation $\bbL D(B)$ is unambiguous. The spectral sequence is then associated to the filtration of $\EXT_{DF}(\bbL D(B),M)$ by the Whitehead tower of $\bbL D(B)$.
\end{proof}

More explicit examples of these ideas will be given in \cref{ssec:plethaq}. 

\section{Koszul algebras}\label{sec:koszul}

This section is concerned with additive theories, and in particular with the notion of a \textit{Koszul algebra} over an additive theory.

\subsection{Coalgebras}\label{ssec:additivebimodules}

Fix additive theories $\calP$ and $\calP'$. To emphasize that we are working in the additive setting, we will refer to $\calP'\hyp\calP$-bimodels as \textit{$\calP'\hyp\calP$-bimodules}; see \cref{ex:bimodules}. In addition, we can extend a $\calP'\hyp\calP$-bimodule $H\colon \LMod_{\calP}^\cn\rightarrow\LMod_{\calP'}^\cn$ to a colimit-preserving functor $\LMod_{\calP}\rightarrow\LMod_{\calP'}$, and will do so without change of notation.

It happens on occasion that a bimodule $H$ has the property that its right adjoint $H^\vee$ preserves colimits; in this case, $H^\vee$ is also a bimodule, with further right adjoint $H^{\vee\vee}$.

\begin{ex}\label{ex:dualbimodule}
Let $A$ and $B$ be ordinary associative algebras and $H$ a discrete $B\hyp A$-bimodule. As mentioned in \cref{ex:bimodules}, we can recover the ordinary dual bimodule $\LMod_B(H,B)$ of $H$ by considering the restriction of $H^\vee$ to the category of finitely generated free $B$-modules; its left Kan extension to the category of all free $B$-modules is then the functor $H^\vee_0(M) = \Hom_B(H,B)\otimes_B M$. There is a comparison map
\[
H^\vee_0\rightarrow H^\vee,\quad\theta\colon\Hom_B(H,B)\otimes_B M\rightarrow\Hom_B(H,M),\quad\theta(f\otimes m)(h) = f(h)m,
\]
which is an isomorphism when $H$ is finitely presented and projective as a left $B$-module.
\tqed
\end{ex}

It is not necessary for $H^\vee$ to preserve colimits to talk about monad structures on $H^\vee$. These are equivalently comonad structures on $H$, and thus deserve to be called \textit{$\calP$-coalgebras}. In this case, $H^\vee$-modules are the analogues of \textit{$H$-contramodules} in the sense of \cite[Section III.5]{eilenbergmoore1965foundations}, but we will not use this name.

\subsection{Cobar complexes}\label{ssec:barcon}

In this subsection we review bar resolutions and cobar complexes in some detail, largely to make our conventions explicit; in particular, this material is not novel, the generalization from ordinary algebras to additive monads being primarily a matter of notation. Fix for the moment an arbitrary category $\calM$---not necessarily a $1$-category---and a monad $T$ on $\calM$. For $M\in\Alg_T$, we may form the \textit{bar construction} $B(T,T,M)$. This is the simplicial object augmented over $M$ with
\[
B_n(T,T,M) = T^{1+n}M,\qquad d_i = T^{i}mT^{n-i}\colon T^{1+1+n}M\rightarrow T^{1+n}M,\quad 0\leq i \leq n+1.
\]
Here, $d_{n+1}$ is to be understood as given by the $T$-module structure on $M$.

\begin{lemma}\label{lem:bares}
The simplicial object $B(T,T,M)$ is a resolution of $M$, in the sense that the augmentation extends to an equivalence
\[
\colim_{n\in\Delta^\op}B_n(T,T,M)\simeq M
\]
in $\Alg_T$.
\end{lemma}
\begin{proof}
The augmented simplicial object $M\leftarrow B(T,T,M)$ is $T$-split, so the claim follows as $\Alg_T\rightarrow\calM$ creates $T$-split geometric realizations \cite[Theorem 4.7.3.5]{lurie2017higheralgebra}.
\end{proof}

We now restrict ourselves to the case where $\calM = \Model_\calP^\heart$ for a discrete additive theory $\calP$ and $T=F$ is a discrete $\calP$-algebra; the rest of this subsection takes place in a $1$-category. Assuming that $F$ is a colimit-preserving monad on a category of the form $\Model_\calP^\heart$ is much more than is necessary, as most of the following is just an explicit comparison of finite formulas; the assumption is made purely for notational convenience.

In this additive setting, we may produce from $B(F,F,M)$ the \textit{unreduced bar resolution} $C^\un(F,F,M)$, which is a chain complex of $F$-modules of the form
\[
C_n^\un(F,F,M) = F^{1+n} M,\qquad d = \sum_{0\leq i \leq n+1}(-1)^id_i\colon F^{1+1+n}M\rightarrow F^{1+n}M,
\]
as well as the \textit{reduced bar resolution} $C(F,F,M)$, which is the quotient chain complex of $C^\un(F,F,M)$ with
\[
C_n(F,F,M) = FF_+^nM,\qquad F_+ = \coker(I\rightarrow F).
\]

Given $M,M'\in\LMod_F^\heart$, define
\[
B_F(M,M') = \Hom_{F\calP}(B(F,F,M),M'),
\]
so that
\[
B_F^{n}(M,M')= \Hom_{F\calP}(F^{1+n}M,M') \cong \Hom_\calP(F^nM,M').
\]
This is a cosimplicial abelian group modeling $\EXT_F(M,M')$ provided that $B(F,F,M)$ consists of projective $F$-modules. From this we extract the \textit{unreduced cobar complex} $C_F^\un(M,M')$ and \textit{reduced cobar complex} $C_F(M,M')\subset C_F^\un(M,M')$; the differential on $C_F^\un(M,M')$ is given
\begin{align*}
\delta&\colon\Hom_\calP(F^nM,M')\rightarrow\Hom_\calP(F^{1+n}M,M'),\\
\delta &= \delta_0 + \sum_{1\leq i \leq n} (-1)^i \delta_i + (-1)^{n+1}\delta_{n+1},\\
\delta_0(f) &= m \circ Ff\colon F^{1+n}M\rightarrow FM'\rightarrow M',\\
\delta_i(f) &= f\circ F^{i-1}mF^{n-i}\colon F^{1+n}M\rightarrow F^nM\rightarrow M',\\
\delta_{n+1}(f) &= f\circ F^n m\colon F^{1+n}M\rightarrow F^nM\rightarrow M'.
\end{align*}

\begin{lemma}\label{prop:cobarcomposition}
Fix $M,M',M''\in\LMod_F^\heart$. Define 
\[
\wr\colon C_F^{\un,n}(M,M')\otimes C_F^{\un,n'}(M',M'')\rightarrow C_F^{\un,n'+n}(M,M'')
\]
as follows: given $f\colon F^nM\rightarrow M'$ and $f'\colon F^{n'}M'\rightarrow M''$, set
\[
(-1)^{nn'}f\wr f' = f'\circ F^{n'}f\colon F^{n'+n}M\rightarrow F^{n'}M'\rightarrow M''.
\]
Then $\wr$ has the following properties:
\begin{enumerate}
\item If $f\in C^n_F(M,M')$ and $f'\in C^{n'}_F(M',M'')$, then $f\wr f'\in C^{n'+n}_F(M,M'')$.
\item $\delta(f\wr f') = \delta(f)\wr f' + (-1)^n f\wr \delta(f')$, and thus $\wr$ passes to pairings
\begin{align*}
C_F^\un(M,M')\otimes C_F^\un(M',M'')&\rightarrow C_F^\un(M,M''),\\
 C_F(M,M')\otimes C_F(M',M'')&\rightarrow C_F(M,M'')
\end{align*}
of cochain complexes.
\item Suppose that $C(F,F,M)$ and $C(F,F,M')$ are projective resolutions of $M$ and $M'$. Then the induced pairing
\[
\wr\colon \Ext^n_F(M,M')\otimes\Ext^{n'}_F(M',M'')\rightarrow\Ext^{n'+n}_F(M,M'')
\]
is the graded opposite of the standard Yoneda composition:
\[
f\wr f' = (-1)^{nn'} f'\circ f.
\]
Here, to be explicit, we take the Yoneda composition as defined in \cite[Section III.5, Theorem III.6.4]{maclane1975homology} as the standard.
\end{enumerate}
\end{lemma}
\begin{proof}
(1)~~ This is clear.

(2)~~ Fix $f\colon F^nM\rightarrow M'$ and $f'\colon F^{n'}M'\rightarrow M''$. The main point is that
\[
\delta_{n+1}(f')\circ F^{n'+1}f = f'\circ F^{n'}\delta_0(f);
\]
this allows us to compute
\begin{align*}
\delta(f'\circ F^{n'}f) &= \sum_{i=0}^{n'+1+n}(-1)^i \delta_i(f'\circ F^{n'}f)\\
&=\sum_{i=0}^{n'}\delta_i(f'\circ F^{n'}f) + (-1)^{n'}\sum_{i = n'+1}^{n'+1+n}(-1)^i \delta_i(f'\circ F^{n'}f)\\
&=\sum_{i=0}^{n'}\delta_i(f')\circ F^{n'}f + (-1)^{n'}\sum_{i=1}^{1+n}(-1)^i f'\circ F^{n'}\delta_i(f)\\
&=\sum_{i=0}^{n'+1}\delta_i(f')\circ F^{n'}f + (-1)^{n'}\sum_{i=0}^{1+n}(-1)^i f'\circ F^{n'}\delta_i(f)\\
&=\delta(f')\circ F^{n'}f + (-1)^{n'}f'\circ F^{n'}\delta(f),
\end{align*}
which yields
\begin{align*}
\delta(f\wr f') &= (-1)^{nn'}\delta(f'\circ F^{n'}f)\\
&=(-1)^{nn'}\delta(f')\circ F^n f + (-1)^{nn'+n'}f'\circ F^{n'}\delta(f)\\
&=(-1)^{nn'+n(n'+1)} f \wr \delta(f') + (-1)^{nn'+n'+(n+1)n'}\delta(f)\wr f' \\
&= \delta(f)\wr f' + (-1)^n f \wr \delta(f')
\end{align*}
as claimed. 

(3)~~ We first introduce a bit of local notation. If $\calA$ is an additive category, there are shift functors
\[
\Ch(\calA)\rightarrow\Ch(\calA),\qquad C \mapsto C[p]
\]
for $p\in\bbZ$ defined on objects by
\[
C[p]_n = C_{n-p},\qquad d^{C[p]}_n = (-1)^p d^C_{n-p},
\]
and on morphisms by
\[
f[p]_n = f_{n-p}.
\]

Now fix a map $f\colon F^{n}M\rightarrow M'$ in $\LMod_\calP^\heart$. This lifts to a map
\[
f^+\colon C^\un(F,F,M)\rightarrow C^\un(F,F,M')[n]
\]
of graded objects by
\[
f^+_{<n} = 0,\qquad f^+_{k+n} = (-1)^{nk} F^{1+k}f.
\]
This satisfies
\[
f_{k+n}^+ \circ d - d \circ f_{1+k+n}^+ = (-1)^{(n+1)k}F^{1+k}\delta(f),
\]
and is thus a chain map whenever $f$ is a cocycle. If $f'\colon F^{n'}M'\rightarrow M''$ is another map in $\LMod_\calP^\heart$, then $f'\circ f^+_{n'+n} = f \wr f'$. We conclude by observing that if $C(F,F,M)$ and $C(F,F,M')$ are projective resolutions of $M$ and $M'$, and $f$ and $f'$ are cocycles, then $f'\circ f^+_{n'+n}$ is a cocycle representing the standard Yoneda composition of the classes represented by $f$ and $f'$, only twisted by $(-1)^{nn'}$.
\end{proof}

When $\calP$ is the theory of $\bbZ$-graded modules over an ordinary $\bbZ$-graded algebra, it is standard practice to insert additional signs in various places when developing the homological algebra of $\LMod_\calP^\heart$, where these additional signs are dependent on the internal degrees of elements involved. The internal degrees of elements of a graded module cannot be defined in a Morita-invariant way, so these signs cannot be incorporated at the present level of generality. In practice one can simply modify the constructions of this section to be compatible with whatever conventions are most convenient for a given category. Here is an example indicating the effect of this in the standard case.

\begin{ex}\label{ex:signs}
Let $k$ be a commutative ring and $A$ be an ordinary projective $\bbZ$-graded augmented $k$-algebra in which $k$ is central. Let $H$ be the cohomology algebra of $A$ defined with conventions from \cite{priddy1970koszul}. Write $s^q$ for $q$-fold shift, so that $(s^qM)_n = M_{n-q}$ for a $\bbZ$-graded object $M$. Then $H$ is a bigraded object with
\[
H^{p,q} = \Ext^p(k,s^qk).
\]
Write $\smile$ for the product on $H$ and $\circ$ for the Yoneda composition on $\Ext$. Then for $x\in H^{p,q}$ and $x'\in H^{p',q'}$ these pairings satisfy
\[
x'\smile x = (-1)^{q(q'-p')}s^q x'\circ x.
\]
Closer to our conventions is the bigraded opposite algebra $H^\op$. This is the algebra with $(H^\op)^p_q = H^{p,-q}$ and product given for $x\in (H^\op)^p_q$ and $x'\in (H^\op)^{p'}_{q'}$ by
\[
x\smile^\op x' = (-1)^{(q-p)(q'-p')} x'\smile x.
\]
If we identify $(H^\op)^p_q = \Ext^p(s^qk,k)$, then this satisfies
\[
x\smile^\op x' = (-1)^{pq'}s^{q'}x\wr x'.
\]
\tqed
\end{ex}

\subsection{Cohomology of augmented algebras}\label{ssec:homologycohomology}

Fix an additive theory $\calP$ and $\calP$-algebra $F$. Suppose that $F$ is \textit{augmented}; that is, that we have chosen a map $\epsilon\colon F\rightarrow I$ of algebras, where $I$ is the initial algebra, given by the identity functor. Restriction along the augmentation gives a functor
\[
\epsilon^\ast\colon \LMod_\calP\rightarrow\LMod_F,\qquad \epsilon^\ast(M) = \ol{M}.
\]
As $\epsilon^\ast$ preserves limits and colimits, it is part of an adjoint triple $\epsilon_!\dashv \epsilon^\ast\dashv\epsilon_\ast$, giving thus a (non-discrete) $\calP$-coalgebra $\epsilon_!\epsilon^\ast$ with right adjoint monad $\epsilon_\ast \epsilon^\ast$. The functor $\epsilon_!\epsilon^\ast$ can be identified as arising from a bar construction: for $M\in\LMod_\calP$, there is an equivalence
\[
\epsilon_!\epsilon^\ast M \simeq \epsilon_! \colim_{n\in\Delta^\op}B_n(F,F,\ol{M})\simeq \colim_{n\in\Delta^\op}B_n(I,F,\ol{M}),
\]
where by definition $B(I,F,\bs) = \epsilon_! B(F,F,\bs)$. Likewise, $\epsilon_\ast\epsilon^\ast$ can be identified as
\[
(\epsilon_\ast\epsilon^\ast M)(P) = \EXT_\calP(P,\epsilon_\ast\epsilon^\ast M) = \EXT_F(\ol{P},\ol{M}),
\]
which may be computed via a cobar construction.

We will make minimal use of these homotopical objects directly, but will make use of their algebraic shadows. Specialize now to the case where $\calP$ is discrete and $F$ is projective. Under these assumptions, $C(I,F,\bs)$ is a chain complex of discrete $\calP$-bimodules modeling $\epsilon_!\epsilon^\ast$. Define
\begin{align*}
H_n(F) &= \pi_n \epsilon_! \epsilon^\ast\colon\calP\rightarrow\LMod_\calP^\heart,\qquad H_n(F)_{P,P'} = H_n C(I,F,\ol{P})_{P'};\\
H^n(F) &= \pi_{-n} \epsilon_\ast\epsilon^\ast\colon\calP\rightarrow\LMod_\calP^\heart,\qquad H^n(F)_{P,P'} = \Ext^n_F(\ol{P}',\ol{P}).
\end{align*}
These extend to endofunctors of $\LMod_\calP^\heart$, and if each $H_n(F)_P$ is projective, then $H_n(F)$ is a bimodule with $H_n(F)^\vee \cong H^n(F)$. The products of \cref{prop:cobarcomposition} give pairings
\[
\wr\colon H^{n'}(F)_{P',P''}\otimes H^n(F)_{P,P'}\rightarrow H^{n'+n}(F)_{P,P''},
\]
and these make $H^\ast(F)$ into a graded monad. When each $H_n(F)$ is projective, the identification $H_n(F)^\vee \cong H^n(F)$ implies that $H_\ast(F)$ can be considered as a graded comonad, although we will not make use of this.

\begin{ex}\label{ex:cohord}
Let $k$ be an ordinary algebra and $A$ an ordinary augmented projective $k$-algebra. Then treating $A$ as an algebra for the theory of left $k$-modules, the above definitions give
\[
H_\ast(A)_{k,k} = \Tor_\ast^A(k,k),\qquad H^\ast(A)_{k,k} = \Ext^\ast_A(k,k).
\]
Here $H^\ast(A)_{k,k}$ is itself an augmented $k$-algebra, with multiplication given by the graded opposite of Yoneda composition. On the other hand, suppose instead that $A$ is an ordinary $\bbZ$-graded augmented projective $k$-algebra. Then still we have
\[
H^\ast(A)_{e_p,e_q} = \Ext^\ast_A(e_q,e_p),
\]
where $e_a$ denotes a copy of $k$ in degree $a$. In particular, we can extract from this a left $k$-module
\[
H^\ast(A)_{e_0} = \Ext^\ast_A(e_\ast,e_0).
\]
Heuristically, this is the ordinary cohomology algebra of $A$. However, note the following subtlety: to make $H^\ast(A)_{e_0}$ into an ordinary algebra requires the additional structure of the isomorphisms $\Ext^\ast_A(e_{a+b},e_b)\cong\Ext^\ast_A(e_a,e_0)$.
\tqed
\end{ex}

\subsection{Homology and compositions}\label{ssec:homologycompositions}

Classically, the homology of ordinary augmented algebras is well-behaved with respect to base change, and there are K\"unneth isomorphisms describing the homology of a tensor product under suitable flatness conditions. We will need the analogues of these where tensor products are replaced with compositions in the sense of \cref{ssec:compositions}.

\begin{lemma}\label{lem:basechangebar}
Suppose given a square
\begin{center}\begin{tikzcd}
\calD'\ar[d,"U'",shift left=1mm]\ar[r,"V'",shift left=1mm]&\calD\ar[l,"F'",shift left=1mm]\ar[d,"U",shift left=1mm]\\
\calC'\ar[r,"V",shift left=1mm]\ar[u,"T'",shift left=1mm]&\calC\ar[u,"T",shift left=1mm]\ar[l,"F",shift left=1mm]
\end{tikzcd}\end{center}
of adjunctions between $1$-categories. Then there is a natural simplicial map
\[
T'B(F,F,C')\rightarrow B(F',F',T'C')
\]
defined for $C'\in\calC'$, which is an isomorphism if the square is distributive.
\end{lemma}
\begin{proof}
In degree $n$, this is the map
\[
T'F(VF)^nVC'\rightarrow F'(V'F')^nV'T'C'
\]
obtained by repeated application of the mate
\[
T'FV\simeq F'TV\rightarrow F'V'T'.
\]
That this is an isomorphism when the square is distributive is straightforward.
\end{proof}

\begin{prop}\label{lem:tensoraug}
Let $\calC$ be a $1$-category, let $T$ and $F$ be monads on $\calC$ together with a distributive law $c\colon FT\rightarrow TF$ allowing us to form the composition monad $TF$, and let
\begin{center}\begin{tikzcd}
\calD'\ar[d,"U'",shift left=1mm]\ar[r,"V'",shift left=1mm]&\calD\ar[l,"F'",shift left=1mm]\ar[d,"U",shift left=1mm]\\
\calC'\ar[r,"V",shift left=1mm]\ar[u,"T'",shift left=1mm]&\calC\ar[u,"T",shift left=1mm]\ar[l,"F",shift left=1mm]
\end{tikzcd}\end{center}
be the associated monadic distributive square. Suppose that $F$ is equipped with an augmentation $\epsilon$ such that $T\epsilon\circ c = \epsilon T$. Then $\epsilon$ lifts to an augmentation on $F'$, and moreover
\[
TB(I,F,\epsilon_F^\ast C)\cong B(I,F',\epsilon_{F'}^\ast TC)
\]
for $C\in\calC$.
\end{prop}
\begin{proof}
The given assumption on the augmentation of $F$ implies that $T\epsilon\colon TF\rightarrow T$ is a map of monads, and this gives rise to the augmentation on $F'$.  The claimed isomorphism of bar constructions follows from \cref{lem:basechangebar} and the isomorphisms $T'\epsilon_F^\ast\simeq \epsilon_{F'}^\ast T$ and $\epsilon_{F'!}T'\simeq T\epsilon_{F!}$.
\end{proof}

\begin{prop}\label{prop:homologydistributive}
Let $\calP$ be a discrete additive theory, and fix discrete projective augmented $\calP$-algebras $T$ and $F$. Suppose that we have chosen a distributive law $c\colon FT\rightarrow TF$ such that $\epsilon_T\epsilon_F\circ c = \epsilon_F\epsilon_T$. Then the composite monad $T\circ F$ is augmented, and if each $H_n(T)$ is projective, then $H_n(T\circ F)\cong \bigoplus_{i+j=n}H_i(F)\circ H_j(T)$.
\end{prop}
\begin{proof}
It is easily verified that $\epsilon_T\epsilon_F$ makes $TF$ into an augmented monad. Now write the monadic distributive square associated to the composite $TF$ as
\begin{center}\begin{tikzcd}
\calD'\ar[d,"U'",shift left=1mm]\ar[r,"V'",shift left=1mm]&\calD\ar[l,"F'",shift left=1mm]\ar[d,"U",shift left=1mm]\\
\calC'\ar[r,"V",shift left=1mm]\ar[u,"T'",shift left=1mm]&\calC\ar[u,"T",shift left=1mm]\ar[l,"F",shift left=1mm]
\end{tikzcd}.\end{center}
Then there is a natural map
\[
\epsilon_{TF!}\epsilon_{TF}^\ast\simeq \epsilon_{F!}\epsilon_{T'!}\epsilon_{F'}^\ast\epsilon_T^\ast\rightarrow\epsilon_{F!}\epsilon_F^\ast\epsilon_{T!}\epsilon_T^\ast
\]
which we claim is an isomorphism; here, these functors are to be interpreted in the derived sense. It is sufficient to verify that this map induces
\[
V\epsilon_{T'!}\epsilon_{F'}^\ast TP\simeq V \epsilon_F^\ast \epsilon_{T!} TP
\]
for $P\in\calP$. The right hand side is simply $P$, and we compute the left hand side to be
\[
V\epsilon_{T'!}\epsilon_{F'^\ast}TP\simeq V\epsilon_{T'!}T'\epsilon_F^\ast P\simeq V\epsilon_F^\ast P\simeq P.
\]
When each $H_n(T)$ is projective, we can split $\epsilon_{T!}\epsilon_T^\ast P\simeq \bigoplus_{n\geq 0}\Sigma^n H_n(T)_P$ for $P\in\calP$. Thus in this case we have
\begin{align*}
H_n(T\circ F)_P &= \pi_n \epsilon_{TF!}\epsilon_{TF}^\ast P\cong \pi_n \epsilon_{F!}\epsilon_F^\ast\epsilon_{T!}\epsilon_T^\ast P\\
&\cong \pi_n \bigoplus_{k\geq 0}\Sigma^k \epsilon_{F!}\epsilon_F^\ast H_k(T)_P\cong \bigoplus_{i+j=n} (H_i(F)\circ H_j(T))_P
\end{align*}
as claimed.
\end{proof}

\subsection{Koszul resolutions}\label{ssec:koszulresolutions}

Our goal for the rest of this section is to generalize Priddy's theory of Koszul algebras and Koszul resolutions \cite{priddy1970koszul} to the setting of algebras over additive theories. The approach we take is strongly influenced by the approach taken in \cite[Section 4]{rezk2017rings}. For us, the purpose of this theory is to give concrete tools for certain homological computations, and so everything that follows should be interpreted as taking place within a $1$-category; in particular all of our theories and algebras are discrete.
Fix an additive theory $\calP$ and $\calP$-algebra $F$.

\begin{defn}
\hphantom{blank}
\begin{enumerate}
\item $F$ is a \textit{filtered algebra} if we have chosen a filtration $F = \colim_{n\rightarrow\infty} F_{\leq n}$ of $F$ by subbimodules such that
\begin{enumerate}
\item $I = F_{\leq 0}\subset F$ is the unit;
\item The product on $F$ restricts to maps $F_{\leq n}\circ F_{\leq m}\rightarrow F_{\leq n+m}$.
\end{enumerate}
\item $F$ is a \textit{graded algebra} if we have chosen a decomposition $F = \oplus_{n\geq 0}F[n]$ of bimodules such that
\begin{enumerate}
\item $I = F[0]\subset F[n]$ is the unit;
\item The product on $F$ restricts to maps $F[n]\circ F[m]\rightarrow F[n+m]$.
\end{enumerate}
In particular, if $F$ is graded then $F$ is augmented.
\item The \textit{associated graded algebra} of a filtered algebra $F$ is the graded algebra $\gr F$ given by
\[
\gr F = \bigoplus_{m\geq 0}F[m],\qquad F[m] = \coker(F_{\leq m-1}\rightarrow F_{\leq m}),
\]
with multiplication induced by that on $F$.
\item $F$ is a \textit{projective filtered algebra} if both $F$ and $\gr F$ are projective.
\tqed
\end{enumerate}
\end{defn}

Suppose now that $F$ is a projective filtered algebra. For $M\in\LMod_\calP^\heart$, there is a filtration $C^\un(F,F,M) = \colim_{m\rightarrow\infty}C^\un(F,F,M)[\leq\! m]$ obtained by declaring
\[
C^\un_n(F,F,M)[\leq\! m]=\im\left(\bigoplus_{m_1+\cdots+m_n=m}FF_{\leq m_1}\cdots F_{\leq m_n}\rightarrow C^\un_n(F,F,M)\right),
\]
and this induces a filtration on $C(F,F,M)$. In particular, we obtain the associated graded complex
\[
\gr C(F,F,M) = \bigoplus_{m\geq 0}C(F,F,M)[m],
\]
where by definition $C(F,F,M)[m]$ fits into a short exact sequence
\[
0\rightarrow C(F,F,M)[\leq\! m-1]\rightarrow C(F,F,M)[\leq\! m]\rightarrow C(F,F,M)[m]\rightarrow 0.
\]

\begin{lemma}\label{lem:barfiltration}
Fix notation as above. Then
\begin{enumerate}
\item $\gr C(F,F,M) = F C(I,\gr F,\ol{M})$;
\item Explicitly,
\[
C_n(F,F,M)[m] = \bigoplus_{\substack{m_1+\cdots+m_n=m\\ m_1,\ldots,m_n\geq 1}}F F[m_1]\cdots F[m_n] M;
\]
\item In particular,
\begin{enumerate}
\item $C_n(F,F,M)[\leq\! m] = 0$ for $n>m$;
\item $C_n(F,F,M)[\leq\! n] = C_n(F,F,M)[n] = F[1]^{\circ n}(M)$.
\end{enumerate}
\end{enumerate}
\end{lemma}
\begin{proof}
Immediate from the definitions.
\end{proof}

If $F$ is augmented, then the above filtration on $C(F,F,M)$ induces a filtration on $C(I,F,M)$. When $F$ is graded, this filtration is split on $C(I,F,\ol{P})$, yielding gradings $H_\ast (F)_P = H_\ast \bigoplus_{n\geq 0}C(I,F,\ol{P})[n]$ and $H^\ast(F) = \prod_{n\geq 0}H^\ast(F)[n]$.

\begin{defn}\label{def:koszul}
Fix a $\calP$-algebra $F$. Say that $F$ is a \textit{homogeneous Koszul $\calP$-algebra} if
\begin{enumerate}
\item $F$ is projective and has been equipped with a grading;
\item $H_n(F)[m]=0$ for $n\neq m$.
\end{enumerate}
Say that $F$ is a \textit{Koszul $\calP$-algebra} if
\begin{enumerate}
\item[(1$'$)] $F$ has been equipped with a projective filtration;
\item[(2$'$)] $\gr F$ is a homogeneous Koszul $\calP$-algebra.
\tqed
\end{enumerate}
\end{defn}

Now suppose that $F$ is a projective filtered algebra, and fix a $\calP$-projective $F$-module $M$. The filtration $C(F,F,M)\simeq\colim_{m\rightarrow\infty} C(F,F,M)[\leq\! m]$ gives rise to a spectral sequence of signature
\[
E^1_{p,q} = F H_q(\gr F)[p](M)\Rightarrow H_q C(F,F,M),\qquad d^r_{p,q}\colon E^r_{p,q}\rightarrow E^r_{p-r,q-1}.
\]

\begin{lemma}\label{lem:rezcvge}
Suppose either of the following is satisfied:
\begin{enumerate}
\item Filtered colimits in $\LMod_\calP^\heart$ are exact;
\item The connectivity of $C(F,F,M)[m]$ goes to $\infty$ as $m$ goes to $\infty$.
\end{enumerate}
Then the above spectral sequence converges. In particular, this holds if $F$ is Koszul.
\end{lemma}
\begin{proof}
Abbreviate $C = C(F,F,M)$. We apply the convergence criteria of \cite[Proposition A.1.2]{balderrama2021deformations}, itself just a variant of standard theorems. Case (1) is essentially standard, so let us just consider case (2). Write $C[p,q] = \coker({C[\leq \!p-1]}\rightarrow C[\leq \! q])$. We must verify the following conditions:

\begin{enumerate}
\item[(a)]The connectivity of $C[\leq\! m]$ goes to $\infty$ as $m$ goes to $-\infty$. 
\item[(b)]$\colim_{r\rightarrow\infty}H_\ast C[p,p+r]\cong H_\ast \colim_{r\rightarrow\infty}C[p,p+r]$ for $p\in\bbZ\cup\{-\infty\}$;
\item[(c$'$)]For all $q\in\bbZ$, the map $H_q C[\leq\! p]\rightarrow H_q C[\leq\! p+1]$ is an isomorphism for all but finitely many $p$.
\end{enumerate}
Condition (a) is clear, as $C[\leq\!-1]=0$. To verify (b) and (c$'$), it is sufficient to verify that for any fixed $n$ and $p$, the map $H_n C[p,p+r]\rightarrow H_n C[p,p+r+1]$ is an isomorphism for all sufficiently large $r$. As $\coker(C[p,p+r]\rightarrow C[p,p+r+1]) = C[p+r+1]$, this follows from the given connectivity assumption.
\end{proof}

Define the chain complex $K(F,F,M)$ by $K_p(F,F,M) = E^1_{p,p}$ as above, with differential obtained from the $d^1$ differential of this spectral sequence. Then $K(F,F,M)$ is a chain complex of the form
\[
F H_0(\gr F)(M)\leftarrow F H_1(\gr F)(M)\leftarrow F H_2(\gr F)(M)\leftarrow\cdots,
\]
or more memorably,
\[
K(F,F,M) = F H_\ast(\gr F)(M),
\]
and this sits as a subcomplex of the bar resolution of $M$.

\begin{theorem}\label{thm:koszulres}
Let $F$ be a Koszul $\calP$-algebra, and let $M$ be an $F$-module which is projective over $\calP$. Then there is a splitting $C(F,F,M)\cong K(F,F,M)\oplus C'$, where $C'$ is a contractible chain complex. In particular,
\[
M\leftarrow K(F,F,M)
\]
is a projective resolution of $M$.
\end{theorem}
\begin{proof}
As $F$ is Koszul, the spectral sequence $F H_q(\gr F)[p](M)\Rightarrow H_q C(F,F,M)$ collapses into a projective resolution $K(F,F,M)\rightarrow M$. The inclusion $K(F,F,M)\rightarrow C(F,F,M)$ is a quasiisomorphism of bounded below complexes of projectives, allowing for the indicated splitting.
\end{proof}

Fix a filtered $\calP$-algebra $F$ and $F$-modules $M$ and $M'$ with $M$ projective over $\calP$. Then there is a \textit{Koszul complex}
\[
K_F(M,M') = \LMod_F(K(F,F,M),M').
\]
This is a quotient of the cobar complex $C_F(M,M')$, and models $\EXT_F(M,M')$ when $F$ is Koszul. We will describe these complexes more explicitly in \cref{ssec:koszulcomplexes}.

\begin{ex}\label{ex:steenrod}
(1)~~The motivating example of a Koszul algebra is the Steenrod algebra \cite{priddy1970koszul}. For simplicity, take $\calA$ to be the mod $2$ Steenrod algebra; then $\calA$ is Koszul with respect to the length filtration on $\calA$. The Koszul complex $K_\calA(\bbF_2,\bbF_2)$ can be enriched to a complex of graded vector spaces, and is known as the \textit{lambda algebra}. See \cref{ex:lambdaalgebra} for more on this example.

(2)~~More generally, there are Steenrod algebras for other forms of mod $p$ cohomology, such as in unstable, motivic, equivariant, synthetic, or other flavors of homotopy theories, and examples suggest that one can expect these to be Koszul as well. These examples require a fairly general notion of Koszul algebra: they are not generally augmented, their coefficients rings do not generally live in their center, and they need not be ordinary graded algebras at all, such as in the equivariant setting where one has additional Mackey functor structure, or in unstable settings where one must account for instability conditions. We will give the unstable example more explicitly in \cref{ex:unstablesteenrod}.
\tqed
\end{ex}

We end by noting the following stability properties of Koszulity.

\begin{lemma}\label{prop:koszulcomposition}
Fix a monadic distributive square
\begin{center}\begin{tikzcd}
\LMod_{F'}^\heart\ar[d,"U'",shift left=1mm]\ar[r,"V'",shift left=1mm]&\LMod_{\calP'}^\heart\ar[l,"F'",shift left=1mm]\ar[d,"U",shift left=1mm]\\
\LMod_F^\heart\ar[r,"V",shift left=1mm]\ar[u,"T'",shift left=1mm]&\LMod_\calP^\heart\ar[u,"T",shift left=1mm]\ar[l,"F",shift left=1mm]
\end{tikzcd},\end{center}
where $\calP$ and $\calP'$ are additive theories and $F$ and $F'$ are projective algebras over them. In particular, there is a distributive law $c\colon FT\rightarrow TF$, and $\LMod_{F'}^\heart\simeq \LMod_{T\circ F}^\heart$ where $T\circ F$ is composition of $T$ and $F$ as a monad on $\LMod_\calP^\heart$.
\begin{enumerate}
\item Suppose that $F$ is a projective filtered algebra, with filtration compatible with the distributive law. If $F$ is Koszul over $\calP$, then $F'$ is Koszul over $\calP'$.
\item Suppose that $F$ and $T$ are projective filtered algebras, with filtrations compatible with the distributive law. Filter $T\circ F$ by $(T\circ F)_{\leq n} = \im(\bigoplus_{i+j=n}T_{\leq i}\circ F_{\leq j}\rightarrow T\circ F)$, so that $\gr(T\circ F)\cong \gr T\circ \gr F$. Then $T\circ F$ is Koszul.
\end{enumerate}
\end{lemma}
\begin{proof}
Given \cref{thm:koszulres}, these follow from \cref{lem:tensoraug} and \cref{prop:homologydistributive}.
\end{proof}

\subsection{Quadratic algebras}\label{ssec:quadraticduality}

Our goal in this subsection is to describe the structure and cohomology of homogeneous Koszul algebras. Fix an additive theory $\calP$, and let $H$ be a $\calP$-bimodule. We can then form the free algebra on $H$ as a type of tensor algebra:
\[
T H = \bigoplus_{n\geq 0}T_n H,\qquad T_n H = H^{\circ n},
\]
with standard multiplication. The right adjoint to this is
\[
\widehat{T}H^\vee = (TH)^\vee = \prod_{n\geq 0}(H^{\circ n})^\vee = \prod_{n\geq 0} H^{\vee\circ n};
\]
this also carries an obvious multiplication, but we want the slightly less obvious multiplication, obtained by twisting the identifications
\[
\widehat{T}_i(H^\vee)\circ \widehat{T}_j(H^\vee) = H^{\vee \circ i}\circ H^{\vee \circ j}\cong H^{\vee\circ (i+j)}\cong \widehat{T}_{i+j}(H^\vee)
\]
by $(-1)^{ij}$; write this multiplication as $\wr$. In fact, these two choices are isomorphic by $x\mapsto (-1)^{\frac{|x|(|x|-1)}{2}}x$, so the distinction is essentially invisible in the examples we will consider; the purpose of this choice is to ensure compatibility with the cobar complex in \cref{thm:quadraticduality} below.

Given a subfunctor $R\subset H\circ H$, we may form the \textit{quadratic algebra}
\begin{gather*}
T(H,R) = T(H)/R = \bigoplus_{n\geq 0}T_n(H,R),\\
T_n(H,R) = \coker\left(\sum_{i+j=n}H^{\circ i-1}\circ R\circ H^{\circ j-1}\rightarrow H^{\circ n}\right),
\end{gather*}
with multiplication inherited from $T H$. Likewise, given some subfunctor $R'\subset H^\vee$, we may form the monad 
\begin{gather*}
\widehat{T}(H^\vee,R') = \prod_{n\geq 0}\widehat{T}_n(H^\vee,R'),\\
\widehat{T}_n(H^\vee,R') = \coker\left(\sum_{i+j=n}H^{\vee\circ i-1}\circ R'\circ H^{\vee\circ j-1}\rightarrow H^{\vee \circ n}\right),
\end{gather*}
with multiplication inherited from $\widehat{T}H^\vee$. This is no longer guaranteed to preserve limits in general, but it will in the cases of relevance to us.

\begin{lemma}\label{lem:projquad}
The following are equivalent:
\begin{enumerate}
\item $T(H,R)$ is projective;
\item $H$ is projective, and for all $P\in\calP$, the map $R_P\rightarrow (H\circ H)_P$ admits a splitting.
\end{enumerate}
\end{lemma}
\begin{proof}
(1)$\Rightarrow$(2). If $T(H,R)$ is projective, then $H$ is projective as $H = T_1(H,R)$ is a summand of $T(H,R)$. The sequence
\[
0\rightarrow R\rightarrow H\circ H\rightarrow T_2(H,R)\rightarrow 0
\]
is exact, so that as $T_2(H,R)$ is projective, it is levelwise split as claimed.

(2)$\Rightarrow$(1) If $H$ is projective and $R\subset H\circ H$ is levelwise split, then each $H^{\circ i-1}\circ R \circ H^{\circ j-1}\subset H^{\circ n}$ is levelwise split. Thus $T_n(H,R)$ is projective, from which it follows that $T(H,R)$ is projective.
\end{proof}

Call a pair $(H,R)$ satisfying the conditions of \cref{lem:projquad} a \textit{quadratic datum}. Fix now a quadratic datum $(H,R)$. By projectivity, the short exact sequence
\[
0\rightarrow R\rightarrow H\circ H\rightarrow T_2(H,R)\rightarrow 0
\]
dualizes to a short exact sequence
\[
0\leftarrow R^\vee\leftarrow H^\vee\circ H^\vee\leftarrow R^\perp\leftarrow 0;
\]
the pair $(H^\vee,R^\perp)$ could be called the \textit{dual quadratic datum} to $(H,R)$, though with this name dual quadratic data are not themselves quadratic data.

\begin{lemma}
Fix a quadratic datum $(H,R)$. Then the monad $\widehat{T}(H^\vee,R^\perp)$ preserves limits, and is thus the right adjoint of a coalgebra.
\end{lemma}
\begin{proof}
As in the proof of \cref{lem:projquad}, the hypotheses imply that $\widehat{T}(H^\vee,R^\perp)$ is levelwise a summand of $\widehat{T}(H^\vee)$, and this proves the lemma.
\end{proof}

\begin{rmk}
The coalgebra which is left adjoint to $\widehat{T}(H^\vee,R^\perp)$ may be identified explicitly as the coquadratic coalgebra
\[
\bigoplus_{n\geq 0}\left(\bigcap_{i+j=n}H^{\circ i-1}\circ R\circ H^{\circ j-1}\right)\subset T(H),
\]
but we will not have a chance to make use of this.
\tqed
\end{rmk}

\begin{theorem}\label{thm:quadraticduality}
\hphantom{blank}
\begin{enumerate}
\item Let $(H,R)$ be a quadratic datum. Then $H^1(T(H,R))[1] \cong H^\vee$, and the inclusion $H^\vee\subset H^\ast(T(H,R))$ extends to an isomorphism $\widehat{T}(H^\vee,R^\perp)\cong\prod_{n\geq 0}H^n(T(H,R))[n]$ of monads.
\item Let $F = \bigoplus_{n\geq 0}F[n]$ be a homogeneous Koszul algebra, and let $R = \ker(F[1]\circ F[1]\rightarrow F[2])$. Then $F\cong T(F[1],R)$ and $H^\ast(F)\cong \widehat{T}(F[1]^\vee,R^\perp)$.
\end{enumerate}
\end{theorem}
\begin{proof}
(1)~~ Abbreviate $C = C(I,T(H,R),\bs)$. By \cref{lem:barfiltration}, there is an isomorphism
\[
H_n C[n] \cong \ker\left(H^{\circ n}\rightarrow\bigoplus_{i+j=n}H^{\circ i-1}\circ T_2(M,R)\circ H^{\circ j-1}\right).
\]
This is left adjoint to $\widehat{T}_n(H^\vee,R^\perp)$, proving (1) additively, and multiplicative compatibility follow by comparing our choice of product on $\widehat{T}(H^\vee,R^\perp)$ with the construction given in \cref{prop:cobarcomposition}. 

(2)~~ We must show only $F\simeq T(F[1],R)$, for the remaining claims follow from Koszulity and (1). By construction, the inclusion $F[1]\rightarrow F$ extends to a map $T(F[1],R)\rightarrow F$ of algebras, and we must verify that this is an isomorphism. By Koszulity, the sequences
\begin{gather*}
\bigoplus_{\substack{i+j=n\\ i,j>0}}F[i]\circ F[j]\rightarrow F[n]\rightarrow 0\\
\bigoplus_{\substack{i+j+k=m\\i,j,k>0}}F[i]\circ F[j]\circ F[k]\rightarrow\bigoplus_{\substack{r+s=m\\r,s>0}}F[r]\circ F[s]\rightarrow F[m]
\end{gather*}
are exact for $n>1$ and $m>2$. The first implies that each $F[1]^{\circ n}\rightarrow F[n]$ is surjective, and thus so too is $T(F[1],R)\rightarrow F$. The second implies that any relation seen in the multiplication $F[r]\circ F[s]\rightarrow F[m]$ with $r+s=m$ and either $r>1$ or $s>1$ is already generated in relations among $F[i]$ with $i<r$ or $i<s$. Thus $T_n(F[1],R)\rightarrow F[n]$ is an injection, so an isomorphism. As $T(F[1],R)\rightarrow F$ is a direct sum of isomorphisms, it is itself an isomorphism.
\end{proof}

\begin{ex}\label{ex:ltkoszul}
Let $R=W[[a]]$ where $W=W(\kappa)$ is the ring of $2$-typical Witt vectors on a perfect field $\kappa$ of characteristic $2$. Define an $R$-bimodule $\Gamma[1]$ as follows. As a left $R$-module, $\Gamma[1]\simeq R\{Q_0,Q_1,Q_2\}$. The right $R$-module structure is determined by
\begin{align*}
Q_i\lambda &= \lambda^\sigma Q_i,\quad \lambda\in W\\
Q_0 a &= a^2 Q_0 - 2a Q_1 + 6 Q_2\\
Q_1 a &= 3 Q_0 + a Q_2\\
Q_2 a &= -a Q_0 + 3 Q_1,
\end{align*}
where $(\bs)^\sigma$ is the Frobenius automorphism of $W$. Let $R\subset\Gamma[1]\circ\Gamma[1]$ by spanned by
\begin{align*}
Q_1Q_0 &= 2 Q_2 Q_1 - 2 Q_0 Q_2,\\
Q_2 Q_0 &= Q_0Q_1+a Q_0Q_2-2Q_1Q_2.
\end{align*}
Now set $\Gamma = T(\Gamma[1],R)$. This is the algebra of additive power operations for a certain Morava $E$-theory at height $h=2$ and $p=2$ computed by Rezk \cite{rezk2008power} (cf.\ \cref{ex:ltcurve}), and is Koszul. By \cref{thm:quadraticduality}, there is an isomorphism $H^\ast(\Gamma) = \widehat{T}(\Gamma[1]^\vee,R^\perp)$; we can compute this explicitly as follows. As $\Gamma[1]$ is finitely generated and free as a left $R$-module, $\Gamma[1]^\vee$ is a bimodule, and is finitely generated and free as a right $R$-module on a basis dual to that of $\Gamma[1]$; write this as $\Gamma[1]^\vee = \{Q^0,Q^1,Q^2\}R$. The left $R$-module structure is given by
\begin{align*}
\lambda Q^i &= Q^i\lambda^\sigma,\quad \lambda\in W\\
a Q^0 &= Q^0 a^2+3 Q^1-Q^2a\\
a Q^1 &= -2 Q^0 a + 3 Q^2\\
a Q^2 &= 6 Q^0 + Q^1 a.
\end{align*}
The subspace $R^\perp\subset\Gamma[1]^\vee\circ\Gamma[1]^\vee$ is spanned by
\begin{gather*}
Q^0Q^0,\qquad Q^1Q^1,\qquad Q^2Q^2,\qquad Q^1Q^0+Q^0Q^2,\\
Q^1Q^2+2 Q^0Q^1,\qquad Q^2Q^1-2Q^0Q^2,\\
Q^2 Q^0 - 2 Q^0 Q^1+ Q^0Q^2 a.
\end{gather*}
Thus
\[
H^\ast(\Gamma)\cong\{1,Q^0,Q^1,Q^2,Q^0Q^1,Q^0Q^2\}R,
\]
with multiplicative structure determined by the preceding relations.
\tqed
\end{ex}

\subsection{Koszul complexes}\label{ssec:koszulcomplexes}

Our goal in this section is to describe the Koszul complexes computing $\Ext$ over a Koszul algebra. We begin with the homogeneous case.

Fix an additive theory $\calP$. Fix a quadratic datum $(H,R)$, and write $F = T(H,R)$. Fix $M,N\in\LMod_F^\heart$ with $M$ projective over $\calP$. Recall from \cref{ssec:koszulresolutions} the Koszul complex $K_F(M,N)$, which is a quotient of the cobar complex $C_F(M,N)$ satisfying $K_F^n(M,N) = H^n(F)[n](N)(M)$. These groups are described by \cref{thm:quadraticduality}, and it remains only to describe the Koszul differential. In some cases, this may be determined by analyzing the surjective map $C_F(M,N)\rightarrow K_F(M,N)$ directly, but we can also proceed as follows. 

Observe first that the algebra structure of $H^\ast(F)$ gives pairings
\[
\wr\colon K_F(M,N)\otimes K_F(N,L)\rightarrow K_F(M,L)
\]
of graded objects. These are compatible with the pairings of \cref{prop:cobarcomposition}, and so are pairings of chain complexes. Observe next that as $F$ is generated by $F[1] = H$, the $F$-module structure on $M$ is determined by a map $H(M)\rightarrow M$, i.e.\ an element of $H^\vee(M)(M) = K^1_F(M,M)$; write $Q^M$ for this element twisted by $-1$, and define $Q^N$ in the same way.

\begin{theorem}\label{thm:homogkoszuldifferential}
The differential on $K_F(M,N)$ is given by
\[
\delta\colon K^n_F(M,N)\rightarrow K^{n+1}_F(M,N),\qquad \delta(f) = Q^M \wr f - (-1)^n f \wr Q^N.
\]
\end{theorem}
\begin{proof}
Recall $C^n_F(M,N) = \Hom_\calP(F^{\circ n}_+M,N)$, where $F \cong I\oplus F_+$, and recall the differential on $C^n_F(M,N)$ from \cref{ssec:barcon}, of the form
\[
\delta = \delta_0 + \sum_{1\leq i \leq n}(-1)^i \delta_i + (-1)^{n+1}\delta_{n+1}\colon C^n_F(M,N)\rightarrow C^{n+1}_F(M,N).
\]
By construction, the inner sum $\sum_{1\leq i \leq n}(-1)^i \delta_i$ is killed by the quotient map $C_F(M,N)\rightarrow K_F(M,N)$. On the other hand,
\begin{align*}
\delta_0(f) &= m\circ Tf = -(-1)^{n} f\wr Q^N;\\
 \delta_{n+1}(f) &= f\circ T^n m = -(-1)^{n} Q^M\wr f.
\end{align*}
Combining these proves the theorem.
\end{proof}

\begin{rmk}
Somewhat more explicitly, $\delta_0(f) = -(-1)^n f\wr Q^N$ is the composite
\begin{align*}
K^n_F(M,N)&=\LMod_\calP(M,H^n(F)(N))\\
&\rightarrow \LMod_\calP(M,H^n(F)\circ H^1(F)(N))\\
&\rightarrow\LMod_\calP(M,H^{n+1}(F)(N)) = K^{n+1}_F(M,N),
\end{align*}
and $\delta_{n+1}(f) = -(-1)^n Q^M\wr F$ is the composite
\begin{align*}
K^n_F(M,N)&=\LMod_\calP(M,H^n(F)(N))\\
&\rightarrow\LMod_\calP(H_1(F)(M),H^n(F)(N))\\
&\simeq\LMod_\calP(M,H^1(F)\circ H^n(F)(N))\\
&\rightarrow \LMod_\calP(M,H^{1+n}(F)(N)) = K^{1+n}_F(M,N).
\end{align*}
In either case, the first map which is not an isomorphism encodes the $F$-module structure on $M$ or $N$, and the second map which is not an isomorphism is obtained from the multiplication on $H^\ast(F)$.
\tqed
\end{rmk}

Now fix a possibly nonhomogeneous Koszul algebra $F$, and fix $M,N\in\LMod_F^\heart$ with $M$ projective over $\calP$. As before, there is a Koszul complex $K_F(M,N) = H^\ast(\gr F)(N)(M)$, still a quotient of $C_F(M,N)$, and we would like to identify its differential.

Write $qR = \ker(F_{\leq 1}\circ F_{\leq 1}\rightarrow F_{\leq 2})$, so that $(F_{\leq 1},qR)$ is a quadratic datum. Observe that $\bigoplus_{m\geq 0}F_{\leq m}$ is a graded algebra, and that the inclusion of $F_{\leq 1}$ extends multiplicatively to $T(F_{\leq 1},qR)\rightarrow\bigoplus_{m\geq 0}F_{\leq m}$.

\begin{theorem}\label{thm:nonhomogkoszul}
\hphantom{blank}
\begin{enumerate}
\item The map $T(F_{\leq 1},qR)\rightarrow\bigoplus_m F_{\leq m}$ is an isomorphism of graded algebras;
\item $T(F_{\leq 1},qR)$ is a homogeneous Koszul algebra;
\item The surjection $T(F_{\leq 1},qR)\rightarrow F$ gives rise to a short exact sequence
\[
0\rightarrow H^\ast(\gr F)\rightarrow H^\ast(T(F_{\leq 1},qR))\rightarrow H^{\ast-1}(\gr F)\rightarrow 0,
\]
which is split when $F$ is augmented.
\end{enumerate}
In particular, $K_F(M,N)\subset K_{T(F_{\leq 1},qR)}(M,N)$ is a subcomplex with differential on the target described by \cref{thm:quadraticduality}.
\end{theorem}
\begin{proof}
(1) Define a finite filtration on each $T_n(F_{\leq 1},qR)$ by 
\[
T_n(F_{\leq 1},qR)[\leq\! m] = \im\left(\bigoplus_{\substack{\epsilon_1+\cdots+\epsilon_n = m \\ \epsilon_1,\ldots,\epsilon_n\in\{0,1\}}} F_{\leq \epsilon_1}\circ\cdots\circ F_{\leq \epsilon_n}\rightarrow T_n(F_{\leq 1},qR)\right).
\]
The map $T_n(F_{\leq 1},qR)\rightarrow F_{\leq n}$ is compatible with filtrations. As $\gr F$ is quadratic, this map induces an isomorphism $\gr T_n(F_{\leq 1},qR)\cong \bigoplus_{m\leq n}F[m]\cong \gr F_{\leq n}$ of associated graded bimodules, and is therefore itself an isomorphism.

(2--3) The filtration $T(F_{\leq 1},qR) = \colim_{m\rightarrow\infty}T(F_{\leq 1},qR)[\leq\!m]$ is multiplicative, though it need not satisfy $T(F_{\leq 1},qR)[\leq \! {0}] = I$. As $\gr F$ is quadratic, the associated graded $\gr T(F_{\leq 1},qR)\cong TI\circ\gr F$ is a ``polynomial algebra'' on $\gr F$, defined by the evident distributive law. By \cref{prop:koszulcomposition}, $\gr T(F_{\leq 1},qR)$ is a homogeneous Koszul algebra, and $H^n(\gr T(F_{\leq 1},qR)) \cong H^n(\gr F) \oplus H^{n-1}(\gr F)$. The spectral sequence $H^\ast(\gr T(F_{\leq 1},qR))\Rightarrow H^\ast (T(F_{\leq 1},qR))$ collapses into a two-step filtration of $H^\ast(T(F_{\leq 1},qR))$ both proving Koszulity and providing the indicated short exact sequences.
\end{proof}

\begin{ex}[cf.\ \cite{bruner1985example}]\label{ex:lambdaalgebra}
Let $\calA$ denote the mod $2$ Steenrod algebra, so that $\calA$ is Koszul with respect to its length filtration. As $\calA$ is an ordinary $\bbZ$-graded algebra, we can upgrade the Koszul complex $K_\calA(\bbF_2,\bbF_2) \cong H^\ast(\gr \calA)$ to an ordinary differential graded algebra in graded $\bbF_2$-modules. This is the mod $2$ \textit{lambda algebra}.

Explicitly, the lambda algebra may be computed as follows. The algebra $\calA$ is generated by $\calA_{\leq 1} = \bbF_2\{\Sq^n:n\geq 0\}$, subject to the quadratic relations
\[
\Sq^{2s-r-1}\Sq^s = \sum_i \binom{r-i-1}{i}\Sq^{2s-1-i}\Sq^{s-r+i}
\]
for $r\geq 0$, together with the additional nonquadratic relation $\Sq^0 = 1$. Thus $\gr\calA$ and $T(\calA_{\leq 1},qR)$ are generated by $\Sq^n$ for $n\geq 1$ and $n\geq 0$ respectively, subject to quadratic relations of the same form. By \cref{thm:quadraticduality}, the ordinary cohomology algebras $H^\ast(\gr \calA)$ and $H^\ast(T(\calA_{\leq 1},qR))$ are generated by elements $\lambda_n$ dual to $\Sq^{n+1}$, subject to the dual relations
\[
\lambda_a\lambda_{2a+b+1} = \sum_j \binom{b-j-1}{j}\lambda_{a+b-j}\lambda_{2a+1+j}
\]
for $b\geq 0$; the distinction between them is that $\lambda_{-1}$ is an element of the former but not the latter.

The standard action of $\calA$ on $\bbF_2$ lifts to the action of $T(\calA_{\leq 1},qR)$ on $\bbF_2$ where $\Sq^0$ acts by the identity. Now $K_{T(\calA_{\leq 1},qR)}(\bbF_2,\bbF_2)=H^\ast(T(\calA_{\leq 1},qR))$, and \cref{thm:homogkoszuldifferential} shows that the Koszul differential is given by the commutator $[\lambda_{-1},\bs]$. Thus the lambda algebra $H^\ast(\gr\calA)\subset H^\ast(T(\calA_{\leq 1},qR))$ is closed under the commutator with $\lambda_{-1}$, and we recover this description of its differential.
\tqed
\end{ex}

\subsection{The PBW criterion}\label{ssec:pbw}

Fix a quadratic algebra $F =  T(H,R)$.

\begin{defn}
An \textit{additive decomposition} of $F$ consists of a decomposition
\[
H\simeq\bigoplus_{i\in B}H_i
\]
of bimodules, together with a subset $S$ of the set $B^\ast$ of words in the alphabet $B$ such that the map $\bigoplus_{w\in S}F_w\rightarrow F$ is an isomorphism, where if $w = (s_1,\ldots,s_n)$ then $F_w = H_{s_1}\circ\cdots\circ H_{s_n}$. This is a \textit{PBW decomposition} if moreover
\begin{enumerate}
\item A word $w = (s_1,\ldots,s_n)$ lives $S$ if and only if each pair $(s_i,s_{i+1})$ lives in $S$;
\item $B$ is equipped with an order, and so $B^\ast$ with the lexicographic order, such that for all $w',w''\in S$, either $w'w''\in S$ or the composite
\[
F_{w'}\circ F_{w'}\rightarrow F\circ F\rightarrow F\rightarrow \bigoplus_{w\leq w'w''}F_w
\]
is null.
\tqed
\end{enumerate}
\end{defn}

Suppose now that $F$ is equipped with a PBW decomposition, with notation as in the definition. Abbreviate $C = C(I,F,\ol{P})$ for varying $P$. For $w\in B^\ast$, define
\[
C_k[\leq\! w] = \bigoplus_{\substack{w_1,\ldots,w_k\in S\\w_1\cdots w_k\leq w}}F_{w_1}\circ\cdots\circ F_{w_k}\subset C_k,
\]
and similarly define $C_k[<\!w]$ and $C_k[w]$. The PBW criterion implies that $C_k[\leq\! w]$ and $C_k[<\!w]$ are quotient chain complexes of $C$, and by construction there are short exact sequences
\[
0\rightarrow C[w]\rightarrow C[\leq \!w]\rightarrow C[<\!w]\rightarrow 0.
\]

\begin{prop}\label{prop:pbw}
Suppose that $F$ is equipped with a PBW decomposition, and fix notation for this as above. Suppose that the lexicographic ordering on $B^\ast$ is well-founded when restricted to subsets of words of a fixed length, and that $H_\ast C\rightarrow H_\ast\lim_w C[\leq\! w]$ is an injection. Then $F$ is Koszul.
\end{prop}
\begin{proof}
The proof is exactly as in \cite[Theorem 5.3]{priddy1970koszul}; we recall the construction in dual form. Under the given hypotheses, it is sufficient to fix a word $w$ of length $m$ and verify that $C[w]$ is acyclic outside degree $m$. To that end, one constructs $s\colon C[w]_k\rightarrow C[w]_{k+1}$ such that $sd+ds$ is the identity on $C[w]_k$ for $k<m$ as follows. Write $w = (r_1,\ldots,r_m)$, and denote decompositions $w=w_1\cdots w_k$ by $(r_1,\ldots,r_{n_1};\ldots;r_{n_{k-1}+1},\ldots,r_m)$. Then $s$ is defined on a summand indexed by a decomposition $w=w_1\cdots w_k$ as follows. If this decomposition is of the form $(r_1;\ldots;r_{j-1};r_j,\ldots,r_{j+l};\ldots)$ with $l\geq 1$ and $r_ir_{i+1}\notin S$ for $i<j$, then $s$ is given by twisting the identification with the summand indexed by $(r_1;\ldots;r_j;r_{j+1},\ldots,r_{j+l};\ldots)$ with a sign of $(-1)^j$. On all other summands, $s=0$.
\end{proof}

The finiteness conditions of \cref{prop:pbw} are satisfied in settings where one may reasonably call $F$ a locally finite algebra. On the other hand, there are settings where one may reasonably say that $F$ has a PBW basis which fails to respect the $\calP$-bimodule structure of $F$ and therefore does not give rise to a PBW decomposition. In such cases, it may nonetheless be possible to deduce Koszulity by filtering the failure away, such as in the proof of \cref{thm:nonhomogkoszul}.

\begin{ex}\label{ex:unstablesteenrod}
Let $\calA$ denote the mod $2$ Steenrod algebra. Following \cite[Section 7]{priddy1970koszul}, $\gr \calA$ has a PBW basis of admissibles, and this provides a proof of its Koszulity. Now let $\calU$ be the monad on the category on graded $\bbF_2$-vector spaces whose algebras are the unstable $\calA$-modules. Then $\calU$ is a quotient algebra of $\calA$, for our general definition of an algebra, and the admissible basis of $\gr\calA$ projects to a PBW decomposition of $\gr\calU$. Thus $\calU$ is itself a Koszul algebra. This in fact recovers the unstable lambda algebra:
\[
K_{\calU}(e_n,e_\ast)\cong\Lambda(n)
\]
as chain complexes, up to choices of grading, where $e_a$ denotes a copy of $\bbF_2$ in degree $a$. We will cover a variant of this example in greater detail in \cref{ssec:cohomologydl}.
\tqed
\end{ex}

\section{Plethories}\label{sec:plethories}

This section is concerned with a generalization of the biring triples of Tall-Wraith \cite{tallwraith1968representable}, or plethories of \cite{borgerwieland2005plethystic}.

\subsection{Exponential monads}\label{ssec:exponentialmonads}

Let $\calP$ be a symmetric monoidal additive theory. Write $\CRing_\calP^\heart$ for the category of commutative monoids in $\LMod_\calP^\heart$, so that $\CRing_\calP^\heart\simeq\Model_{S\calP}^\heart$ where $S P = \bigoplus_{n\geq 0}P^{\otimes n}/\Sigma_n$ for $P\in\calP$. We refer to the objects of $\CRing_\calP^\heart$ as $\calP$-rings.

\begin{lemma}\label{lem:expmon}
The following concepts are equivalent:
\begin{enumerate}
\item Colimit-preserving monads $T$ on $\CRing_\calP^\heart$, i.e.\ $S\calP$-algebras;
\item Monads $\bbT$ on $\LMod_\calP^\heart$ which preserve filtered colimits and reflexive coequalizers and which are equipped with the structure of a strong monoidal functor $\bbT\colon (\LMod_\calP^\heart,\oplus)\rightarrow (\LMod_\calP^\heart,\otimes)$ in such a way that for all $X,Y\in\LMod_\calP^\heart$, the dashed arrow in
\begin{center}\begin{tikzcd}
\bbT(\bbT(X)\otimes\bbT(Y))\ar[d,dashed]\ar[r,"\simeq"]&\bbT\bbT(X\oplus Y)\ar[d]\\
\bbT(X)\otimes\bbT(Y)\ar[r,"\simeq"]&\bbT(X\oplus Y)
\end{tikzcd}\end{center}
endows $\bbT(X)\otimes\bbT(Y)$ with the structure of a $\bbT$-algebra.
\end{enumerate}
\end{lemma}
\begin{proof}
This is implicit in \cite{rezk2009congruence}. Given the monad $T$, one can verify that $\bbT = T \circ S$ has the indicated structure. Conversely, given $\bbT$, any $B\in\Alg_\bbT^\heart$ is naturally an object of $\CRing_\calP^\heart$ via the map
\[
B\otimes B\rightarrow\bbT(B)\otimes\bbT(B)\simeq\bbT(B\oplus B)\rightarrow\bbT(B)\rightarrow B,
\]
giving a monadic functor $\Alg_\bbT^\heart\rightarrow\CRing_\calP^\heart$ which preserves sifted colimits. One can then verify that for $A,B\in\Alg_\bbT^\heart$, we have $A\coprod B = A\otimes B$, with $\bbT$-algebra structure analogous to the diagram in (2), and thus $\Alg_\bbT^\heart\rightarrow\CRing_\calP^\heart$ preserves all colimits.
\end{proof}

Monads as in \cref{lem:expmon} are called \textit{exponential monads}.

\begin{defn}\label{def:pleth}
The equivalent data of \cref{lem:expmon} is a \textit{$\calP$-plethory}.
\tqed
\end{defn}

We will generally treat a $\calP$-plethory as its underlying exponential monad. Given a $\calP$-plethory $\Lambda$, we will make use of the notation $\Lambda_{P,P'} = \Lambda(P)(P')$ for $P,P'\in\calP$. We write $\Ring_\Lambda^\heart = \Alg_\Lambda^\heart$, and refer to its objects as $\Lambda$-rings.

Among the main pieces of structure of a $\calP$-plethory $\Lambda$ are maps
\begin{align*}
\Delta^+&\colon\Lambda_P\rightarrow\Lambda_{P\oplus P}\simeq\Lambda_P\otimes\Lambda_P;\\
\epsilon^+&\colon\Lambda_P\rightarrow\Lambda_0\simeq\bbone;\\
\Delta^\times&\colon \Lambda_{P\otimes P'}\rightarrow\Lambda_P\otimes\Lambda_{P'};\\
\epsilon^\times&\colon \Lambda_\bbone\rightarrow\bbone.
\end{align*}
Here, $\Delta^+$ and $\epsilon^+$ come from the diagonal $P\rightarrow P \oplus P$ and unique map $P\rightarrow 0$. The map $\Delta^\times$ is equivalent to a natural transformation $\ev_P\times \ev_{P'}\rightarrow \ev_{P\otimes P'}$, and classifies the multiplication present on $\Lambda$-rings, and likewise $\epsilon^\times$ classifies the multiplicative identity. In fact these maps are present given just the underlying $S\calP$-bimodel of $\Lambda$, only one can no longer interpret them as corresponding to natural operations.

\subsection{Cobialgebroids}\label{ssec:cobialgebroids}

Fix a symmetric monoidal additive theory $\calP$.

\begin{defn}\label{def:cobialgebroid}
A (discrete) \textit{$\calP$-cobialgebroid} is a $\calP$-algebra $\Gamma$ together with a lift of $\Gamma$ to a monoid in the category of oplax symmetric monoidal endofunctors of $\Model_\calP^\heart$, or equivalently, with a lift of $\Gamma^\vee$ to a comonoid in the category of lax symmetric monoidal endofunctors of $\Model_\calP^\heart$.
\tqed
\end{defn}

Denote the category of $\calP$-cobialgebroids by $\coBiAlg_\calP^\heart$.

\begin{lemma}
Let $\calC$ be a symmetric monoidal $1$-category, and let $U\colon\calD\rightarrow\calC$ be a plethystic functor with associated monad $F$ and comonad $F^\vee$. The following are equivalent:
\begin{enumerate}
\item The structure of a symmetric monoidal category on $\calD$ together with the structure of a strong symmetric monoidal functor on $U$;
\item A lift of $F$ to a monoid in oplax symmetric monoidal endofunctors of $\calC$;
\item A lift of $F^\vee$ to a comonoid in lax symmetric monoidal endofunctors of $\calC$.
\end{enumerate}
\end{lemma}
\begin{proof}
The equivalence of (2) and (3) is clear, so we consider their relation with (1). Given the data of (2), we can make $\calD = \Alg_F$ into a symmetric monoidal category, where for $F$-algebras $A$ and $B$, their tensor product is $A\otimes B$ with $F$-algebra structure $F(A\otimes B)\rightarrow F(A)\otimes F(B)\rightarrow A \otimes B$. This is seen to refine to the data of (1). Suppose then we have been given the data of (1). As $U$ is strong monoidal, there is for $A,B\in\calC$ a map $F(A\otimes B)\rightarrow F(A)\otimes F(B)$ adjoint to $A\otimes B\rightarrow UF(A)\otimes UF(B)\simeq U(F(A)\otimes F(B))$. This is seen to refine to the data of (2).
\end{proof}

\begin{rmk}
Let $\Gamma$ be a $\calP$-algebra. Then one can be more explicit about the structure necessary to upgrade $\Gamma$ to a $\calP$-cobialgebroid: to lift $\Gamma$ to an oplax symmetric monoidal functor, we require maps
\[
\Gamma(M\otimes N)\rightarrow\Gamma(M)\otimes\Gamma(N),\qquad\Gamma(\bbone)\rightarrow\bbone,
\]
natural in $M$ and $N$ and subject to the evident counity, coassociativity, and cocommutativity conditions, and for this to make $\Gamma$ into a $\calP$-cobialgebroid we further require that the product $\Gamma\circ\Gamma\rightarrow\Gamma$ respects this structure.
\tqed
\end{rmk}

\begin{rmk}\label{rmk:distributecobialgebroid}
Let $\Gamma$ be a $\calP$-cobialgebroid, and let $A$ be a monoid in the monoidal category $\Model_\Gamma^\heart$. In particular, $A$ overlies a monoid in $\Model_\calP^\heart$, giving $A\otimes\bs$ the structure of a monad. There is a distributive law of $\Gamma$ across $A\otimes\bs$ given by the composite
\[
\Gamma(A\otimes M)\rightarrow\Gamma(A)\otimes\Gamma(M)\rightarrow A\otimes \Gamma(M),
\]
and this rise to a composite monad $A\otimes\Gamma$. Algebras for this monad are exactly modules over the monoid $A$ in $\Model_\Gamma^\heart$. This generalizes \cref{ex:semitensor}.
\tqed
\end{rmk}

We will write $\Ring_\Gamma^\heart$ for the category of commutative monoids in $\LMod_\Gamma^\heart$. Observe that the forgetful functor $\Ring_\Gamma^\heart\rightarrow\CRing_\calP^\heart$ is plethystic.

\begin{ex}\label{ex:twistedbialgebra}
Let $R$ be an ordinary commutative ring, $\calR$ be the theory of $R$-modules with its usual symmetric monoidal structure, and $F$ be an $\calR$-cobialgebroid. Unwinding the definitions, we see this amounts to the following. First, as $F$ is a bimodule, we can write $F(M) = \Gamma\otimes M$ for an ordinary $R$-bimodule $\Gamma$. Abbreviate $\otimes = \otimes_R$, and use subscripts $l$ and $r$ to denote tensoring with respect to the left or right $R$-module structure on $\Gamma$. Then there are left $R$-module maps
\begin{align*}
m&\colon \Gamma\tins{r}{l}\Gamma\rightarrow\Gamma;\\
\epsilon^\times&\colon \Gamma\rightarrow R;\\
\Delta^\times&\colon \Gamma\rightarrow\Gamma\tins{l}{l}\Gamma.
\end{align*}
The map $m$ makes $\Gamma$ into an $R$-algebra, and $\epsilon^\times$ with $\Delta^\times$ satisfy the evident counity, coassociativity, and cocommutativity conditions. The map $\Delta^\times$ is required in addition to be a map of right $R$-modules with respect to the two right $R$-module structures on the target given by the action of $R$ on the left and right factor. This corresponds to the fact that $F(M\otimes N)\rightarrow F(M)\otimes F(N)$ is natural in both variables, and is what is necessary to extend $\Delta^\times$ to the natural transformation
\[
\Gamma\tins{r}{} M\otimes N\rightarrow (\Gamma\tins{r}{}M)\tins{l}{l}(\Gamma\tins{r}{}N),\qquad \gamma\otimes m \otimes n \mapsto \sum \gamma_{(1)}^\times\otimes m \otimes\gamma_{(2)}^\times \otimes n.
\]
 Compatibility of $m$ with $\epsilon^\times$ amounts to asking for 
\[
\epsilon^\times(1) = 1,\qquad \epsilon^\times(\gamma\gamma') = \epsilon^\times(\gamma\cdot\epsilon^\times(\gamma')),
\]
and compatibility of $m$ with $\Delta^\times$ amounts to asking for
\[
\sum \gamma_{(1)}^\times\gamma_{(1)}^{\prime\times}\otimes\gamma_{(2)}^\times\gamma_{(2)}^{\prime\times} = \sum (\gamma\gamma')_{(1)}^\times\otimes (\gamma\gamma')_{(2)}^\times
\]
in $\Gamma\tins{l}{l}\Gamma$. In the above, we have written $\Delta^\times(\gamma) = \sum \gamma_{(1)}^\times\otimes\gamma_{(2)}^\times$. We find that $\Gamma$ is a \textit{twisted $R$-bialgebra}, for instance as discussed in \cite[Section 9]{borgerwieland2005plethystic}.  A $\Gamma$-ring is a commutative $R$-ring $A$ equipped with an action of $\Gamma$ such that $\gamma\cdot (aa') = \sum (\gamma_{(1)}^\times\cdot a)(\gamma_{(2)}^\times\cdot a')$ for all $a,a'\in A$ and $\gamma\in\Gamma$.
\tqed
\end{ex}

\begin{ex}\label{ex:ltc}
Let $\Gamma$ be the $R$-algebra of \cref{ex:ltkoszul}. Then $\Gamma$ is an $R$-cobialgebroid, with augmentation
\[
\epsilon^\times(Q_0) = 1,\qquad \epsilon^\times(Q_1) = 0,\qquad\epsilon^\times(Q_2) = 0,
\]
and coproduct
\begin{align*}
\Delta^\times(Q_0) &= Q_0\otimes Q_0 + 2 Q_1 \otimes Q_2 + 2 Q_2 \otimes Q_1;\\
\Delta^\times(Q_1) &= Q_0\otimes Q_1+  Q_1\otimes Q_0 + a Q_1\otimes Q_2 + a Q_2 \otimes Q_1 + 2 Q_2 \otimes Q_2;\\
\Delta^\times(Q_2) &= Q_0\otimes Q_2 + Q_2\otimes Q_0 +  Q_1 \otimes Q_1 + a Q_2\otimes Q_2.
\end{align*}
Thus $\LMod_\Gamma$ is a symmetric monoidal category with underlying tensor product $\otimes_R$, where if $M$ and $N$ are left $\Gamma$-modules, then $M\otimes_R N$ has left $\Gamma$-module structure
\begin{align*}
Q_0(m\otimes n) &= Q_0(m)\otimes Q_0(n) + 2 Q_1(m)\otimes Q_2(n) + 2 Q_2(m)\otimes Q_1(n);\\
Q_1(m\otimes n) &= Q_0(m)\otimes Q_1(n)+Q_1(m)\otimes Q_0(n)\\
&+a Q_1(m)\otimes Q_2(n)+ a Q_2(m)\otimes Q_1(n) + 2 Q_2(m)\otimes Q_2(n);\\
Q_2(m\otimes n) &= Q_0(m)\otimes Q_2(n) + Q_2(m)\otimes Q_0(n)\\
&+Q_1(m)\otimes Q_1(n)+aQ_2(m)\otimes Q_2(n).
\end{align*}
We will continue this example in \cref{ex:ltep} and \cref{ex:ltring}.
\tqed
\end{ex}

\subsection{Additive operations}\label{ssec:additiveops}

Fix a symmetric monoidal additive theory $\calP$. The primary feature that distinguishes $\calP$-plethories from more ordinary algebraic structures such as $\calP$-algebras is the presence of nonlinear structure, and in dealing with $\calP$-plethories one wants to avoid dealing with this nonlinear structure by any means possible. The use of $\calP$-cobialgebroids is one thing that enables this. This generalizes the use of twisted bialgebras in \cite[Section 10]{borgerwieland2005plethystic}, as recalled in \cref{ex:twistedbialgebra}, with some main ingredients already appearing in \cite[Proposition 4.6]{tallwraith1968representable}.

There is a purely formal means by which one can extract from any $\calP$-plethory a $\calP$-cobialgebroid. First, going in the other direction, if $\Gamma$ is a $\calP$-cobialgebroid, then as $\Ring_\Gamma^\heart\rightarrow\CRing_\calP^\heart$ is plethystic, it is associated to some $\calP$-cobialgebroid $S\Gamma$. This defines a functor $S\colon\coBiAlg_\calP^\heart\rightarrow\Pleth_\calP^\heart$, which can be described more explicitly as follows: if $\Gamma$ is a $\calP$-cobialgebroid, then as $\Gamma^\vee$ is lax symmetric monoidal, it passes to a limit-preserving comonad $S\Gamma^\vee$ on $\Ring_\calP^\heart$, and this is right adjoint to the colimit-preserving monad $S\Gamma$ on $\Ring_\calP^\heart$. This functor $S$ preserves colimits, and therefore has a right adjoint sending a plethory $\Lambda$ to a $\calP$-cobialgebroid $\Gamma(\Lambda)$. We do not know if $\Gamma(\Lambda)$ admits a nice description in general, but it will under certain extra flatness conditions.

Now fix a $\calP$-plethory $\Lambda$. As in \cref{prop:yoneda}, we may identify
\[
\Lambda_{P_1\oplus\cdots\oplus P_n,P}\simeq\Hom_{\Fun(\Ring_\Lambda^\heart,\Set)}(\ev_{P_1}\times\cdots\times\ev_{P_n},\ev_P).
\]
Let $\Gamma_{(P_1,\ldots,P_n),P}\subset\Lambda_{P_1\oplus\cdots\oplus P_n,P}$ be the subset consisting of those operations which are $n$-multilinear. By allowing $P_1,\ldots,P_n,P$ to vary, this construction can be seen as a functor $\calP^{\times n}\rightarrow\LMod_\calP^\heart$.

\begin{lemma}
The $\calP$-module $\Gamma_{(P_1,\ldots,P_n)}$ is isomorphic to the total fiber of the $n$-cube obtained by tensoring together the maps
\[
\Delta^+ - \eta\otimes P_i-P_i\otimes\eta\colon \Lambda_{P_i}\rightarrow\Lambda_{P_i\oplus P_i}\simeq\Lambda_{P_i}\otimes\Lambda_{P_i}.
\]
\end{lemma}
\begin{proof}
For purely notational convenience, we consider the case $n=1$; the general case is identical in nature. Here we are claiming that there is an equalizer diagram
\begin{center}\begin{tikzcd}
\Gamma_P\ar[r]&\Lambda_P\ar[r,"\Delta^+",shift left=1mm]\ar[r,"\eta\otimes P + P \otimes \eta"',shift right=1mm]&\Lambda_{P\oplus P}
\end{tikzcd}.\end{center}
Evaluating on $P'\in\calP$, this is asking for an equalizer diagram
\begin{center}\begin{tikzcd}
\Hom_{\Fun(\Model_\Lambda^\heart,\Ab)}(\ev_P,\ev_{P'})\ar[d]\\
\Hom_{\Fun(\Model_\Lambda^\heart,\Set)}(\ev_P,\ev_{P'})\ar[d,"\Delta^+",shift left=2mm]\ar[d,"\eta\otimes P + P \otimes \eta"',shift right=2mm]\\
\Hom_{\Fun(\Model_\Lambda^\heart,\Set)}(\ev_P\times\ev_P,\ev_{P'})
\end{tikzcd}.\end{center}
If we fix a natural operation $\sigma\colon \ev_P\rightarrow\ev_{P'}$, then
\[
\Delta^+(\sigma)(x,y) = \sigma(x+y),\qquad (\eta\otimes P + P \otimes\eta)(x,y) = \sigma(x)+\sigma(y),
\]
and these agree precisely when $\sigma$ is additive.
\end{proof}

In particular, there are maps $\Gamma_{P_1}\otimes\cdots\otimes\Gamma_{P_n}\rightarrow\Gamma_{(P_1,\ldots,P_n)}$. Say that $\Lambda$ has \textit{well-behaved additive operations} if this is an isomorphism for all $P_1,\ldots,P_n\in\calP$, and moreover $\Gamma = \Gamma_{(\bs),(\bbs)}$ is a $\calP$-bimodule.

\begin{theorem}\label{thm:plethcobi}
Suppose that $\Lambda$ has well-behaved additive operations. Then the $\calP$-plethory structure of $\Lambda$ restricts to a $\calP$-cobalgebroid structure on $\Gamma$, and $\Gamma(\Lambda)\cong\Gamma$.
\end{theorem}
\begin{proof}
We begin by building the $\calP$-cobialgebroid structure on $\Gamma$. By hypothesis, $\Gamma$ is a $\calP$-bimodule. As additive operations are closed under composition, we have for $P,P',P''\in\calP$ a bilinear composition map $\Gamma_{P',P''}\times\Gamma_{P,P'}\rightarrow\Gamma_{P,P''}$, and these together with the identity maps make $\Gamma$ into a $\calP$-algebra. The counit of $\Gamma$ is the composite
\[
\epsilon^\times\colon\Gamma_{\bbone}\subset\Lambda_{\bbone}\rightarrow\bbone.
\]
For the coproduct on $\Gamma$, consider the diagram
\begin{center}\begin{tikzcd}
\Gamma_{P\otimes P'}\ar[r]\ar[d,dashed,"\Delta^\times"]&\Lambda_{P\otimes P'}\ar[d,"\Delta^\times"]\\
\Gamma_P\otimes\Gamma_{P'}\ar[r]&\Lambda_P\otimes\Lambda_{P'}
\end{tikzcd}.\end{center}
The right vertical map classifies multiplication for $\Lambda$-rings, so the clockwise composite lands in $\Gamma_{(P,P')}\subset\Lambda_{P\oplus P'}\cong\Lambda_P\otimes\Lambda_P'$. This is isomorphic to $\Gamma_P\otimes\Gamma_{P'}$ by assumption, so there is a lift through the dashed map. The axioms of counity, coassociativity, cocommutativity, and compatibility with composition follow from the corresponding facts about operations on $\Lambda$-rings. It remains to verify that $\Gamma\cong \Gamma(\Lambda)$. Observe first that the functor $S$ on $\calP$-cobialgebroids satisfies, as the notation suggests, $(S\Gamma)(P) = S(\Gamma(P))$. Indeed, there are isomorphisms
\[
\CRing_\calP((S\Gamma)(P),R)\cong (S\Gamma^\vee)(R)(P)\cong \Gamma^\vee(R)(P)\cong\LMod_\calP(\Gamma(P),R),
\]
so that $(S\Gamma)(P)$ has the necessary universal property. It follows from the definition of $\Gamma$ that the composite $\Gamma(\Lambda)\rightarrow S\Gamma(\Lambda)\rightarrow\Lambda$ factors uniquely through $\Gamma\subset\Lambda$, so $S\Gamma(\Lambda)\rightarrow\Lambda$ factors uniquely through $S\Gamma\rightarrow\Lambda$. So $S\Gamma\rightarrow\Lambda$ has the necessary universality property to be the counit of the adjunction, giving $S\Gamma\cong S\Gamma(\Lambda)$, from which it follows that $\Gamma\cong\Gamma(\Lambda)$.
\end{proof}

From here on, we will always assume that our plethories have well-behaved additive operations when we speak of their associated cobialgebroids, as this will be the case for the examples we are interested in. In fact, all of our examples will have very well-behaved additive operations, in the sense that $\Gamma\rightarrow\Lambda$ will be levelwise additively split. This ensures that the well-behaved nature of $\Lambda$ is preserved under various operations, such as base change.

\begin{rmk}
Let $\Lambda$ be a $\calP$-plethory and $B\in\Ring_\Lambda^\heart$. In particular, $B\in\CRing_\calP^\heart$, so there is a category $\LMod_B^\heart$ of left $B$-modules. The forgetful functor $B/\Ring_\Lambda^\heart\rightarrow B/\CRing_\calP^\heart$ is plethystic, so realizing $B/\Ring_\Lambda^\heart$ as the category of rings for a $B$-plethory $B\otimes\Lambda$. The monadic structure is as given by \cref{ex:slicedist}. There is a diagram
\begin{center}\begin{tikzcd}
\Gamma(B\otimes\Lambda)&B\ar[l]\\
\Gamma(\Lambda)\ar[u]&I\ar[u]\ar[l]
\end{tikzcd}\end{center}
of $\calP$-algebras, where $B$ stands for the $\calP$-algebra $B\otimes\bs$, which extends to a map
\[
B\otimes\Gamma(\Lambda)\rightarrow\Gamma(B\otimes\Lambda)
\]
of algebras, which is an isomorphism in the nice cases we will consider. Here, $B\otimes\Gamma(\Lambda)$ has algebra structure as indicated in \cref{rmk:distributecobialgebroid}. See \cref{ex:theta} for an explicit example of this.
\tqed
\end{rmk}

\subsection{Cotangent spaces}\label{ssec:plethaq}

Fix a symmetric monoidal additive theory $\calP$ and a $\calP$-plethory $\Lambda$. We are interested in the cohomology of $\Lambda$-rings. In particular, we need to identify the relevant categories of abelian group objects. This is as indicated by the general theory of \cref{ssec:abalg}, but it is worth making this more explicit.

We first review the essentially classical case where $\Lambda=S$ so that $\Ring_\Lambda^\heart\simeq\CRing_\calP^\heart$. Given $B\in\CRing_\calP^\heart$ and $M\in\LMod_B^\heart$, one may form the square-zero extension $B\ltimes M\in\Ab(B/\CRing_\calP^\heart/B)$. As an object of $\CRing_\calP^\heart = \CMon(\LMod_\calP^\heart)$, this is given by $B\ltimes M = B \oplus M$ with multiplication
\begin{align*}
(B\oplus M)\otimes (B\oplus M)&\simeq B\otimes B\oplus B\otimes M \oplus M \otimes B \oplus M\otimes M\\
&\simeq B \otimes B \oplus B \otimes M \oplus B\otimes M \oplus M\otimes M\rightarrow B\oplus M
\end{align*}
arising from the multiplication on $B$, the $B$-module structure of $M$, and killing $M\otimes M$. This has obvious structure as an object of $B/\CRing_\calP^\heart/B$, and structure as an abelian group object therein of
\[
(B\ltimes M)\times_B(B \ltimes M)\cong B\oplus M \oplus M\rightarrow B\oplus M,\qquad ((b,m'),(b,m''))\mapsto (b,m'+m'').
\]

\begin{lemma}\label{lem:classiccotangent}
The above construction describes an equivalence between the following categories:
\begin{enumerate}
\item The category $\LMod_B^\heart$;
\item The full subcategory of $B/\CRing_\calP^\heart/B$ spanned by the square-zero extensions of $B$;
\item The category $\Ab(\CRing_\calP^\heart/B)\simeq\Ab(B/\CRing_\calP^\heart/B)$.
\end{enumerate}
Moreover,
\begin{enumerate}
\item[(4)]Abelianization $D_B\colon B/\CRing_\calP^\heart\rightarrow\LMod_B^\heart$ is given by $D_B(A) = B\otimes_A \Omega_{A|\calP}$, where $\Omega_{A|\calP} = J/J^2$, where $J = \ker(A\otimes A\rightarrow A)$;
\item[(5)]Abelianization $Q_B\colon B/\CRing_\calP^\heart/B\rightarrow\LMod_B^\heart$ is given by $Q_B(A) = I/I^2$, where $I = \ker(A\rightarrow B) = \coker(B\rightarrow A)$.
\end{enumerate}
\end{lemma}
\begin{proof}
These statements are standard in the case where $\calP$ is the theory of commutative rings, and the same proofs carry over. That $M\mapsto B\ltimes M$ yields an equivalence from $\LMod_B^\heart$ to the category of square-zero extensions of $B$ is clear, with inverse sending a square-zero extension to its augmentation ideal. The categories $\Ab(\CRing_\calP^\heart/B)$ and $\Ab(B/\CRing_\calP^\heart/B)$ are equivalent as abelian groups are pointed. The category $\Ab(B/\CRing_\calP^\heart/B)$ may be identified as the category of square-zero extensions of $B$ as follows. Fix $A\in\Ab(B/\CRing_\calP^\heart/B)$. Additively we may split $A = B\oplus M$, and we must show that in fact $A\cong B\ltimes M$ multiplicatively. Because $A$ is a $B$-ring, it follows that $M$ is a $B$-module, and the multiplication on $A$ is necessarily determined by the multiplication on $B$, the $B$-module structure of $M$, and some map $M\otimes M\rightarrow M$ which we must show is zero. Fixing $P,P'\in\calP$, we must show that the map $m\colon M(P)\otimes M(P')\rightarrow M(P\otimes P')$ is zero. Write $\mu\colon A\times_B A\rightarrow A$ for the abelian group object multiplication on $A$; we will only use the fact that $\mu$ is a unital pairing. As $\mu$ is unital, the maps $\mu\colon A(P)\times_{B(P)}A(P)\rightarrow A(P)$ satisfy $\mu(x,0) = x = \mu(0,x)$ for $x\in M(P)$. Thus for $x\in M(P)$ and $x'\in M(P')$, we may compute $m(x\otimes x') = m(\mu(x,0)\otimes \mu(0,x')) = \mu(m(x,0)\otimes m(0,x')) = \mu(0,0) = 0$.

So we have shown the categories of (1)--(3) to be equivalent, and it remains only to verify the claims of (4) and (5). Write $D_B$ and $Q_B$ for the abelianization functors in question, and $D_B'$ and $Q_B'$ for their proposed descriptions. Then $D_B(A) = Q_B(B\otimes A)$ for $A\in\CRing_\calP^\heart/B$, and we claim first that also $D_B'(A) = Q_B'(B\otimes A)$. By definition, $Q_B(B\otimes A) = I/I^2$ where $I = \ker(B\otimes A\rightarrow B) = \ker(B\otimes_A (A\otimes A)\rightarrow B\otimes_A A)$. As $A\otimes A\rightarrow A$ admits an $A$-linear splitting, we can pull $B$ out to get $I = B\otimes_A J$ where $J = \ker(A\otimes A\rightarrow A)$. Thus $D_B'(A) = B\otimes_A (J/J^2) = I/I^2 = Q_B'(B\otimes A)$ as claimed. So it is sufficient to verify just that $Q_B = Q_B'$. Fix $A\in B/\CRing_\calP^\heart/B$, and write $A = B\oplus I$ additively, so that $Q_B'(A) = I/I^2$. Let $B\ltimes M$ be some square-zero extension of $B$. Then maps $A\rightarrow B\ltimes M$ in $B/\CRing_\calP^\heart/B$ are equivalent to maps $I\rightarrow M$ of nonunital $B$-rings. As $M$ has trivial multiplication, this factors uniquely through the quotient nonunital ring $I/I^2$, which has trivial multiplication. We find that the quotient ring $B\ltimes I/I^2$ of $A$ is the square-zero extension of $B$ associated to $Q_B(A)$, and thus $Q_B(A) = I/I^2 = Q_B'(A)$.
\end{proof}

We now turn to considering a general $\calP$-plethory $\Lambda$. Observe that $\bbone$, the unit of $\LMod_\calP^\heart$, is naturally a $\Lambda$-ring by the unique map $\bbone\rightarrow \Lambda^\vee(\bbone)$ of $\calP$-rings. Equivalently, $\bbone$ is a $\Lambda$-ring by the fact that $\Ring_\Lambda^\heart\rightarrow\CRing_\calP^\heart$ preserves the empty colimit. It follows that there is a good category of $\Ring_\Lambda^{\aug,\heart}=\Ring_\Lambda^\heart/\bbone$ of augmented $\Lambda$-rings. For all $P\in\calP$, the map $P\rightarrow 0$ of $\calP$-modules gives a map $\Lambda_P\rightarrow\Lambda_0 = \bbone$ of $\Lambda$-rings, so we can regard each $\Lambda_P$ as an object of $\Ring_\Lambda^{\aug,\heart}$. Define now
\[
\Delta(\Lambda)\colon \calP\rightarrow\LMod_\calP^\heart,\qquad \Delta(\Lambda)_{P,P'} = Q(\Lambda_P)_{P'}.
\]
We call $\Delta(\Lambda)$ the \textit{cotangent algebra} of $\Lambda$. The functor $\Delta(\Lambda)$ preserves coproducts, so can be regarded as a $\calP$-bimodule; it is a $\calP$-algebra by the following.

\begin{theorem}\label{thm:plethaq}
\hphantom{blank}
\begin{enumerate}
\item The category $\Ab(\Ring_\Lambda^{\aug,\heart})$ is equivalent to the full subcategory of $\Ring_\Lambda^{\aug,\heart}$ spanned by those $\Lambda$-rings whose underlying $\calP$-ring is a square-zero extension of $\bbone$. Moreover, the diagram
\begin{center}\begin{tikzcd}
\Ab(\Ring_\Lambda^{\aug,\heart})\ar[r]\ar[d]&\LMod_\calP^\heart\ar[d]\\
\Ring_\Lambda^{\aug,\heart}\ar[r]&\CRing_\calP^{\aug,\heart}
\end{tikzcd}\end{center}
is distributive.
\item The underlying $\calP$-bimodule of the $\calP$-algebra associated to the plethystic forgetful functor $\Ab(\Ring_\Lambda^{\aug,\heart})\rightarrow\LMod_\calP^\heart$ is given by $\Delta(\Lambda)$.
\item Fix $B\in\Ring_\Lambda^\heart$, so that $B/\Ring_\Lambda^\heart\simeq \Ring_{B\otimes\Lambda}^\heart$ as in \cref{ex:slicedist}, so that $B/\Ring_\Lambda^\heart/B \simeq \Ring_{B\otimes\Lambda}^{\aug,\heart}$. Then as a $\calP$-algebra, $\Delta_B(B\otimes\Lambda)\cong B\otimes\Delta(\Lambda)$ is a composition of the monad $B\otimes\bs$ with $\Delta(\Lambda)$. Moreover, the diagram
\begin{center}\begin{tikzcd}
\Ab(\Ring_\Lambda^\heart/B)\ar[r]\ar[d]&\Ab(\CRing_\calP^\heart/B)\ar[d]\\
\Ring_\Lambda^\heart/B\ar[r]&\CRing_\calP^\heart/B
\end{tikzcd}\end{center}
is distributive.
\end{enumerate}
\end{theorem}
\begin{proof}
Given \cref{lem:classiccotangent}, these are just specializations of the general theory of \cref{ssec:abalg}.
\end{proof}

\begin{rmk}
The composite
$
\Gamma(\Lambda)\rightarrow\Lambda\rightarrow\Delta(\Lambda)
$
is a map of $\calP$-algebras.
\tqed
\end{rmk}

\begin{rmk}
The constructions of $\Gamma(\Lambda)$ and $\Delta(\Lambda)$ are formally dual: if we additively split $\Lambda_P = \widetilde{\Lambda}_P\oplus\bbone$, then
\begin{align*}
\Delta(\Lambda)_P &= \coker(\widetilde{\Lambda}_{P\oplus P}\rightarrow\widetilde{\Lambda}_P)\\
\Gamma(\Lambda)_P &= \ker(\widetilde{\Lambda}_P\rightarrow\widetilde{\Lambda}_{P\oplus P}),
\end{align*}
the maps being obtained from the codiagonal $P\oplus P\rightarrow P$ and diagonal $P\rightarrow P\oplus P$ respectively. In other words, $\Delta(\Lambda)$ is the linearization of the functor $\Lambda$, and dually we might call $\Gamma(\Lambda)$ the colinearization of $\Lambda$.
\tqed
\end{rmk}

\begin{ex}\label{ex:theta}
Let $\Lambda$ be the $\bbZ$-plethory of $\theta$-rings, also known as $\delta$-rings \cite{joyal1985delta}, as well as by other names. In brief, $\theta$-rings are commutative rings $B$ equipped with an operation $\theta\colon B\rightarrow B$ satisfying all the identities necessary to make
\[
\psi(b) = b^p + p \theta(b)
\]
generically a ring map. The underlying commutative ring of $\Lambda$ can be identified as
\[
\Lambda = \bbZ[\theta_n : n\geq 0],
\]
where $\theta_n = \theta^{\circ n}$. Here, we are making use of the correspondence between elements of the ring $\Lambda$ and natural operations on $\theta$-rings, so for example the operation $\psi$ is given by the element
\[
\psi = \theta_0^p + p \theta_1.
\]
The operation $\psi$ freely generates the additive operations, and we can identify $\Gamma(\Lambda) = \bbZ[\psi]$ as a $\bbZ$-algebra. As $\psi$ is a ring homomorphism, the cobialgebroid structure is
\[
\epsilon^\times(\psi) = 1,\qquad \Delta^\times(\psi) = \psi\otimes\psi.
\]
Evidently $\Delta(\Lambda) = \bbZ[\theta]$ as a $\bbZ$-algebra, and the map $\Gamma(\Lambda)\rightarrow\Delta(\Lambda)$ is given by
\[
\bbZ[\psi]\rightarrow\bbZ[\theta],\qquad \psi\mapsto p \theta.
\]

Now say $B$ is a $\theta$-ring, so that $B\otimes\Lambda$ is a $B$-plethory. The general recipe of \cref{ex:slicedist} for computing the plethory structure on $B\otimes\Lambda$ translates into the following. Note $B\otimes\Lambda = B[\theta_n : n \geq 0]$, and it is sufficient to determine the composition $\theta_1 \circ b$ for $b\in B$. As an element of $B\otimes\Lambda$, the element $\theta_1\circ b$ represents the natural operation
\[
(\theta_1\circ b)(a) = \theta(b\cdot a) = \theta(b) a^p + b^p \theta(a) + p \theta(b)\theta(a),
\]
defined for $\theta$-rings under $B$. Thus we have
\[
\theta_1\circ b = \theta_1(b)\theta_0^p + b^p \theta_1 + p \theta(b) \theta_1.
\]
This is nothing but a specialization of the general formula
\[
(b\otimes\sigma)\circ(b'\otimes\sigma') = \sum b \sigma_{(1)}^\times(b')\otimes \sigma_{(2)}^\times\circ \sigma'.
\]

We can identify $\Gamma(B\otimes\Lambda) = B\otimes\Gamma(\Lambda) = B[\psi]$ as $\bbZ$-bimodules. For clarity, write $\psi_1 = \psi$ as an element of $B\otimes\Gamma(\Lambda)$. The algebra structure is determined by the general distributive law of \cref{rmk:distributecobialgebroid}, which is in turn described by \cref{ex:semitensor}, and thus
\[
\Gamma(B\otimes\Lambda) = B\langle \psi_1\rangle/(\psi_1\cdot b = \psi(b)\cdot\psi_1)
\]
as a $B$-algebra. Similarly $\Delta(B\otimes\Lambda) = B\otimes\Delta(\Lambda) = B[\theta]$ as $\bbZ$-bimodules, but as $\Delta(\Lambda)$ does not carry a coproduct, the algebra structure cannot be determined in the same way. There are two good ways to proceed. The method that works in general is to observe that the composition law on $B\otimes\Lambda$ implies that $\theta\circ b \equiv \psi(b)\theta$ mod indecomposables, and thus
\[
B\otimes\Delta(\Lambda) = B\langle\theta\rangle / (\theta\cdot b = \psi(b)\cdot\theta).
\]
The second method is to observe that the existence of an algebra map
\[
B\langle \psi_1\rangle/(\psi_1\cdot b = \psi(b)\cdot \psi_1)\rightarrow B\otimes\bbZ[\theta],\qquad \psi_1\mapsto p \theta
\]
determines the algebra structure on $B\otimes \bbZ[\theta]$, at least when $B$ is $p$-torsion free.
\tqed
\end{ex}

Call $\Lambda$ \textit{smooth} when its associated $S\calP$-algebra is smooth in the sense of \cref{ssec:abalg}; that is, if $\Lambda(P)\in\CRing_\calP$ is smooth for all $P\in\calP$. When $\Lambda$ is smooth, \cref{prop:smoothab} and \cref{thm:grothss} allow one to split computations of the Quillen cohomology of $\Lambda$-rings into computations of the Quillen cohomology of $\calP$-rings together with computations of $\Ext$ over the additive $\calP$-algebra $\Delta(\Lambda)$.

It is convenient to have a relative version of this. Given a map $\Lambda'\rightarrow\Lambda$ of $\calP$-plethories, $\Lambda$ may be viewed as an algebra for the theory of $\Lambda'$-rings, and we say that $\Lambda$ \textit{smooth relative to $\Lambda'$} when it is smooth in this sense.

\begin{ex}\label{ex:alternating}
Let $R$ be an ordinary commutative ring, and $\calP$ the theory of $\bbZ$-graded $R$-modules, regarded as a symmetric monoidal theory with symmetrizer employing the Koszul sign rule. Then the category $\CRing_{R_\ast}^\heart\simeq\CMon(\Mod_{R_\ast}^\heart)$ may be identified as the category of ordinary $\bbZ$-graded $R$-algebras $B$ such that
\[
bb' = (-1)^{|b||b'|}b'b
\]
for all $b,b'\in B$. In particular, if we write $R\{e_n\}$ for a copy of $R$ in degree $n$, then
\[
SR\{e_n\} = R[e_n]/((1-(-1)^n)e_n^2).
\]
When $2$ is neither zero nor a unit in $R$, the monad $S$ need not preserve projective objects, and the homotopy theory of simplicial $R_\ast$-rings may not behave as one would like. One fix is to instead work with \textit{alternating $R_\ast$-algebras}, i.e.\ those $B\in\CRing_{R_\ast}^\heart$ such that $b^2 = 0$ when $|b|$ is odd. Denote this category by $\Ring_{R_\ast}^\heart$. Then the inclusion $\Ring_{R_\ast}^\heart\rightarrow\CRing_{R_\ast}^\heart$ is plethystic, and moreover $\Ab(\Ring_{R_\ast}^\heart/B) \simeq \Mod_{B_\ast}^\heart\simeq\Ab(\CRing_{R_\ast}^\heart/B)$, so everything we have done for $R_\ast$-plethories carries over verbatim to the relative setting over $\Ring_{R_\ast}^\heart$.
\tqed
\end{ex}

\subsection{Suspension maps}\label{ssec:susps}

We record here a definition of an additional piece of structure present in plethories that encode homotopy operations. Fix a symmetric monoidal additive theory $\calP$ as before. Let $E$ be an object of $\calP$ which is invertible under the tensor product. Then there is an automorphism of the category of endofunctors of $\LMod_\calP$ given by
\[
H\mapsto H^E,\qquad H^E(M) = E\otimes H(E^{-1}\otimes M).
\]
This is compatible with compositions of endofunctors, and so preserves monads and comonads. In addition, it preserves bimodules, and $(H^E)^\vee = (H^{\vee})^{ E^{-1}}$.

\begin{defn}\label{def:suspensions}
\hphantom{blank}
\begin{enumerate}
\item If $F$ is a $\calP$-algebra, we say that $F$ is equipped with \textit{$E$-suspensions} if we have chosen a map $\sigma\colon F^E\rightarrow F$ of algebras.
\item If $\Gamma$ is a $\calP$-bialgebroid, we say that $\Gamma$ is equipped with \textit{$E$-suspensions} if we have equipped the underlying algebra of $\Gamma$ with $E$-suspensions in such a way that for all $M\in\LMod_\calP^\heart$, the diagram
\begin{center}\begin{tikzcd}
\Gamma^E(M)\ar[rr]\ar[d,"="]&&\Gamma(M)\\
E\otimes\Gamma(E^{-1}\otimes M)\ar[r,"E\otimes\Delta^\times"]&E\otimes \Gamma(E^{-1})\otimes\Gamma(M)\ar[r,"\sigma\otimes\Gamma(M)"]&\Gamma(\bbone)\otimes\Gamma(M)\ar[u,"\epsilon^\times\otimes\Gamma(M)"']
\end{tikzcd}\end{center}
commutes.
\item If $\Lambda$ is a $\calP$-plethory, we say that $\Lambda$ is equipped with \textit{$E$-suspensions} if we have chosen a map $\sigma\colon\Delta(\Lambda)^E\rightarrow\Gamma(\Lambda)$ of algebras such that the composite $\Gamma(\Lambda)^E\rightarrow\Delta(\Lambda)^E\rightarrow\Gamma(\Lambda)$ equips $\Gamma(\Lambda)$ with $E$-suspensions.
\tqed
\end{enumerate}
\end{defn}

This definition is not intended to cover all cases where one may wish to speak of suspension maps, but only those needed for the examples we will encounter. In all of our explicit examples, $\LMod_\calP$ will be a category of $\bbZ$-graded objects. Here we will always take $E$ to be a copy of the monoidal unit in degree $1$, and will just say ``equipped with suspensions'', as we trust no confusion should arise.

\begin{ex}
Let $k$ be an ordinary commutative ring, and consider the theory of $\bbZ$-graded left $k$-modules. Let $E$ denote a copy of $k$ in degree $1$. If $B$ is an ordinary $\bbZ$-graded $k$-algebra, then underlying monad of $B$ is equipped with $E$-suspensions given by the identifications
\[
E\otimes B \otimes E^{-1}\otimes M = B\otimes M.
\]
As remarked in \cref{ex:cohord}, when $B$ is augmented, this is the sort of structure needed to define $H^\ast(B)$ as an ordinary $\bbZ$-graded algebra.
\tqed
\end{ex}

Given an object $T$ equipped with a suspension map $\sigma\colon T^E\rightarrow T$, we can define the \textit{costabilization} of $T$ to be $\lim_{n\rightarrow\infty} T^{E^{n}}$. This is a monad when $\sigma$ is a map of monads.

\begin{ex}
Let $\calU$ be the algebra for the theory $\bbZ$-graded $\bbF_2$-modules whose modules are the unstable modules over the mod $2$ Steenrod algebra. Then $\calU$ is naturally equipped with suspensions $\sigma\colon \calU^E\rightarrow\calU$ given by $\sigma(\Sq^I) = \Sq^I$, with the understanding that this element may be zero in the target even when nonzero in the source. The costabilization of $\calU$ is isomorphic to the mod $2$ Steenrod algebra.
\tqed
\end{ex}

\section{\texorpdfstring{$\bbE_\infty$}{E-infinity} rings over \texorpdfstring{$\bbF_p$}{Fp}}\label{sec:p}

In this section we describe what the content of the preceding sections looks like in the context of power operations for $\bbE_\infty$ algebras over $\bbF_p$.

Throughout this section, we will write $e_a$ for a generic $\bbZ$-graded module generated by an element in degree $a$. 

\subsection{Plethories of power operations}\label{ssec:ppleth}

To illustrate the relevant ideas, we show how one can identify that a plethory of mod $p$ power operations exists, even before the hard work of computing its structure has been carried out. Our work is simplified by the fact that $\bbF_p$ is a field, but the approach taken here generalizes to other contexts.

Let $\Mod_{\bbF_p}$ denote the category of $H\bbF_p$-module spectra; we will just call these $\bbF_p$-modules. This is a symmetric monoidal category under $\otimes = \otimes_{\bbF_p}$. Let $\bbP$ denote the free $\bbE_\infty$ algebra monad on $\Mod_{\bbF_p}$,
\[
\bbP V = \bigoplus_{n\geq 0}\bbP_n V,\qquad \bbP_n V = V^{\otimes n}_{h\Sigma_n},
\]
and $\CAlg_{\bbF_p}$ the resulting category of $\bbE_\infty$ algebras over $\bbF_p$. Let $\CAlg_{\bbF_p}^\free\subset\CAlg_{\bbF_p}$ denote the essential image of $\bbP$. Both $\Mod_{\bbF_p}$ and $\CAlg_{\bbF_p}^\free$ are theories, and thus so are their homotopy categories $\h\Mod_{\bbF_p}$ and $\h\CAlg_{\bbF_p}^\free$. 

The theory $\h\Mod_{\bbF_p}$ is easily identified:
\[
\h\Mod_{\bbF_p}\simeq\Mod_{\bbF_{p\ast}}^\heart,\qquad \LMod_{\h\Mod_{\bbF_p}}\simeq\Mod_{\bbF_{p\ast}}.
\]
Here $\Mod_{\bbF_{p\ast}}$ is the (derived) category of $\bbZ$-graded $\bbF_p$-modules; we will just call these $\bbF_{p\ast}$-modules. Moreover, the symmetric monoidal structure on $\h\Mod_{\bbF_{p}}$ obtained from that on $\Mod_{\bbF_p}$ is exactly the standard symmetric monoidal structure on $\Mod_{\bbF_{p\ast}}^\heart$ with symmetrizer obeying the Koszul sign rule. 

By contrast, the theory $\h\CAlg_{\bbF_p}^\free$ is more complicated. Abstractly, one can say that it is exactly the theory of operations acting on the homotopy groups of $\bbE_\infty$ algebras over $\bbF_p$; for example, following \cref{prop:yoneda}, there are isomorphisms
\[
\Hom_{\Fun(\CAlg_{\bbF_p},\Set)}(\pi_i,\pi_j) \cong\pi_j\bbP\Sigma^i \bbF_p\cong\Hom_{h\CAlg_{\bbF_p}^\free}(\bbP\Sigma^j\bbF_p,\bbP\Sigma^i\bbF_p).
\]

By construction, $\bbP$ induces a map $\h\Mod_{\bbF_p}\rightarrow\h\CAlg_{\bbF_p}^\free$ of theories, and restriction along this makes the category of $\h\CAlg_{\bbF_p}^\free$-models strongly monadic over $\Mod_{\bbF_{p\ast}}$, i.e.\ we may view $\h\CAlg_{\bbF_p}^\free$-models as $\bbF_{p\ast}$-modules equipped with some extra structure. Write $\DL$ for the associated monad on $\Mod_{\bbF_{p\ast}}$. The following is then a consequence of basic properties of $\bbP$.

\begin{prop}
The natural isomorphisms $\bbP(U\oplus V)\simeq\bbP U\otimes\bbP V$ give natural isomorphisms $\pi_\ast \bbP(U\oplus V)\simeq \pi_\ast\bbP U\otimes\pi_\ast \bbP V$ which equip $\DL$ with the structure of an exponential monad on $\Mod_{\bbF_{p\ast}}^\heart$, and thus $\DL$ is a $\bbF_{p\ast}$-plethory. Moreover, the natural maps $\Sigma \bbP_n V\rightarrow \bbP_n \Sigma V$ defined for $n\geq 1$ equip $\DL$ with suspensions.
\qed
\end{prop}

\begin{rmk}
As all $\bbF_p$-modules are free, $\Ring_\DL^\heart$ may also be identified as the category of $\bbH_\infty$ algebras over $\bbF_p$.
\tqed
\end{rmk}

This is as far as purely formal considerations can take us; to get further one needs real knowledge of the structure of mod $p$ power operations.

\subsection{Dyer-Lashof operations}\label{ssec:dl}

Our goal now is to recall the structure of mod $p$ power operations, and package it into the plethystic framework. We find it most convenient to proceed by introducing some of the relevant algebra first. We begin by recalling a certain algebra $\calB$ of power operations; this algebra has various names in the literature, such as the big, or generalized, Steenrod algebra, or the Kudo-Araki-May algebra. It is essentially the algebra of all generalized Steenrod operations in the sense of \cite{may1970general}, only with different naming conventions.

We follow the convention that the binomial coefficient $\binom{n}{m}$ vanishes unless $0\leq m \leq n$.

\begin{defn}[$p=2$]
$\calB$ is the ordinary $\bbZ$-graded associative $\bbF_2$-algebra generated by symbols $Q^s$ of degree $s$ for all $s\in \bbZ$, and subject to the relations
\[
Q^{2s+r+1}Q^s = \sum_{0\leq i < \frac{r}{2}}\binom{r-i-1}{i}Q^{2s+i+1}Q^{r+s-i}
\]
for $r\geq 0$. Here, the bounds of summation are not necessary, but indicate when the binomial coefficients may be nonzero. Given a sequence $I = (r_1,\ldots,r_k)$, write $Q^I = Q^{r_1}\cdots Q^{r_k}$, and call $I$ and $Q^I$ \textit{admissible} if $r_i\leq 2r_{i+1}$ for each $i$. Define the \textit{excess} of $I$ by $e(I) = r_1-r_2-\cdots - r_k$. Given an integer $u$, call a $\calB$-module $M$ \textit{$u$-unstable} if $Q^r m = 0$ for any $m\in M$ with $r<|m|+u$; when $u=0$ we omit it from the name and notation.
\tqed
\end{defn}

\begin{defn}[$p>2$]
$\calB$ is the ordinary $\bbZ$-graded associative $\bbF_p$-algebra generated by symbols $Q^s_\epsilon$ of degree $2s(p-1)-\epsilon$ for $\epsilon\in\{0,1\}$ and $s\in\bbZ$, and subject to the relations

\begin{align*}
Q^{ps+r+1}Q^s &= \sum_{0\leq i < \frac{p-1}{p}s}(-1)^{i+1}\binom{(p-1)(r-i)-1}{i}Q^{ps+i+1}Q^{r+s-i}\\
Q^{ps+r}Q^s_1 &= \sum_{0\leq i \leq \frac{p-1}{p}r}(-1)^i\binom{(p-1)(r-i)}{i}Q_1^{ps+i}Q^{r+s-i}\\
&+\sum_{0\leq i < \frac{p-1}{p}r}(-1)^{i+1}\binom{(p-1)(r-i)-1}{i} Q^{ps+i}Q^{r+s-i}_1\\
Q^{ps+r+1}_1Q^s &= \sum_{0\leq i < \frac{p-1}{p}r} (-1)^{i+1}\binom{(p-1)(r-i)-1}{i} Q^{ps+i+1}_1Q^{r+s-i}\\
Q^{ps+r}_1Q^s_1 &= \sum_{0\leq i < \frac{p-1}{p}r} (-1)^{i+1}\binom{(p-1)(r-i)-1}{i}Q^{ps+i}_1Q^{r+s-i}_1
\end{align*}
for $r\geq 0$, where we have abbreviated $Q^s = Q^s_0$. Here, the bounds of summation are not necessary, but indicate when the binomial coefficients may be nonzero. Given a sequence $I = (\epsilon_1,r_1,\ldots,\epsilon_k,r_k)$ in $\{0,1\}\times\bbZ$, write $Q^I = Q^{r_1}_{\epsilon_1}\cdots Q^{r_k}_{\epsilon_k}$, and call $I$ and $Q^I$ \textit{admissible} if $r_i\leq pr_{i+1}-\epsilon_{i+1}$ for each $i$. Define the \textit{length} of $I$ to be $k$ and the \textit{excess} of $I$ to be $e(I) = 2r_1-\epsilon_1-(2r_2(p-1)-\epsilon_2)-\cdots-(2r_k(p-1)-\epsilon_k)$. Given an integer $u$, call a $\calB$-module $M$ \textit{$u$-unstable} if $Q^r_\epsilon m = 0$ for any $m\in M$ with $2r-\epsilon < |m|+u$; when $u=0$, we omit it from the name and notation.
\tqed
\end{defn}

For any integer $u$, write $F^u$ for the free $u$-unstable $\calB$-module functor. Then $F^u$ is a quotient algebra of $\calB$, where here we refer to the general notion of algebra over a theory, as well as of $F^{u'}$ for $u'<u$. Write $E = e_1$, and for $M\in\Mod_{\bbF_{p\ast}}$, write $sM=E\otimes M$. If $M$ is an $F^u$-module, then $sM$ is an $F^{u-1}$-module, and this provides an isomorphism $(F^u)^E\cong F^{u-1}$; together with the quotient maps $F^{u-1}\rightarrow F^u$, this equips each algebra $F^u$ with suspensions, and for the most part reduces us to considering just $F = F^0$.

\begin{lemma}
\hphantom{blank}
\begin{enumerate}
\item $\calB$ has a basis consisting of $Q^I$ for all admissible sequences $I$;
\item $F(e_n)$ has a basis consisting of $Q^I e_n$ with $I$ admissible of excess at least $n$.
\end{enumerate}
\end{lemma}
\begin{proof}
This is \cite[Propositions 11.2 and 12.2]{mandell1998einfinity}; see also \cite[Lectures 6-7]{lurie2007sullivan} for a detailed algebraic proof when $p=2$ which, provided one assumes the analogous fact for unstable modules over the Steenrod algebra \cite[Proposition 1.6.2]{schwartz1994unstable}, generalizes to $p>2$.
\end{proof}

\begin{rmk}
The abelian category $\LMod_F^\heart$ has enough projectives, given by the free $F$-modules. In addition, the right adjoint $F^\vee$ supplies it with enough injectives. By definition,
\[
F^\vee(e_a)_b = \Mod_{\bbF_p}(F(e_b),e_a),
\]
so the injective modules $F^\vee(e_a)$ can be seen as analogues of the Brown-Gitler modules seen in the study of unstable modules over the Steenrod algebra.
\tqed
\end{rmk}

\begin{defn}
A \textit{$\DL$-ring} is a graded commutative $\bbF_{p\ast}$-ring equipped with an $F$-module structure such that
\begin{enumerate}
\item $Q^0 (1) = 1$, and otherwise $Q^r_\epsilon (1) = 0$;
\item 
\begin{enumerate}
\item If $p=2$, then $Q^r x = x^2$ when $r=|x|$,
\item If $p>2$, then $Q^r x = x^p$ when $2r=|x|$;
\end{enumerate}
\item
\begin{enumerate}
\item If $p\geq 2$, then $Q^r(xy) = \sum_{i+j=r}Q^i(x)Q^j(y)$,
\item If $p>2$, then $Q^r_1(xy) = \sum_{i+j=r}(Q^i_1(x)Q^j(y)+(-1)^{|x|}Q^i(x)Q^j_1(y))$.
\tqed
\end{enumerate}
\end{enumerate}
\end{defn}

From the definition, we see that $\DL$-rings are the models of a finite product theory living over $\Mod_{\bbF_{p\ast}}^\heart$; write the associated monad as $\DL$. Let $\DL_n$ denote the free $\bbF_{p\ast}$-ring on symbols $Q^I e_n$ where $I$ is an admissible sequence satisfying $e(I)>n$, graded so that $|Q^I e_n| = |Q^I|+n$. The action of $\calB$ on the canonical element of $\DL(e_n)_n$ gives a map $\DL_n\rightarrow\DL(e_n)$ of $\bbF_{p\ast}$-rings.

\begin{theorem}[{\cite[III.1.1, IX.2.1]{brunermaymccluresteinberger1986hinfinity}}]\label{thm:pops}
The structure of mod $p$ power operations can be summarized as follows.
\begin{enumerate}
\item The homotopy groups of any object of $\CAlg_{\bbF_p}$ naturally form a $\DL$-ring, and the resulting maps $\DL_n\rightarrow\DL(e_n)\rightarrow\pi_\ast\bbP\Sigma^n\bbF_p$ are all isomorphisms;
\item In particular, $\DL$ agrees with the $\bbF_{p\ast}$-plethory of \cref{ssec:ppleth}, and is a smooth $\bbF_{p\ast}$-plethory with suspensions;
\item There are isomorphisms $\Gamma(\DL)\cong F$ and $\Delta(\DL)\cong F^1$, and the suspension map $\sigma\colon \Delta(\DL)^E\rightarrow\Gamma(\DL)$ is given by the isomorphism $(F^1)^E\cong F$.
\qed
\end{enumerate}
\end{theorem}

Throughout the rest of this section, we will abbreviate $\Delta = \Delta(\DL) = F^1$.

\begin{ex}
The costabilization $\lim_{n\rightarrow\infty}\Gamma(\DL)^{E^n}$ of $\DL$ can be identified as the completion of $\calB$ with respect to excess. Here, $E^n = E^{\otimes n} = e_n$. This object arises naturally when considering ``stable power operations'', as realized by the endomorphism spectrum of the forgetful functor $U\colon\CAlg_{\bbF_p}\rightarrow\Sp$; see \cite[Lecture 24]{lurie2007sullivan} and \cite[Section 10]{glasmanlawson2020stable}.
\tqed
\end{ex}

\subsection{Unstable \texorpdfstring{$\calA$}{A}-modules}\label{ssec:unstablea}

Let $\calA$ denote the mod $p$ Steenrod algebra. As observed by Mandell \cite[Theorem 1.4]{mandell1998einfinity}, there is a quotient map
\[
\calB\rightarrow\calB/(Q^0 = 1)\cong\calA,\qquad
\begin{cases}
Q^s\mapsto \Sq^{-s},&\text{when }p=2;\\
Q^s_\epsilon\mapsto \beta^\epsilon P^{-s},&\text{when }p>2.
\end{cases}
\]
\begin{defn}
An \textit{unstable $\calA$-module}, resp., \textit{unstable $\calA$-ring}, is an $F$-module, resp., $\DL$-ring, whose underlying $\calB$-module structure factors through the quotient map $\calB\rightarrow\calA$.
\tqed
\end{defn}

Unstable $\calA$-modules and unstable $\calA$-rings (more commonly called unstable $\calA$-algebras) have been the study of much rich study; see \cite{schwartz1994unstable} for a textbook account, and \cite{lurie2007sullivan} for an account that treats the relation with $\calB$. Essentially everything we do with $F$-modules and $\DL$-rings has an analogue for unstable $\calA$-modules and unstable $\calA$-rings.

Write $\calU$ for the $\bbF_{p\ast}$-algebra such that $\LMod_\calU^\heart$ is the category of unstable $\calA$-modules (itself often written $\calU$), as in \cref{ex:unstablesteenrod}. Write $\Ring_\calU^\heart$ for the category of unstable $\calA$-rings. By definition, unstable $\calA$-rings embed fully faithfully into $\DL$-rings. This is no longer the case at the level of simplicial rings, i.e.\ the functor $\Ring_\calU\rightarrow\Ring_\DL$ is no longer fully faithful. The situation here is exactly the same as that appearing in \cite{mandell1998einfinity}, and can be dealt with the same way.

Write $T\colon\Mod_{\bbF_{p\ast}}\rightarrow\Ring_\calU$ for the free unstable $\calA$-ring functor.

\begin{lemma}\label{lem:ualgrep}
For all $n\in\bbZ$, there is a (homotopy) pushout square
\begin{center}\begin{tikzcd}
\DL(e_n)\ar[d]\ar[r,"\phi"]&\DL(e_n)\ar[d]\\
\bbF_p\ar[r]&T(e_n)
\end{tikzcd}\end{center}
in $\Ring_{\DL}$, where $\phi$ classifies the element $e_n - Q^0 e_n$.
\end{lemma}
\begin{proof}
See \cite[Section 12]{mandell1998einfinity}.
\end{proof}

If $R$ is a discrete commutative $\bbF_p$-ring, then $R$ can be viewed as an $\bbE_\infty$ algebra over $\bbF_p$. The resulting $\DL$-ring structure on $R=\pi_\ast R$ is forced by the axioms to satisfy $Q^0 x = x^p$ and otherwise $Q^r_\epsilon x = 0$.

Call a field $\kappa$ of characteristic $p$ \textit{Artin-Schreier closed} if the map $\lambda\mapsto \lambda - \lambda^p$ is surjective on $\kappa$. In particular, this holds if $\kappa$ is algebraically closed. 

\begin{prop}\label{thm:ualgemb}
Let $\kappa$ be an Artin-Schreier closed field. Then the composite
\[
\Ring_\calU\rightarrow\Ring_{\DL}\rightarrow\Ring_{\kappa\otimes\DL}
\]
is fully faithful.
\end{prop}
\begin{proof}
We first verify this on discrete objects. Fix $R,S\in\Ring_\calU^\heart$. Then $\Hom_\calU(R,S) = \Hom_\DL(R,S)$, and we must show this is isomorphic to $\Hom_{\kappa\otimes\DL}(\kappa\otimes R,\kappa\otimes S)$. The latter is isomorphic to $\Hom_{\DL}(R,\kappa\otimes S)$ by adjunction, so we must show that every map $f\colon R\rightarrow\kappa\otimes S$ of $\DL$-rings factors through $\bbF_p\otimes S\subset\kappa\otimes S$. As $Q^0$ acts by the identity on $R$, every such map lands in the fixed points of $Q^0$ on $\kappa\otimes S$. As $Q^0$ acts on $\kappa\otimes S$ by $Q^0(\lambda\otimes s) = \lambda^p\otimes s$, the set of fixed points of $Q^0$ on $\kappa\otimes S$ is exactly $\bbF_p\otimes S$, proving the claim.

Now fix $R,S\in\Ring_\calU$ which are not necessarily discrete, and consider the map
\[
\Map_\calU(R,S)\rightarrow\Map_\DL(R,\kappa\otimes S)
\]
which we are claiming is an isomorphism. As $\Ring_\calU\rightarrow\Ring_\DL$ is stable under colimits, we may by resolving $R$ reduce to the case where $R = T(e_n)$ is a free unstable $\calA$-ring, and so reduce to verifying that the map
\[
S_n = \Map_\calU(T(e_n),S) \rightarrow\Map_\DL(T(e_n),\kappa\otimes S)
\]
is an isomorphism. By \cref{lem:ualgrep}, there is a fiber sequence
\[
\Map_\DL(T(e_n),\kappa\otimes S)\rightarrow\kappa\otimes S_n\rightarrow\kappa\otimes S_n,
\]
where the second map is given on homotopy groups by $\lambda\otimes s\mapsto (\lambda - \lambda^p)\otimes s$. The claim follows from the long exact sequence in homotopy groups.
\end{proof}

\subsection{Cohomology of \texorpdfstring{$\DL$}{DL}-rings}\label{ssec:dlringcoh}

Because $\DL$ is a smooth $\bbF_{p\ast}$-plethory, the general aspects of the cohomology of $\DL$-rings is as described in \cref{ssec:plethaq}. Explicitly, fix $R\in\Ring_\DL^\heart$, $B\in\Ring_{R\otimes\DL}^\heart$, $A\in\Ring_{R\otimes\DL/B}^\heart$, and $M\in\Ab(\Ring_{R\otimes\DL/B}^\heart)\simeq\LMod_{B\otimes\Delta}^\heart$. Here, if $p\geq 2$, then $B\otimes\Delta$ has multiplication satisfying
\[
(b\otimes Q^r)\cdot (b'\otimes Q^{r'}) = \sum_{i+j=r} b\, Q^i(b')\otimes Q^j Q^{r'},
\]
and if $p>2$, satisfying
\[
(b\otimes Q^r_1)\cdot (b'\otimes Q^{r'}) = \sum_{i+j=r}\left(b\, Q_1^i(b')\otimes Q^j Q^{r'} + (-1)^{|b'|}b\, Q^i(b')\otimes Q^j_1 Q^{r'}\right);
\]
these follow from the recipe of \cref{thm:plethaq}. By smoothness, $\bbL\Omega_{A|R}$ upgrades to an $A\otimes\Delta$-module, and
\[
\calH_{R\otimes\DL/B}(A;M) \simeq\EXT_{B\otimes\Delta}(B\otimes_A^\bbL \bbL\Omega_{A|R},M),
\]
In particular, if $A$ is smooth over $R$, then $H^\ast_{R\otimes\DL/B}(A;M) = \Ext^\ast_{B\otimes\Delta}(B\otimes_A\Omega_{A|R},M)$.

There is a complementary method by which these Quillen cohomology computations can, in certain cases, be reduced to linear computations. Observe that the forgetful functor $\Ring_{\DL}^\heart\rightarrow\LMod_F^\heart$ admits a left adjoint $S_F$, described by
\[
S_F M = \begin{cases}
SM/(Q^r x = x^2\text{ for }r = |x|),&\text{ for }p=2;\\
SM/(Q^r x = x^p\text{ for }2r= |x|),&\text{ for }p>2.
\end{cases}
\]
In fact this is already derived, i.e.\ agrees with its total derived functor on discrete objects, as can be seen from the description 
\[
S_F M = SM \tins{\psi}{S\Psi M}\bbF_p,
\]
where
\[
(\Psi M)_n = \begin{cases}
M_{n/p},&\text{ when }2,p|n;\\
0,&\text{ otherwise;}
\end{cases}\qquad
\psi(m) = \begin{cases}
m^2 - Q^{|m|}m,&\text{ for }p=2;\\
m^p - Q^{|m|/2}m,&\text{ for }p>2.
\end{cases}
\]
As a consequence, if $M\in\LMod_F^\heart$ and $B\in\Ring_\DL$, then $\Map_\DL(S_FM,B)\simeq\Map_F(M,B)$, and this provides an approach to computing the cohomology of $\DL$-rings in the image of $S_F$. 

These constructions easily extend to bases other than $\bbF_p$ and to augmented settings. In particular, if $M\in\LMod_F$ and $N\in\LMod_\Delta$, then
\[
\calH_{\DL/\bbF_{p}}(S_FM;N)\simeq\EXT_F(M,N).
\]

\begin{ex}
The module $e_n$ always carries an $F$-module structure where each $Q^r_\epsilon$ acts by zero; this is the action obtained from the augmentation on $F$. If $n\geq 0$, then $e_{-n}$ carries a second $F$-module structure, where $Q^0$ acts by the identity. If $e_{-n}'$ refers to this $F$-module structure, then $S_F(e_{-n}')$ is isomorphic to the cohomology algebra of the $n$-sphere, where if $n$ is even and $p$ is odd then we must take the ``homotopy theorist's even sphere'' $J_{p-1}S^n$ \cite{gray1993ehp}.
\tqed
\end{ex}

\subsection{The big lambda algebra}\label{ssec:cohomologydl}

We turn now to the construction of Koszul complexes computing $\Ext_F$. This discussion applies equally well to $\Ext_\Delta$, or to $\Ext_{F^u}$ for $u\in\bbZ$, as well to $\Ext_\calU$ (\cref{prop:unstablesteenrodkoszul}). Moreover, by \cref{prop:koszulcomposition}, it extends to $\Ext_{B\otimes F}$ for $B\in\Ring_F^\heart$, and to related contexts. These Koszul complexes have a number of predecessors, particularly with work of Miller \cite{miller1978spectral} in the connective setting, and with the unstable lambda complexes seen in work on the unstable Adams spectral sequence \cite{bousfieldkan1972homotopy}. We have found that working in the full $\bbZ$-graded setting serves to clarify some of the algebra.

We begin by observing that $F$ is a quadratic algebra. If we write its length grading as $F = \bigoplus_{n\geq 0}F[n]$, then the generating bimodule is
\[
F[1](e_a) = \begin{cases}
\bbF_2\{Q^r e_a:r\geq a\},&\text{when }p=2;\\
\bbF_p\{Q^r_\epsilon e_a:2r-\epsilon\geq a\},&\text{when }p>2.
\end{cases}
\]
The relations are just the image of the relations defining $\calB$ under the projection $\calB[1]\otimes\calB[1]\rightarrow F[1]\circ F[1]$. 
\begin{lemma}\label{lem:uakoszul}
The algebra $F$ is Koszul over $\bbF_{p\ast}$.
\end{lemma}
\begin{proof}
The algebra $F$ is locally finite, and its admissible basis may be viewed as a PBW decomposition, so \cref{prop:pbw} applies. 
\end{proof}
Thus there is indeed a theory of Koszul resolutions for computing $\Ext_F$, which we gain access to as soon as we understand the cohomology of $F$. 

We will describe the cohomology of $F$ in two ways, each shedding light on different aspects of the computation. The first approach proceeds by comparing the cohomology of $F$ with the cohomology of the ordinary algebra $\calB$, and the second approach proceeds by directly applying \cref{thm:quadraticduality}.

The first approach is plausible due to the following.

\begin{lemma}\label{lem:cohinj}
The surjection $\calB\rightarrow F$ yields an injection $H^\ast(F)\subset H^\ast(\calB)$.
\end{lemma}
\begin{proof}
It is sufficient to show dually that $H_\ast(\calB)\rightarrow H_\ast(F)$ is a surjection. As $F$ is Koszul, we need consider only the map on diagonal cohomology. This isis co given in degree $m$ by
\[
\bigcap_{i+j=m}B[1]^{\otimes i-1}\otimes R \otimes B[1]^{\otimes j-1}\rightarrow \bigcap_{i+j=m}F[1]^{\circ i-1}\circ R'\circ F[1]^{\circ j-1},
\]
where $R\subset B[1]\otimes B[1]$ is the bimodule of Adem relations and $R'\subset F[1]\circ F[1]$ is its image, so this is clear.
\end{proof}

And it is appealing due to the following.

\begin{lemma}\label{lem:cohbig}
Let $D$ be the diagonal cohomology algebra of $\calB$, defined with conventions as in \cite{priddy1970koszul} (cf.\ \cref{ex:signs}). Let $\hat{Q}^r_\epsilon\in D[1]$ be dual to $Q^r_\epsilon$. Then there is an injection $\calB\rightarrow D$ of algebras, with dense image, given by
\[
\begin{cases}
Q^r\mapsto \hat{Q}^{-r-1},&\text{when }p=2;\\
Q^r_\epsilon\mapsto \hat{Q}^{-r}_{1-\epsilon},&\text{when }p>2.
\end{cases}
\]
\end{lemma}
\begin{proof}
Though not quite stated in this form, \cite[Theorem 2.5]{priddy1970koszul} gives generators and relations for the diagonal cohomology of an arbitrary ordinary quadratic algebra over a field, with the caveat that if the algebra in question is not locally finite, then these generators may only be topological generators, and the relations obtained may involve infinite sums. In the case of $\calB$, it follows by direct computation that the relations obtained between the topological generators $\hat{Q}^r_\epsilon\in D$ are finite, and after the indicated change of indices are exactly the relations defining $\calB$.
\end{proof}

From here it is not difficult to proceed to fully describe the cohomology of $F$. We first lay out some conventions. In the present setting, it is best to compute the cohomology of $F$ with conventions that are standard when dealing with $\bbZ$-graded modules, only with pairings opposite to Yoneda composition. So for $x\in \Ext^n(e_a,e_b)$ and $y\in \Ext^m(e_b,e_c)$, write 
\[
xy = (-1)^{n(b-c+m)}y\circ x,
\]
where $\circ$ is the Yoneda composition of extensions. With this choice, our pairings are compatible with the graded opposite of \cite{priddy1970koszul}, as discussed in \cref{ex:signs}. 

Now if we let $\lambda_r\in\Ext^1_\calB(e_a,e_{a-r-1})$ be the image of $Q^r$ under \cref{lem:cohbig} when $p=2$, and $\lambda_r^\epsilon\in\Ext^1_\calB(e_a,e_{a-2r(p-1)+\epsilon-1})$ be the image of $Q^r_\epsilon$ under \cref{lem:cohbig} when $p>2$, then the multiplicative relations between the $\lambda$'s are exactly the relations in the graded opposite $\calB^\op$. In particular, the subspace of $\Ext^n_\calB(e_a,e_{a-\ast})$ generated under finite sums by products of the $\lambda$'s is isomorphic, up to shifts in degree, to $\calB^\op[n]$, and so has basis given by elements $\lambda_I$ where, if $p=2$, then $I = (r_1,\ldots,r_n)$ with $2r_i\geq r_{i+1}$ for each $i$, and if $p>2$, then $I = (r_1,\epsilon_1,\ldots,r_n,\epsilon_n)$ with $pr_i-\epsilon_i\geq r_{i+1}$ for each $i$; call these \textit{coadmissible sequences}.

Observe that as $F$ is locally finite, the inclusion $\Ext^n_F(e_a,e_{a-\ast})\subset\Ext^n_\calB(e_a,e_{a-\ast})$ lands in the subspace isomorphic to $\calB^\op[n]$ generated under finite sums by the elements $\lambda_I$. We have now all but proved the following.

\begin{theorem}\label{thm:dlcontra}
With the above notation, $\Ext^n_F(e_a,e_{a-\ast})\subset\Ext^n_\calB(e_a,e_{a-\ast})$ has as basis those $\lambda_I$ where $I$ is a coadmissible sequence satisfying:
\begin{enumerate}
\item If $p=2$, then $I = (r_1,\ldots,r_n)$ with $r_1 < -a$;
\item If $p>2$, then $I = (r_1,\epsilon_1,\ldots,r_n,\epsilon_n)$ with $2r_1-\epsilon_1 < -a$.\end{enumerate}
\end{theorem}
\begin{proof}
As $F$ has both generators and admissible basis induced by those of $\calB$, the functor $F$ may be viewed as a subfunctor of $\calB$, though not as a subalgebra. Using \cref{lem:cohinj}, it is seen that $\Ext^n_F(e_a,e_{a-\ast})\subset\Ext^n_\calB(e_a,e_{a-\ast})$ has image spanned by those coadmissible $\lambda_I$ which lift to elements of $\Hom_{\bbF_p}(\calB[1]^{\otimes n}e_a,e_{a-\ast})$ dual to simple tensors in $F[1]^{\circ n}e_a$. 

(1)~~ Consider the case $p=2$. Choose a coadmissible sequence $I = (r_1,\ldots,r_n)$. Then we must determine when $\lambda_I\in \Hom_{\bbF_p}(\calB[1]^{\otimes n}e_a,e_{a-\ast})$ is dual to an element of $F[1]^{\circ n}e_a$. By definition, this element is dual to $Q^{-r_n-1}\otimes\cdots\otimes Q^{-r_1-1}$, and for this to be an element of $F[1]^{\circ n}e_a$ the sequence $I$ must satisfy the instability condition
\[
-r_{i+1}-1\geq (-r_i-1)+\cdots+(-r_1-1)+a
\]
for each $i$. Write $I' = (s_1,\ldots,s_n) = (-r_n-1,\ldots,-r_1-1)$, so that this instability condition is
\[
s_i\geq s_{i+1}+\cdots+s_n+a
\]
for each $i$. Coadmissibility of $I$ is equivalent to the complete unadmissibility condition on $I'$ of $s_i>2s_{i+1}$ for each $i$; thus if $s_i\geq s_{i+1}+\cdots+s_n+a$ for some $i$, then
\[
s_{i-1}\geq 2s_i=s_i+s_i\geq s_i+s_{i-1}+\cdots+s_n+a.
\]
So in fact the instability condition is equivalent to just $s_n\geq a$, which itself is equivalent to $r_1<-a$ as claimed.

(2)~~ Consider the case $p>2$. Choose a coadmissible sequence $I = (r_1,\epsilon_1,\ldots,r_n,\epsilon_n)$, so that we must determine when $Q^{-r_n}_{1-\epsilon_n}\otimes\cdots\otimes Q^{-r_1}_{1-\epsilon_1}$ is an element of $F[1]^{\circ n}e_a$. Writing $I' = (\delta_1,s_1,\ldots,\delta_n,\epsilon_n) = (1-\epsilon_n,-r_n,\ldots,1-\epsilon_1,-r_1)$, the relevant instability condition is
\[
2s_i-\delta_i\geq (2s_{i+1}(p-1)-\delta_{i+1})+\cdots+(2s_k(p-1)-\delta_k)+a
\]
for each $i$. Coadmissibility of $I$ is equivalent to the complete unadmissibility condition on $I'$ of $s_i>ps_{i+1}-\delta_{i+1}$ for each $i$; thus if the above is satisfied for some $i$ then
\[
2s_{i-1}-\delta_{i-1}>2(ps_i-\delta_i)-\delta_{i-1}\geq 2s_i(p-1)-\delta_i+\cdots+2s_n(p-1)-\delta_n+a-\delta_{i-1},
\]
which in turn implies the instability condition at $i-1$ as $\delta_{i-1}\in\{0,1\}$. So the instability condition is equivalent to just $2s_n-\delta_n\geq a$, which in turn is equivalent to $2r_1-\epsilon_1<-a$.
\end{proof}

The second approach to $H^\ast(F)$ is through the following.

\begin{theorem}\label{thm:qdualityb}
For $a\in\bbZ$, the graded vector space $H^\ast(F)(e_a)$ has basis given by those $\lambda_I$ where $I$ is a coadmissible sequence satisfying
\begin{enumerate}
\item If $p=2$, then $I = (r_1,\ldots,r_n)$ with $2r_1+r_2+\cdots+r_n + n < -a$;
\item If $p>2$, then $I = (r_1,\epsilon_1,\ldots,r_n,\epsilon_n)$ with $2(pr_1-\epsilon_1)+2r_2(p-1)-\epsilon_2+\cdots+2r_n(p-1)-\epsilon_n+n<-a$.
\end{enumerate}
\end{theorem}
\begin{proof}
By \cref{thm:quadraticduality} and \cref{lem:uakoszul}, there is an isomorphism $H^\ast(F) = \widehat{T}(F[1]^\vee,R^\perp)$ where $R\subset F[1]\circ F[1]$ is the image of the Adem relations. Here, to maintain consistency with the sign conventions set out above, we should replace $R^\perp$ with $(1\otimes t)(R^\perp)$ where $t(x) = (-1)^{|x|}x$. Note
\[
F[1]^\vee(e_a)_b = \Mod_{\bbF_{p\ast}}(F[1](e_b),e_a).
\]
As before, when $p=2$, write $\lambda_r$ for the element of $F[1]^\vee(e_a)$ dual to $Q^{-r-1}$, with the understanding that this does not exist for all $a$, and when $p>2$, write $\lambda_r^\epsilon$ for the element of $F[1]^\vee(e_a)$ dual to $Q^{-r}_{1-\epsilon}$, with the same understanding. Either by appealing to \cref{lem:cohbig}, or else by carrying out the same sorts of computations, we find that $R^\perp$ is exactly the space of relations opposite to the Adem relations, only restricted to those elements which live in $F[1]^\vee$. Thus $H^\ast(F)(e_a)$ has basis given by those $\lambda_I$ where $I$ is coadmissible and $\lambda_I$ lifts to $F[1]^{\vee\circ n}(e_a)$.

(1)~~ Consider the case $p=2$. Here $F[1]^\vee(e_a)_b$ contains an element $\hat{Q}^{a-b}$ dual to $Q^{a-b}$ whenever $a-b\geq b$, i.e.\ $2b\leq a$. So $F[1]^\vee(e_a)$ is the space of $\hat{Q}^r$ with $2r\geq a$, and in general $F[1]^{\vee\circ n}(e_a)$ is the space of $\hat{Q}^{r_1}\otimes\cdots\otimes \hat{Q}^{r_n}$ with $2r_i+r_{i+1}+\cdots+r_n\geq a$ for each $i$. This is the space of $\lambda_{r_1}\otimes\cdots\otimes\lambda_{r_n}$ with $2r_i+r_{i+1}+\cdots+n-i+1<-a$ for each $i$. Coadmissibility of $I$ reduces this condition to the case $i=1$, which is the condition claimed.

(2)~~ Consider the case $p>2$. Here $F[1]^\vee(e_a)$ contains an element $\hat{Q}^r_\epsilon$ dual to $Q^r_\epsilon$ when $2(pr-\epsilon)\geq a$, so in general $F[1]^{\vee \circ n}(e_a)$ consists of those elements $\hat{Q}^{r_1}_{\epsilon_1}\otimes\cdots\otimes \hat{Q}^{r_n}_{\epsilon_n}$ with $2(pr_i-\epsilon_i)+2r_{i+1}(p-1)-\epsilon_{i+1}+\cdots+2r_n(p-1)-\epsilon_n\geq a$ for each $i$. These are the elements $\lambda_{r_1}^{\epsilon_1}\otimes\cdots\otimes\lambda_{r_n}^{\epsilon_n}$ with $2(pr_i-\epsilon_i)+2r_{i+1}(p-1)-\epsilon_{i+1}+\cdots+2r_n(p-1)-\epsilon_n+n-i+1<-a$ for each $i$, and again coadmissibility allows one to reduce to the case $i=1$, which is the condition claimed.
\end{proof} 

It is not difficult to directly translate between \cref{thm:dlcontra} and \cref{thm:qdualityb}. In effect, the first describes $\Ext_F^\ast(e_a,e_{\ast})$, whereas the second describes $\Ext_F^\ast(e_\ast,e_a)$.

\begin{rmk}
In the preceding, we have avoided the subtle point that though $\calB$ has a PBW basis, we cannot apply \cref{prop:pbw} to deduce that it is Koszul, as the necessary finiteness conditions are not obviously satisfied. Nonetheless $\calB$ is Koszul; this was considered in \cite{brunetticiampella2007priddy}, and we can give an alternate proof as follows. Let $H_u$ be the subalgebra, in our general sense, of $\calB^\op$, defined so that $H_u(e_a) = H^\ast(F_u)(e_a) = H^\ast(F)(e_{a+u})$ is as computed in \cref{thm:qdualityb}, up to the necessary shifts in degree. Then $H_u$ is a locally finite quadratic algebra admitting a PBW decomposition, so is Koszul. That $\calB^\op$, and thus $\calB$, is Koszul follows from the further observation that $\calB^\op\cong\colim_{u\rightarrow - \infty}H_u$.
\tqed
\end{rmk}

We end by noting the following, previously mentioned in \cref{ex:unstablesteenrod}. We omit the proof, as one may proceed exactly as in the above.

\begin{prop}\label{prop:unstablesteenrodkoszul}
The unstable Steenrod algebra $\calU$ is Koszul with respect to its length filtration. Moreover, $\gr\calU\simeq F/(Q^r:r\geq 0)$, and under this $\Ext^n_{\gr\calU}(e_a,e_b)$ is isomorphic to the subspace of $\Ext^n_F(e_a,e_b)$ spanned by those $\lambda_I$ where $I$ is a sequence of nonnegative integers.
\qed
\end{prop}

\subsection{Mapping spaces}\label{ssec:pmaps}

We now give applications to the homotopy theory of $\bbE_\infty$ algebras over $\bbF_p$.

\begin{theorem}\label{thm:pmaps}
Fix $R\in\CAlg_{\bbF_p}$, and choose $S\in\CAlg_R$ such that $R_\ast\rightarrow S_\ast$ is surjective (such as $S=0$ or $S=R$). Choose $A,B\in\CAlg_{R/S}$, and fix a map $\phi\colon A_\ast\rightarrow B_\ast$ in $\Ring_{R_\ast\otimes\DL/S_\ast}$. Let $\CAlg_{R/S}^\phi(A,B)$ be the space of lifts of $\phi$ to a map in $\CAlg_{R/S}$. Then there is a decomposition
\[
\CAlg_{R/S}^\phi(A,B)\simeq\lim_{n\rightarrow\infty}\CAlg_{R/S}^{\phi,\leq n}(A,B),
\]
with layers fitting into fiber sequences
\[
\CAlg_{R/S}^{\phi,\leq n}(A,B)\rightarrow\CAlg_{R/S}^{\phi,\leq n-1}(A,B)\rightarrow\calH^{n+1}_{R_\ast\otimes\DL/B_\ast}(\pi_\ast A;\pi_\ast \Omega^n F)
\]
for $n\geq 1$, where $F = \Fib(B\rightarrow S)$. In particular,
\begin{enumerate}
\item There are successively defined obstructions in $H^{n+1}_{R_\ast\otimes\DL/B_\ast}(\pi_\ast A;\pi_\ast \Omega^n F)$ for $n\geq 1$ to exhibiting a point of $\CAlg_{R/S}^\phi(A,B)$;
\item Once a point $f$ of $\CAlg_{R/S}^\phi(A,B)$ is chosen, there is a fringed spectral sequence of signature
\end{enumerate}
\[
E_1^{p,q} = H^{p-q}_{R_\ast\otimes\DL/B_\ast}(\pi_\ast A;\pi_\ast \Omega^p F)\Rightarrow \pi_q(\CAlg_{R/S}(A,B),f),\qquad d_r^{p,q}\colon E_r^{p,q}\rightarrow E_r^{p+r,q-1}.
\]
\end{theorem}
\begin{proof}
This is an application of \cite[Theorem 5.3.1]{balderrama2021deformations} to $\calP = \CAlg_{\bbF_p}^\free$, as we now describe.

In the language of \cite{balderrama2021deformations}, $\calP$ is a homotopical form of algebraic theory called a \textit{loop theory}. In particular, it has a category $\Model_\calP\subset\Psh(\calP)$ of models with a further distinguished subcategory $\Model_\calP^\Omega\subset\Model_\calP$ of \textit{loop models}, and \cite[Theorems 3.3.3, 3.3.7]{balderrama2021deformations} imply that the Yoneda embedding $h\colon \CAlg_{\bbF_p}\rightarrow \Psh(\calP)$ restricts to an equivalence $\CAlg_{\bbF_p}\simeq\Model_{\calP}^\Omega\subset\Model_\calP\subset\Psh(\calP)$. The precise definitions of these terms are not needed here; we only need that this equivalence ensures
\[
\CAlg_{R/S}^\phi(A,B)\simeq \Map_{R/\calP/S}^\phi(A,B).
\]

Observe that $\h\calP$ is exactly the theory of $\DL$-rings as defined in \cref{ssec:ppleth}. Under this correspondence, if $R\in \CAlg_{\bbF_p}$ then $\pi_0 h(R) = R_\ast$ in $\Ring_{\DL}\simeq \Model_{\h\calP}$. The condition that $R_\ast\rightarrow S_\ast$ is surjective is then equivalent to the condition $\pi_0 h(R)\rightarrow \pi_0 h(S)$ is surjective. 

Now write $p\colon B\rightarrow S$, or equivalently $p\colon h(B)\rightarrow h(S)$, for the given structure map. Then $\pi_\ast \Omega^n F = \Pi_n p$ as an abelian group object in $\Ring_{R_\ast\otimes \DL/B_\ast}\simeq \pi_0 h(R)/\Model_{\h\calP}/\pi_0 h(B)$, from which it follows that
\[
\calH^{n+1}_{R_\ast\otimes\DL/B_\ast}(\pi_\ast A;\pi_\ast \Omega^n F)\simeq \Map_{\pi_0 h(R)/\h\calP/\pi_0 h(B)}(\pi_0 h(A),B^{n+1}_{\pi_0 h(B)}\Pi_n p).
\]

The theorem currently under consideration now follows by a direct translation from \cite[Theorem 5.3.1]{balderrama2021deformations}, only writing $\CAlg_{R/S}^{\phi,\leq n}(A,B)$ in place of $\Map_{h(R)/\calP/h(S)}^{\phi,\leq n}(h(A),h(B))$.
\end{proof}

\begin{ex}\label{ex:finp}
Let $\Fin_p$ denote the category of $p$-finite spaces, i.e.\ those spaces $X$ such that $X$ is truncated, $\pi_0 X$ is finite, and $\pi_n(X,x)$ is a finite $p$-group for all $x\in X$ and $n\geq 1$. By Mandell's $p$-adic homotopy theory  \cite{mandell1998einfinity}, interpreted in the $\infty$-categorical context by Lurie \cite{lurie2011rational}, there is a fully faithful embedding
\[
\Fin_p^\op\rightarrow\CAlg_{\ol{\bbF}_p},\qquad X\mapsto C(X;\ol{\bbF}_p),
\]
where $C(X;\ol{\bbF}_p)$ is the spectrum of $\ol{\bbF}_p$-valued cochains on $X$, and this extends to a fully faithful embedding $\Pro(\Fin_p)^\op\simeq\Ind(\Fin_p^\op)\rightarrow\CAlg_{\ol{\bbF}_p}$. In particular, given $X,Y\in\Gpd_\infty$ which are simply connected and of finite type, there is an equivalence
\[
\Map(X,Y_p^\wedge)\simeq\CAlg_{\ol{\bbF}_p}(C(Y;\ol{\bbF}_p),C(X,\ol{\bbF}_p)).
\]

By \cref{thm:ualgemb}, we may view the obstruction theory of \cref{thm:pmaps} in this context as giving an unstable Adams spectral sequence. To note a special case, observe that the construction at the end of \cref{ssec:dlringcoh} easily translates to describe a free functor $S_\calU\colon \LMod_\calU\rightarrow \Ring_\calU^\aug$. In the Massey-Peterson case, where $X$ and $Y$ are pointed simply connected spaces of finite type and $H^\ast Y\cong S_\calU M$ for some $M\in\LMod_\calU$, this gives a spectral sequence
\[
E_1^{p,q} = \Ext_\calU^{p-q}(M;\widetilde{H}^{\ast-p}X)\Rightarrow\pi_q\Map_\ast(X,Y_p^\wedge),\qquad d_r^{p,q}\colon E_r^{p,q}\colon E_r^{p+r,q-1}.
\]
The Koszul complexes guaranteed by \cref{prop:unstablesteenrodkoszul} recover the lambda complexes for computing these $\Ext$ groups.
\tqed
\end{ex}

\begin{ex}\label{ex:affineline}
Recall from \cref{ex:salgebras} that the forgetful functor $U\colon\CRing_{\bbF_p}\rightarrow\CAlg_{\bbF_p}^\cn$ is plethystic, and write its right adjoint as $\bbA^1$. Then for $R\in\CAlg_{\bbF_p}^\cn$ there is an equivalence
\[
\bbA^1(R) \simeq\CAlg_{\bbF_p}(\bbF_p[t],R),
\]
where $\bbF_p[t]$ has homotopy groups concentrated in degree $0$. More generally, for any $a\in\bbZ$, we can consider the $\bbF_{p\ast}$-ring $S(e_a)$ as a differential graded algebra with trivial differential, in this way upgrade it to an object of $\CAlg_{\bbF_p}$, and for $A\in\CAlg_{\bbF_p}$ consider the space $\CAlg_{\bbF_p}(S(e_a),A)$. The following are some comments about what \cref{thm:pmaps} says about computing these spaces.

Observe first that the $\bbF_{p\ast}$-plethory $\DL$ is augmented over the initial $\bbF_{p\ast}$-plethory; this is not a purely formal fact, but is easily seen from the structure of $\DL$. Restriction along the augmentation is itself a plethystic functor $\CRing_{\bbF_{p\ast}}\!\rightarrow \Ring_{\DL}$. More generally, there are plethystic functors $\CRing_{R/S}\rightarrow\Ring_{R\otimes\DL/S}$ for $R\in\Ring_{\DL}^\heart$ and $S\in\Ring_{R\otimes\DL}^\heart$; write $G$ for the right adjoints to these. The filtration of $\CAlg_{\bbF_p}(S(e_a),R)$ guaranteed by \cref{thm:pmaps} has layers that can now be identified as
\[
\calH^{n+1}_{\DL/R_\ast}(S(e_a),\pi_\ast \Omega^n R) \simeq \Map_{\DL/R_\ast}(S(e_a),B^{n+1}_{R_\ast}\pi_\ast R^{S^n_+})\simeq G(B^{n+1}_{R_\ast}\pi_\ast R^{S^n_+})_a,
\]
where $B^n_{R_\ast}$ denotes $n$-fold delooping in the slice category over $R_\ast$. 

Consider for simplicity the case where $R$ is augmented. Then the resulting spectral sequence for computing $\pi_\ast\CAlg_{\bbF_p}^\aug(S(e_a),R)$ is of signature
\[
E_1^{p,q} = \calH^{p-q}_{\DL/R_\ast}(S(e_a);\pi_\ast \Omega^p R)\simeq \Ext_\Delta^{p-q}(e_a,s^{-p}R_{\ast})\Rightarrow \pi_q\CAlg_{\bbF_p}^\aug(S(e_a),R),
\]
where $\Delta$ acts trivially on $e_a$. 

For further simplicity, specialize to $p=2$, and write $M = R_\ast$; we can describe the Koszul complex $K_\Delta(e_a,s^{-\ast}M)$ computing the above $\Ext$ groups as follows. Consider the space of tensors $\lambda_I\otimes m$ where $m\in M$ and $I = (r_1,\ldots,r_k)$ is a coadmissible sequence satisfying $r_1 < -a-1$ and $r_1+\cdots+r_k + k \geq -m$. Now, complete this space to allow for infinite sums of the form $\sum_i \lambda_{I_i}\otimes m_i$ so long as for any fixed $n\in\bbZ$, there are finitely many nonzero terms involving $m_i$ with $|m_i| = n$. This is $K_\Delta(e_a,s^{-\ast}M)$. The differential is given by
\[
\delta(\lambda_I\otimes m) = \sum_{r\in\bbZ}\lambda_I\lambda_{-r-1}\otimes Q^r(m).
\]

Return now to the special case of $\bbA^1$. Possibly more well-known is the subspace
\[
\bbG_m(R) = \CAlg_{\bbF_p}(\bbF_p[t^{\pm 1}],R)\simeq\bbA^1(R)^\times\subset\bbA^1(R)
\]
of $\bbA^1(R)$ given by the strict units of $R$. All path components of $\bbA^1(R)$ and $\bbG_m(R)$ are equivalent, so these objects only differ in $\pi_0$. The Goerss-Hopkins spectral sequence computing $\pi_\ast\bbG_m(R)$, and relevant Koszul complex, has been studied by Fung \cite{fung2019strict}. The perspective on $\bbG_m(R)$ afforded by viewing it as the spectrum of units of $\bbA^1(R)$ extends \cite[Proposition 3.11]{fung2019strict} to show that $\pi_{n}\bbG_m(R)$ is always an $\bbF_p$-vector space for $n>0$.
\tqed
\end{ex}

\subsection{Andr\'e-Quillen-Goodwillie towers}\label{ssec:paq}

Fix $R\in\CAlg_{\bbF_p}$, and consider the category $\CAlg_R^\aug$ of augmented $\bbE_\infty$ algebras over $R$. The functor $\CAlg_R^\aug\rightarrow\Mod_R$ sending $A$ to its augmentation ideal $\Fib(A\rightarrow R)$ is monadic; write $\widetilde{\bbP}$ for the associated monad on $\Mod_R$, given by $\widetilde{\bbP}M = \bigoplus_{n\geq 1}\bbP_n M$. For $A\in\CAlg_R^\aug$, one can consider in this context the topological Andr\'e-Quillen homology and cohomology of $A$ relative to $R$; write these as $\TAQ^R(A)$ and $\TAQ_R(A)$. Here, $\TAQ_R(A) = \Mod_R(\TAQ^R(A),R)$, and $\TAQ^R\colon\CAlg_R^\aug\rightarrow \Mod_R$ may be characterized as the unique functor which preserves  geometric realizations and sends $\bbP M$ to $M$ for $M\in\Mod_R$. Put another way, $\TAQ^R$ is left adjoint to restriction along the augmentation $\widetilde{\bbP}\rightarrow I$.

We can do something a little more general. Consider the augmentation ideal functor
\[
U\colon\CAlg_R^\aug\rightarrow \Mod_R.
\]
This has a Goodwillie tower, which we write as
\[
U\rightarrow\cdots\rightarrow P_n^R\rightarrow P_{n-1}^R\rightarrow\cdots\rightarrow P_1^R\rightarrow 0.
\]
The $n$'th term $P_n^R\colon \CAlg_R^\aug\rightarrow\Mod_R$ can be characterized as the unique functor which preserves geometric realizations and satisfies
\[
P_n^R(\bbP M) = \widetilde{\bbP} M / \bigoplus_{k>n}\bbP_k M \simeq \bigoplus_{1\leq k \leq n}\bbP_k M.
\]
In particular, $P_1^R = \TAQ^R$. See \cite[Section 3]{kuhn2003localization} for more on this tower.

By restricting $P_n^R$ to $\CAlg_R^{\aug,\free}$ and taking homotopy groups we obtain a functor $\h\CAlg_R^{\aug,\free}=\Ring_{R_\ast\otimes\DL}^{\aug,\free}\rightarrow\Mod_{R_\ast}^\heart$. By left Kan extension, this extends to a functor
\[
\ol{P}{}_n^{R_\ast}\colon \Ring_{R_\ast\otimes\DL}^{\aug,\heart}\rightarrow\Mod_{R_\ast}^\heart.
\]
When $n=1$, this is left adjoint to
\[
\Mod_{R_\ast}\rightarrow\Ring_{R_\ast\otimes\DL}^\aug,\qquad M_\ast\mapsto R_\ast\ltimes\ol{M}_\ast.
\]
Thus, where $\epsilon\colon R_\ast\otimes\Delta\rightarrow R_\ast$ is the augmentation and $Q_{R_\ast}\colon \Ring_{R_\ast\otimes\DL}^\aug\rightarrow\LMod_{R_\ast\otimes\Delta}$ is the functor of indecomposables, there is an equivalence
\[
\bbL\ol{P}{}_1^{R_\ast}\simeq\epsilon_! \bbL Q_{R_\ast},
\]
where $\epsilon_!$ is considered in the derived sense. This can be computed via a bar construction, as in \cref{ssec:homologycohomology}.

\begin{theorem}\label{thm:paqss}
Fix notation as above, and fix $A\in\CAlg_R^\aug$. 
\begin{enumerate}
\item There is a convergent spectral sequence in $\Mod_{R_\ast}^\heart$ of signature
\[
E^1_{p,q} = s^q \bbL_{p+q}\ol{P}{}_n^{R_\ast}(A_\ast) \Rightarrow \pi_{\ast+p}P_n^R(A),\qquad d^r_{p,q}\colon E^r_{p,q}\rightarrow E^r_{p-r,q-1};
\]
\item There is a conditionally convergent spectral sequence of signature
\[
E_1^{p,q} = \Ext_{R_\ast\otimes\Delta}^{p+q}(\bbL Q_{R_\ast}(A_\ast),\ol{R}_{\ast+p})\Rightarrow \TAQ^q_R(A),\qquad d_r^{p,q}\colon E_r^{p,q}\rightarrow E_r^{p+r,q+1}.
\]
\end{enumerate}
\end{theorem}
\begin{proof}
Given the preceding discussion, the first spectral sequence is obtained by an application of \cite[Theorem 4.2.2]{balderrama2021deformations}. The second spectral sequence can be obtained by patching together the filtrations of $\CAlg_R^\aug(A,R\ltimes \Sigma^n \ol{R})$ for various $n$ given by \cref{thm:pmaps} and applying the isomorphisms $H^{p+q}_{R_\ast\otimes\DL/R_\ast}(A_\ast; \pi_\ast \Omega^{p}\ol{R})\cong\Ext_{R_\ast\otimes\Delta}^{p+q}(\bbL Q_{R_\ast}(A_\ast),\ol{R}_{\ast+p})$.
\end{proof}

A form of this spectral sequence was originally constructed by Basterra \cite{basterra1996andre}.

\begin{ex}\label{ex:millerss}
Work of Miller \cite{miller1978spectral} produces a spectral sequence converging to $H_\ast X$ for a connective spectrum $X$, with initial page depending on $H_\ast\Omega^\infty X$ as a ring over the Dyer-Lashof algebra. This can be understood from the perspective of \cref{thm:paqss}: there is an equivalence
\[
\TAQ^{\bbF_p}(\bbF_p\otimes\Sigma^\infty_+\Omega^\infty X) \simeq \bbF_p\otimes X
\]
when $X$ is connective \cite[Example 3.9]{kuhn2003localization}, and thus \cref{thm:paqss} gives a spectral sequence
\[
E^1_{p,q} = s^q \bbL_{p+q}\ol{P}{}_1^{\bbF_{p\ast}}(H_\ast \Omega^\infty X) = s^q \pi_{p+q}\epsilon_! \bbL Q(H_\ast\Omega^\infty X)\Rightarrow H_{\ast+p}X.
\]
The underlying ring of $H_\ast\Omega^\infty X$ splits as
\[
H_\ast\Omega^\infty X = \bbF_p[\pi_0 X]\otimes H_\ast\Omega^\infty_0 X,
\]
where $\bbF_p[\pi_0 X]$ is an abelian group ring and $H_\ast\Omega^\infty_0 X$ is a connected Hopf algebra. The structure theory of graded bicommutative Hopf algebras implies that $\bbL Q(H_\ast\Omega^\infty X)$ is always $1$-truncated, and so the $E^1$-page of this spectral sequence is somewhat accessible even when $H_\ast\Omega^\infty X$ is not smooth.
\tqed
\end{ex}

\section{Lubin-Tate spectra}\label{sec:e}

In this section we apply the machinery developed in the preceding sections to the study of $K(h)$-local $\bbE_\infty$ algebras over a Lubin-Tate spectrum. Before we get to these applications, we must cover some additional algebraic topics.

\subsection{Even-periodic plethories}\label{ssec:evenperiodic}

Fix an ordinary $\bbZ$-graded commutative ring $R = R_\ast$, with associated category $\Mod_R = \Mod_{R_\ast}$ of $\bbZ$-graded modules. Consider $\Mod_R$ as a symmetric monoidal category with symmetrizer employing the Koszul sign rule, and abbreviate both $\otimes_{R_\ast}$ and $\otimes_{R_0}$ to just $\otimes$. Write $E$ for a copy of $R$ generated in degree $1$.

\begin{defn}
\hphantom{blank}
\begin{enumerate}
\item $R$ is \textit{even-periodic} if $R_1 = 0$, and for all $k\in\bbZ$, the map $R_k\otimes R_2\rightarrow R_{k+2}$ is an isomorphism. The following definitions will be made under the assumption that $R$ is even-periodic.
\item An $R$-cobialgebroid $\Gamma$ is \textit{even-periodic} if 
\begin{enumerate}
\item As a functor, $\Gamma$ preserves the full subcategory of $\Mod_R^\heart$ spanned by those $R$-modules which are concentrated in even degrees;
\item We have chosen a map $\Gamma^E\rightarrow\Gamma$ equipping $\Gamma$ with suspensions (\cref{def:suspensions}) which is an isomorphism when restricted to the full subcategory $\Mod_R^\heart$ spanned by those $R$-modules which are concentrated in even degrees.
\end{enumerate}
\item An $R$-plethory $\Lambda$ is \textit{even-periodic} if
\begin{enumerate}
\item The underlying exponential monad of $\Lambda$ preserves the full subcategory of $\Mod_R^\heart$ spanned by those $R$-modules which are concentrated in even degrees;
\item We have chosen an isomorphism $\Delta(\Lambda)^E\rightarrow\Gamma(\Lambda)$ which equips $\Lambda$ with suspensions (\cref{def:suspensions}) and makes $\Gamma(\Lambda)$ into an even-periodic cobialgebroid.
\tqed
\end{enumerate}
\end{enumerate}
\end{defn}

For the rest of this subsection, $R$ is assumed to be even-periodic. Because $R$ is even-periodic, multiplication gives an isomorphism $R_{-2}\otimes R_2\rightarrow R_0$, and so $R_2$ is an invertible $R_0$-module. Denote this module by $L$. The ring $R$ may then be identified as $R = \bigoplus_{n\in\bbZ}L^n$, where $L^n = L^{\otimes n}$ consists of elements in degree $2n$.

Let $\Mod_{R_\star} = \Mod_{R_0}^{\times 2}$ denote the category of $\bbZ/(2)$-graded $R_0$-modules. Then the functor
\[
\Mod_R\rightarrow\Mod_{R_\star},\qquad M_\ast\mapsto (M_0,M_{-1})
\]
is an equivalence of categories, for it has essential inverse
\[
(M_0,M_{-1})\mapsto M_\ast,\qquad M_{2n-\epsilon} = L^{n}\otimes M_{-\epsilon}.
\]
In particular, $\Mod_{R_0}$ may be identified with the full subcategory of $\Mod_R$ spanned by those $R$-modules which are concentrated in even degrees. Under the above equivalence, the symmetric monoidal structure on $\Mod_R$ transfers to a symmetric monoidal structure on $\Mod_{R_\star}$ whose tensor product may be identified as
\begin{align*}
(M_0,M_{-1})\otimes (M_0',M_{-1}') = (&M_0\otimes M_0'\oplus L\otimes M_{-1}\otimes M_{-1}',\\
&M_0\otimes M_{-1}'\oplus M_{-1}\otimes M_0'),
\end{align*}
where the symmetrizer acts with a sign on $L$.

Let $L^{1/2} = E^{-1}$, considered as either an object of $\Mod_{R}$ or $\Mod_{R_\star}$. Then
\[
L^{1/2} = (0,R_0),\qquad L^{1/2}\otimes L^{1/2} = L^1 = (L,0).
\]
So for every $M\in\Mod_R$ there are unique $R_0$-modules $M_0$ and $M_{-1}$ such that
\[
M\cong M_0\oplus L^{1/2}\otimes M_{-1}.
\]

Fix next an even-periodic cobialgebroid $\Gamma$, and abbreviate $\Gamma_{n,m} = \Gamma(E^n R)(E^mR)$. Even-periodicity of $R$ implies
\[
L\tins{}{l}\Gamma_{n,m} \cong \Gamma_{n,m+2},\qquad \Gamma_{n,m}\tins{r}{} L \cong \Gamma_{n-2,m}.
\]
Here, each $\Gamma_{n,m}$ is an $R_0$-bimodule, and we have used subscripts to indicate which $R_0$-module structure we are taking a tensor product with respect to: $l$ is for left, and $r$ is for right. The assumption that $\Gamma$ preserves even objects implies that $\Gamma_{n,m}=0$ unless $n\equiv m\pmod{2}$. The suspension maps for $\Gamma$ are maps $\Gamma_{n-1,m-1}\rightarrow\Gamma_{n,m}$; even-periodicity of $\Gamma$ implies that these are isomorphisms when $n$ and $m$ are even, and they are algebra maps when $n=m$. The algebra structure on $\Gamma$ is thus essentially determined by $\Gamma_{0,0}$, and there is an equivalence of categories $\LMod_\Gamma\simeq\LMod_{\Gamma_\star}$ overlying the equivalence $\Mod_R\simeq\Mod_{R_\star}$, where $\LMod_{\Gamma_\star} = \LMod_{\Gamma_{0,0}}^{\times 2}$ is the category of $\bbZ/(2)$-graded $\Gamma_{0,0}$-modules.

The coproduct on $\Gamma$ is encoded by maps
\[
\Delta^\times\colon \Gamma_{n+n',m+m'}\rightarrow\Gamma_{n,m}\tins{l}{l}\Gamma_{n',m'},
\]
which for $n=n'=0=m=m'$ contribute to the $R_0$-cobialgebroid structure of $\Gamma_{0,0}$. As $\Mod_{\Gamma_\star}$ is strongly monoidal over $\Mod_{R_\star}$, its tensor product must take the form
\begin{align*}
(M_0,M_{-1})\otimes (M_0',M_{-1}') = (&M_0\otimes M_0'\oplus L\otimes M_{-1}\otimes M_{-1}',\\
&M_0\otimes M_{-1}'\oplus M_{-1}\otimes M_0').
\end{align*}
But this does not fully describe the tensor product: missing is a description of the $\Gamma_{0,0}$-module structure. On the summands $M_0'\otimes M_0''$, $M_0'\otimes M_{-1}''$, and $M_{-1}'\otimes M_0''$, this action is obtained just from the coproduct on $\Gamma_{0,0}$ and isomorphism $\Gamma_{-1,-1}\cong\Gamma_{0,0}$, so consider the remaining summand $L\otimes M_{-1}'\otimes M_{-1}''$. Here by definition the action arises via the map
\begin{align*}
\Gamma_{0,0} &\cong L\tins{}{l} \Gamma_{-2,-2}\tins{r}{} L^{-1}\\
&\rightarrow L\tins{}{l}\Gamma_{-1,-1}\tins{l}{l}\Gamma_{-1,-1}\tins{r}{} L^{-1}
\cong L\tins{}{l}\Gamma_{0,0}\tins{l}{l}\Gamma_{0,0}\tins{r}{} L^{-1}.
\end{align*}
On the other hand, there is an action of $\Gamma_{0,0}$ on $L\otimes M_{-1}\otimes M_{-1}'$ obtained from the  iterated coproduct of $\Gamma_{0,0}$ and the action of $\Gamma_{0,0}$ on $L$ by way of the double suspension $\Gamma_{0,0}\rightarrow\Gamma_{2,2}$. These actions agree; the definition of a cobialgebroid with suspensions was chosen in order to make this so. This gives a full understanding of $\LMod_{\Gamma_\star}$ as a symmetric monoidal category; we can summarize the situation as follows.

\begin{prop}
Let $\Gamma$ be an even-periodic $R$-cobialgebroid. In particular, $\Gamma_{0,0}$ is an $R_0$-cobialgebroid, and there is a chosen $\Gamma_{0,0}$-module structure on $L = R_2$. Then there is an equivalence of symmetric monoidal categories
\[
\LMod_\Gamma\simeq\LMod_{\Gamma_\star},\qquad M_\ast\mapsto (M_0,M_{-1}),
\]
where $\LMod_{\Gamma_\star} = \LMod_{\Gamma_{0,0}}^{\times 2}$ is the category of $\bbZ/(2)$-graded $\Gamma_{0,0}$-modules, with symmetric monoidal product given by
\begin{align*}
(M_0,M_{-1})\otimes (M_0',M_{-1}') = (&M_0\otimes M_0'\oplus L\otimes M_{-1}\otimes M_{-1}',\\
&M_0\otimes M_{-1}'\oplus M_{-1}\otimes M_0'),
\end{align*}
where $\Gamma_{0,0}$ acts on each term through its coproduct and the symmetrizer acts on $L$ with a sign.
\qed
\end{prop}

\begin{warning}
Note that although $L$ is an invertible $R_0$-module, it is generally not invertible as a $\Gamma_{0,0}$-module.
\tqed
\end{warning}

\begin{rmk}\label{rmk:z2mult}
Let $A$ be an object of $\Ring_{\Gamma_\star}^\heart$. Given $x,y\in A_{-1}$, one may wish to form their product $xy$, and to consider how the elements of $\Gamma_{0,0}$ act on this product. However $xy\in A_{-2}$, i.e.\ this product takes us outside the $\bbZ/(2)$-graded setting. To remain in the $\bbZ/(2)$-graded setting, it is more correct to say that the product of two elements of $A_{-1}$ is given by a map
\[
L\otimes A_{-1}\otimes A_{-1}\rightarrow A_0
\]
of $\Gamma_{0,0}$-modules. If $L$ is trivializable, then by choosing a trivialization we may treat this as a map $A_{-1}\otimes A_{-1}\rightarrow A_0$, but one must not forget the presence of $L$ in considering how this map interacts with $\Gamma$.
\tqed
\end{rmk}

\begin{ex}\label{ex:ltep}
Let $\Gamma$ be the $R = W(\kappa)[[a]]$-cobialgebroid of \cref{ex:ltkoszul} and \cref{ex:ltc}. Then $\Gamma$ upgrades to an even-periodic cobialgebroid over the even-periodic ring $R[u^{\pm 1}]$, where $|u| = 2$. The $\Gamma$-module structure on $\omega = L = R\{u\}$ encoding this is given by
\[
Q_0 u = 0,\qquad Q_1 u = -u,\qquad Q_2 u = 0.
\]
If $M$ is a $\Gamma$-module, then $\omega\otimes M$ consists of elements $um$ for $m\in M$, and has $\Gamma$-module structure
\[
Q_0(um) = - 2uQ_2(m),\quad Q_1(um) = -uQ_0(m) - au Q_2(m),\quad Q_2(um) = - u Q_1(m).
\]

A $\bbZ/(2)$-graded $\Gamma$-ring is then a $\bbZ/(2)$-graded $R$-ring $A$ equipped with an action of $\Gamma$ such that if either $x\in A_0$ or $y\in A_0$, then $\Gamma$ acts on $xy$ via \cref{ex:ltc}; and if $x,y\in A_{-1}$, then $\Gamma$ acts on $xy\in A_0$ by
\begin{align*}
Q_0(xy) = &- 2 Q_0(x) Q_2(y) - 2 Q_2(x)Q_0(y) - 2 Q_1(x)Q_1(y) - 2 a Q_2(x) Q_2(y),\\
Q_1(xy) = &- Q_0(x) Q_0(y) - 2 Q_1(x) Q_2(y) - 2 Q_2(x)Q_1(y) \\
& - a Q_0(x) Q_2(y) - a Q_2(x) Q_0(y) - a Q_1(x) Q_1(y)  - a^2 Q_2(x)Q_2(y),\\
Q_2(xy) = &- Q_0(x) Q_1(y) - Q_1(x) Q_0(y)\\
& - a Q_1(x) Q_2(y) - a Q_2(x) Q_1(y) - 2 Q_2(x) Q_2(y).
\end{align*}
See especially \cref{ex:ltring} and \cref{ex:ltcurve} for more on this example.
\tqed
\end{ex}

Now fix an even-periodic $R$-plethory $\Lambda$, and abbreviate $\Gamma = \Gamma(\Lambda)$ and $\Delta = \Delta(\Lambda)$. We can picture the relevant suspension maps as fitting into a diagram
\begin{center}\begin{tikzcd}
&\Gamma_{-1,-1}\ar[d,"\simeq"]&\Gamma_{0,0}\ar[d]&\Gamma_{1,1}\ar[d,"\simeq"]&\Gamma_{2,2}\ar[d]&\Gamma_{3,3}\\
\Delta_{-2,-2}\ar[ur,"\simeq"]&\Delta_{-1,-1}\ar[ur,"\simeq"]&\Delta_{0,0}\ar[ur,"\simeq"]&\Delta_{1,1}\ar[ur,"\simeq"]&\Delta_{2,2}\ar[ur,"\simeq"]
\end{tikzcd}.\end{center}
As $R$ is even-periodic, there are canonical equivalences $\LMod_{\Gamma_{n,n}}\simeq\LMod_{\Gamma_{n+2,n+2}}$ and $\LMod_{\Delta_{n,n}}\simeq\LMod_{\Delta_{n+2,n+2}}$ for each $n\in\bbZ$, given by tensoring with $L$. Coupling these with the isomorphisms in the above diagram yields the following.

\begin{prop}\label{prop:epmorita}
There is a canonical Morita equivalence $
\LMod_{\Delta}\simeq\LMod_{\Gamma},
$
and the composite
$
\LMod_{\Gamma}\simeq\LMod_{\Delta}\rightarrow\LMod_{\Gamma}
$, the second functor being restriction along $\Gamma\rightarrow\Delta$, 
is given by $L^{1/2}\otimes\bs$.
\qed
\end{prop}

Although $\LMod_\Delta$ is what appears when considering the Quillen cohomology of $\Lambda$-rings, in many examples of interest it is possible to produce a partial section of $L^{1/2}\otimes\bs\colon \LMod_\Gamma\rightarrow\LMod_\Gamma$, thereby allowing one to reduce to computations in $\LMod_\Gamma$. See \cref{ex:ltcurve} for an example.

\subsection{Quasicoherent sheaves}\label{ssec:quasicoherent}

Fix an ordinary commutative ring $R$, abbreviating $\otimes = \otimes_R$, and fix an $R$-cobialgebroid $\Gamma$. If we forget the algebra and right $R$-module structures on $\Gamma$, then we are left with nothing more than a (counital, coassociative, cocommutative) $R$-coalgebra. Under suitable niceness assumptions, $R$-coalgebras give one approach to the theory of formal schemes over $R$; this is best known when $R$ is a field and these niceness assumptions are automatic, see for instance \cite[Section I.6]{demazure1972lectures}. It turns out that the entire $R$-cobialgebroid structure of $\Gamma$ may be understood this way, at least under suitable niceness assumptions.

In the following, we will freely use the language of formal schemes as developed in \cite{strickland2000formal}, particularly the technical notions of solid formal schemes and coalgebraic formal schemes; however, we will write $\sch$ for what is written there as $\mathrm{sch}$, informally defined as $\sch C = \Spf(C^\vee)$ for an $R$-coalgebra $C$ with good basis. We abbreviate ``coalgebra with good basis'' to ``good coalgebra'', and write $\cotimes$ for the completed tensor product of pro-$R$-modules.

The category of formal schemes is, in particular, a certain full subcategory of the category of functors $\CRing\rightarrow\Set$. For our purposes, a \textit{formal category scheme} will be a category object in $\Fun(\CRing,\Set)$ whose object object and morphism object are both formal schemes.

\begin{prop}
Fix an $R$-cobialgebroid $\Gamma$ which is good as an $R$-coalgebra. Then the pair $(\Spec R,\sch \Gamma)$ naturally carries the structure of a formal category scheme.
\end{prop}
\begin{proof}
We must describe the structure of a category object on the pair $(\Spec R,\sch \Gamma)$. The source map $s\colon\sch \Gamma\rightarrow\Spec R$ is simply the map arising from the definition of $\sch\Gamma$ as a formal $R$-scheme. In describing the remaining maps, we will make use of the fact that $\sch\Gamma$ is a solid formal scheme, so that it suffices to work with $\Gamma^\vee = \LMod_R(\Gamma,R)$ as a formal ring. 
The target map $t\colon \sch\Gamma\rightarrow\Spec R$ is dual to the map of formal rings
\[
t\colon R\rightarrow\Gamma^\vee,\qquad t(r)(\gamma) = \epsilon(\gamma r).
\]
The unit map $\iota\colon\Spec R\rightarrow\sch\Gamma$ is dual to the map of formal rings
\[
\iota\colon\Gamma^\vee\rightarrow R,\qquad\iota(f) = f(1).
\]
To define $c\colon \sch \Gamma {}_s\times_{\Spec R,t}\sch\Gamma\rightarrow\Gamma$, observe first that $\sch \Gamma \tims{s}{\Spec R,t}\sch \Gamma$ is a solid formal scheme represented by
$
\Gamma^\vee\mathop{{}_s\cotimes_{t}} \Gamma^\vee.
$
As $\Gamma$ admits a good basis, it may be written as $\Gamma\cong\colim_\alpha\Gamma_\alpha$ where $\Gamma_\alpha\subset\Gamma$ is a standard coalgebra, in which case $\Gamma^\vee\cong\lim_\alpha\Gamma_\alpha^\vee$ as a formal ring with each $\Gamma_\alpha^\vee$ discrete and finitely generated free as a right $R$-module. It follows that
\begin{align*}
\Gamma^\vee\mathop{{}_s\cotimes_{t}} \Gamma^\vee&\cong (\lim_\alpha\Gamma_\alpha^\vee)\mathop{{}_s\cotimes_t}\Gamma^\vee\\
&\cong \lim_\alpha (\Gamma_\alpha^\vee\tins{s}{t} \Gamma^\vee)\\
&\cong \lim_\alpha(\Gamma\tins{r}{l}\Gamma_\alpha)^\vee \cong (\Gamma\tins{r}{l}\Gamma)^\vee.
\end{align*}
A similar argument can be used to show that $\Gamma\tins{r}{l}\Gamma$ is itself a good $R$-coalgebra, and the above then gives an isomorphism
\[
\sch\Gamma \tims{s}{\Spec R,t}\sch\Gamma\cong\sch(\Gamma\tins{r}{l}\Gamma)
\]
of formal $R$-schemes. The composition map is now dual to the product on $\Gamma$. 

That $(\Spec R,\sch \Gamma)$ is a category scheme with this structure amounts to a direct translation between definitions.
\end{proof}

Fix a good $R$-cobialgebroid $\Gamma$. As algebras for the monad $\Gamma$ are equivalent to coalgebras for the comonad $\Gamma^\vee$, we may encode a $\Gamma$-module as an $R$-module equipped with a coaction $P\colon M\rightarrow\Gamma^\vee(M)\cong \Gamma^\vee\mathop{{}_s\cotimes}M$ satisfying the evident counity and coassociativity conditions. The coaction $P$ is left $R$-linear, where the left $R$-module structure on $\Gamma^\vee$ is through the target map $t\colon R\rightarrow \Gamma^\vee$.

By definition, an object of $\Ring_\Gamma^\heart$ is an $R$-ring $A$ equipped with a $\Gamma$-module structure satisfying the Cartan formulas encoded by the coproduct $\Gamma\rightarrow\Gamma\tins{l}{l}\Gamma$. When the $\Gamma$-module structure on $A$ is encoded by a coaction $P\colon A\rightarrow \Gamma^\vee\mathop{{}_s\cotimes} A$, this is equivalent to the condition that $P$ is a homomorphism of rings.

\begin{ex}\label{ex:ltring}
Let $\Gamma$ be the $R$-cobialgebroid of \cref{ex:ltkoszul} and \cref{ex:ltc}, also encountered in \cref{ex:ltep}, and suppose for simplicity that $\kappa = \bbF_2$ so that $R = \bbZ_2[[a]]$. The length grading $\Gamma = \bigoplus_{n\geq 0}\Gamma[n]$ is a decomposition of $R$-coalgebras, so as each $\Gamma[n]$ is finitely generated and free over $R$, the coalgebra $\Gamma$ is good, and
\[
\sch\Gamma \cong \coprod_{n\geq 0}\Spec \Gamma[n]^\vee.
\]

As $\Gamma$ is quadratic by definition (see \cref{ex:ltkoszul}), a $\Gamma$-module is determined by an $R$-module $M$ equipped with a left $R$-linear map $P\colon M\rightarrow \Gamma[1]^\vee\tins{s}{}M$ such that there exists a factorization through the dashed arrow in the diagram
\begin{center}\begin{tikzcd}
M\ar[r,"P"]\ar[d,dashed]&\Gamma[1]^\vee\tins{s}{}M\ar[d,"{\Gamma[1]^\vee\otimes P}"]\\
\Gamma[2]^\vee\tins{s}{}M\ar[r,"c\otimes M"]&\Gamma[1]^\vee\tins{s}{t}\Gamma[1]^\vee\tins{s}{}M
\end{tikzcd}.\end{center}

There is an isomorphism of rings
\[
\Gamma[1]^\vee \cong R[d]/(d^3 = a d + 2),
\]
where the basis $\Gamma[1]^\vee = \{1,d,d^2\}R$ is dual to the basis $\Gamma[1] = R\{Q_0,Q_1,Q_2\}$. The map $t$ is the ring homomorphism
\[
t\colon R\rightarrow \Gamma[1]^\vee, \qquad t(a) = a^2 + 3d -ad^2,
\]
and $R$-linearity of $P$ is with respect to $t$, i.e.\ $P(am) = P(m)(a^2 + 3d -ad^2)$ for $m\in M$.

The category $\Ring_\Gamma^\heart$ is more pleasantly understood from this perspective. If $A$ is an $R$-ring and $\Gamma$-module, with structure map $P\colon A\rightarrow \Gamma[1]^\vee\tins{s}{}A \cong A[d]/(d^3 = ad+2)$, then $A$ is a $\Gamma$-ring precisely when $P$ is a ring homomorphism.

It is possible to compute $\Gamma[2]^\vee$ directly given our knowledge of $\Gamma$. Doing so is backwards, as $\Gamma$ is computed in \cite{rezk2008power} by computing in the formal category scheme $(\Spf R,\sch \Gamma)$, whose interpretation we recall in \cref{ssec:epow}. Better yet, as explained in \cite{rezk2013power}, one may avoid directly dealing with $\Gamma[2]^\vee$ altogether. Write
\[
\Gamma[1]^\vee\tins{s}{t}\Gamma[1]^\vee\cong R[d',d]/(d^3 = ad+2,d'^3 = (a^2 + 3d -ad^2)d'+2).
\]
Then there is a Cartesian square
\begin{center}\begin{tikzcd}
\Gamma[2]^\vee\ar[r,"c"]\ar[d,"\epsilon"]&\Gamma[1]^\vee\tins{s}{t}\Gamma[1]^\vee\ar[d,"f"]\\
R\ar[r,"s"]&\Gamma[1]^\vee
\end{tikzcd}\end{center}
of rings, where $f$ is the right $R$-linear map given by $f(d) = d$ and $f(d') = a-d^2$. Thus an $R$-linear map $P\colon M\rightarrow\Gamma[1]^\vee\tins{s}{}M$ makes $M$ into a $\Gamma$-module precisely when there exists a (necessarily unique) map $\Psi$ filling in
\begin{center}\begin{tikzcd}
M\ar[r,"P"]\ar[dd,"\Psi",dashed]&\Gamma[1]^\vee\tins{s}{}M\ar[d,"{\Gamma[1]^\vee\otimes P}"]\\
&\Gamma[1]^\vee\tins{s}{t}\Gamma[1]^\vee\tins{s}{}M\ar[d,"f\otimes M"]\\
M\ar[r,"s\otimes M"]&\Gamma[1]^\vee\tins{s}{}M
\end{tikzcd}.\end{center}
When $M = R$, the map $\Psi$ happens to be the identity. This implies that $\Psi$ is $R$-linear in general, although it need not be the identity in general.

See in particular \cref{ex:ltcurve} for more on this example.
\tqed
\end{ex}

Let $\Mod^\heart\colon\CRing^\heart\rightarrow\Cat$ denote the pseudofunctor
\[
R\mapsto \Mod_R^\heart,\qquad (f\colon R\rightarrow S) \mapsto (S\otimes_R\bs\colon\Mod_R^\heart\rightarrow\Mod_S^\heart).
\]
Given some other pseudofunctor $\calX\colon \CRing^\heart\rightarrow\Cat$, define $\QCoh(\calX)^\heart$ to be the category of pseudonatural transformations $\calX^\op\rightarrow\Mod^\heart$. This is an additive symmetric monoidal category.

\begin{prop}
Let $\Gamma$ be a good $R$-cobialgebroid, and $\calX = (\Spec R,\sch \Gamma)$. Then there is an equivalence $\LMod_\Gamma^\heart\simeq\QCoh(\calX)^\heart$ of symmetric monoidal categories.
\end{prop}
\begin{proof}
There is a well-known symmetric monoidal equivalence between the category of comodules for a commutative Hopf algebroid and the category of quasicoherent sheaves on the associated presheaf of groupoids \cite{hovey2001morita}. The claim at hand is no different from this, so we will indicate the construction but omit detailed verifications of naturality.

The equivalence $\LMod_\Gamma^\heart\rightarrow\QCoh(\calX)^\heart$ is constructed as follows. Fix $M\in\LMod_\Gamma^\heart$, so we wish to construct a pseudonatural transformation $\calF_M\colon \calX^\op\rightarrow\Mod^\heart$. Fix a ring $S$; then the functor $\calF_M^S\colon\calX(S)^\op\rightarrow\Mod_S^\heart$ is defined as follows. Fix $g\in\calX(S)$ realized by a map $g\colon R\rightarrow S$. Then $\calF_M^S$ is given on objects by $\calF_M^S(g) = S\tins{g}{} M$. Fix $\alpha\colon g'\rightarrow g$ in $\calX(S)$, realized by a diagram
\begin{center}\begin{tikzcd}
&\Spec R\\
\Spec S\ar[ur,"g"]\ar[r,"\alpha"]\ar[dr,"g'"']&\sch \Gamma\ar[u,"t"']\ar[d,"s"]\\
&\Spec R
\end{tikzcd}.\end{center}
Then $\alpha$ is dual to a map $\alpha\colon \Gamma^\vee\rightarrow S$ of formal rings, i.e.\ one that factors through some discrete quotient $\Gamma_\alpha^\vee$, where $\Gamma_\alpha\subset\Gamma$ is a standard coalgebra. Then $\calF_M^S$ is given on morphisms by declaring $\calF_M^S(g)$ to be the composite
\begin{align*}
\calF_M(g) = S\tins{g}{}M\rightarrow S\tins{g}{}\Gamma^\vee(M)
&\cong S\tins{g}{t}\Gamma^\vee \mathop{{}_s\cotimes} M\\
&\rightarrow S \tins{g}{g} S\tins{g'}{} M
\rightarrow S\tins{g'}{}M = \calF_M(g').
\end{align*}

The inverse equivalence $\QCoh(\calX)^\heart\rightarrow\LMod_\Gamma^\heart$ is constructed as follows. Fix $\calF\colon\calX^\op\rightarrow\Mod^\heart$. Let $i\in\calX(R)$ be classified by the identity of $R$, and write $M = \calF_R(i)$. A $\Gamma^\vee$-comodule structure on $M$ can be defined as follows. Note first that $\calX$ extends to a functor on pro-rings in the evident way; in particular $\calX(\Gamma^\vee)$ is a category, and there are elements $s,t\in\calX(\Gamma^\vee)$ classified by the source and target maps of $\Gamma^\vee$. The identity map of $\Gamma^\vee$ corresponds to a map $c\colon s\rightarrow t$ in $\calX(\Gamma^\vee)$, and this gives a $\Gamma^\vee$-linear map
\[
\Gamma^\vee\mathop{{}_t\cotimes} M\rightarrow\Gamma^\vee\mathop{{}_s\cotimes} M \cong \Gamma^\vee(M).
\]
This is adjoint to a map $M\rightarrow\Gamma^\vee(M)$ which defines a $\Gamma$-module structure on $M$, and the inverse equivalence sends $\calF$ to $M$ with this $\Gamma$-module structure.
\end{proof}

\begin{ex}
Let $\sigma\colon R\rightarrow R$ be a ring homomorphism, and consider the $R$-cobialgebroid
\[
\Gamma = R \langle \psi \rangle / (\psi \cdot r = \sigma(r)\cdot\psi),\qquad \Delta^\times(\psi)=\psi\otimes\psi,\qquad\epsilon(\psi)=1;
\]
compare \cref{ex:theta}. Then 
\[
\Gamma\cong\bigoplus_{n\geq 0}R\{\psi^n\}
\]
with $\psi^n$ grouplike, so
\[
\sch\Gamma \cong\coprod_{n\geq 0}\Spec R.
\]
The target map on the $n$'th component is given by restriction along $\sigma^n$. The formal category scheme $\calX\colon\CRing^\heart\rightarrow\Cat$ obtained from this $R$-cobialgebroid sends a ring $S$ to the category $\calX(S)$ identified as follows. An object of $\calX(S)$ is a map $f\colon\Spec S\rightarrow\Spec R$. Given $f,f'\in\calX(S)$, a morphism $\alpha\colon f\rightarrow f'$ is a decomposition $\Spec S = \coprod_{0\leq n \ll \infty}\Spec(S_n)$ such that $f'|_{\Spec(S_n)} = (\sigma^n)^\ast f$. A quasicoherent sheaf on $\calX$ is an $R$-module equipped with a $\sigma$-semilinear homomorphism $\psi\colon M\rightarrow M$, i.e.\ an additive map such that $\psi(r\cdot m) = \sigma(r)\cdot\psi(m)$ for $r\in R$ and $m\in M$. This is a quasicoherent sheaf of rings if $M$ is a ring and $\psi$ is a ring homomorphism.
\tqed
\end{ex}

\subsection{Power operations for Morava \texorpdfstring{$E$}{E}-theory}\label{ssec:epow}

Let $\kappa$ be a perfect field of positive characteristic $p$, and $\bbG_0\rightarrow\Spec(\kappa)=X_0$ be a formal group of finite height $h$. Let $\bbG\rightarrow X$ be the universal Lubin-Tate deformation \cite{lubintate1966formal} of this formal group, and $E$ be the associated Lubin-Tate spectrum, also referred to as a Morava $E$-theory. Write $\frakm\subset E_0 = \calO_X$ for the maximal ideal. 

By the Goerss-Hopkins-Miller theorem \cite{goersshopkins2004moduli} \cite{goersshopkinsxxxxmoduli}, $E$ is a $K(h)$-local even-periodic $\bbE_\infty$ ring spectrum, and this construction is functorial in the input $(\kappa,\bbG_0)$ and fiberwise isomorphisms. There results a theory of $E$-power operations acting on the homotopy groups of $K(h)$-local $\bbE_\infty$ algebras over $E$, and these are now well-understood conceptually owing to work of Ando, Hopkins, Strickland, and Rezk. The formulation by Rezk \cite{rezk2009congruence}, building on work of Strickland \cite{strickland1998morava}, itself building on calculations of Kashiwabara \cite{kashiwabara2001brownpeterson}, is the most convenient approach for our purposes. It seems easiest, both for the writer and the reader, to collect what we need in one place, so we will summarize some of the structure of these operations in one big statement. 

Write $\widehat{\bbP}$ for the free $K(h)$-local $\bbE_\infty$ algebra monad on $\Mod_E$, so that there is a decomposition $\widehat{\bbP} = L_{K(h)}\bigoplus_{n\geq 0}\widehat{\bbP}_n$ with $\widehat{\bbP}_n M = L_{K(h)}M^{\otimes n}_{\h\Sigma_n}$. Write $\CAlg_E^\loc$ for the category of $K(h)$-local $\bbE_\infty$ algebras over $E$; we will abuse terminology and refer to these as \textit{$E$-algebras}. We will write $\otimes$ for any of $\otimes_{E}$, $\otimes_{E_\ast}$, and $\otimes_{E_0}$, leaving which we mean to context.

\begin{theorem}[\cite{rezk2009congruence}, \cite{rezk2017rings}]\label{thm:eops}
There is a monad $\bbT$ on the category of $E_\ast$-modules satisfying and determined by the following three items:
\begin{enumerate}
\item The functor $\bbT$ preserves filtered colimits and reflexive coequalizers.
\item There are natural maps $\bbT(M_\ast)\rightarrow\pi_\ast\widehat{\bbP}M$ for $M\in\Mod_E$ compatible with the monad structures on $\bbT$ and $\widehat{\bbP}$. In particular, the homotopy groups of any $A\in\CAlg_E^\loc$ naturally form a $\bbT$-algebra.
\item There is a decomposition $\bbT\cong\bigoplus_{n\geq 0}\bbT_n$ compatible with the summands $\widehat{\bbP}_n\subset\widehat{\bbP}$, and if $M$ is a finitely generated and free $E$-module then the map $\bbT_n(M_\ast)\rightarrow\pi_\ast\widehat{\bbP}_n M$ is an isomorphism.
\end{enumerate}
In addition,
\begin{enumerate}
\item[(4)] $\bbT$ is an exponential monad, thus an $E_\ast$-plethory, with exponential structure inherited from the natural equivalences $\widehat{\bbP}_n(M\oplus N)\simeq\bigoplus_{i+j=n} \widehat{\bbP}_i M \otimes \widehat{\bbP}_j N$.
\item[(5)]$\bbT$ is an even-periodic plethory, with suspension maps inherited from the natural maps $\Sigma\widehat{\bbP}_{n}M\rightarrow\widehat{\bbP}_{n}\Sigma M$ defined for $n>0$.
\item[(6)]$\bbT$ is smooth, in fact free, relative to alternating $E_0$-algebras (cf.\ \cref{ex:alternating}).
\end{enumerate}
Write $\Gamma = \Gamma(\bbT)_{0,0}\subset \bbT(E_\ast)_0$ for the ordinary $E_0$-cobialgebroid underlying the even-periodic cobialgebroid $\Gamma(\bbT)$.
\begin{enumerate}
\item[(7)]Let $\Gamma[n]$ denote the intersection of $\Gamma$ with $\bbT(E_\ast)_{p^n}$. Then $\Gamma = \bigoplus_{n\geq 0}\Gamma[n]$ is a graded algebra. Moreover, this is a decomposition of coalgebras, and each $\Gamma[n]$ is finitely generated and free as a left $E_0$-module. In particular, $\Gamma$ is a good $E_0$-cobialgebroid. Moreover, each $\Gamma[n]^\vee$ is a complete local ring with residue field $\kappa$.
\item[(8)]Let $\calX = (\Spec E_0,\sch\Gamma)$ be the formal category scheme associated to $\Gamma$, and let $\Def\subset\calX$ be the full subcategory spanned by $\Spf E_0$. In other words, $\Def$ is the formal category scheme with objects $\Spf E_0$ and morphisms $\coprod_{n\geq 0}\Spf \Gamma[n]^\vee$, where $\Gamma[n]^\vee$ is given its adic topology. Consider $\Def$ as a presheaf of categories on the category of formal schemes $Y$ such that $\calO_Y$ is a complete local ring equipped with its adic topology. Then $\Def(Y)$ is the category with
\begin{enumerate}
\item Objects: Deformations of $\bbG_0$ to $Y$. These can be summarized as diagrams of Cartesian squares
\begin{center}\begin{tikzcd}
\bbG_0\ar[d]&\bbH_0\ar[l,"\alpha"']\ar[d]\ar[r]&\bbH\ar[d]\\
X_0&Y_0\ar[l,"j"']\ar[r]&Y
\end{tikzcd},\end{center}
where $Y_0 = \Spec(\calO_Y/\frakm_Y)\subset Y$ is the special fiber of $Y$ and $\alpha$ is a group homomorphism. We will write these as $\langle \bbH,j,\alpha\rangle$, or just $\bbH$ when the remaining structure is understood.
\item Morphisms: A morphism $f\colon \langle\bbH,j,\alpha\rangle\rightarrow \langle\bbH',j',\alpha'\rangle$ in $\Def(Y)$ classified by a map landing in the connected component $\Spf\Gamma[n]^\vee\subset\sch\Gamma$ is a deformation of the $n$-fold Frobenius homomorphism of $\bbG_0$. These can be summarized as homomorphisms $f\colon \bbH\rightarrow\bbH'$ over $Y$ such that the diagram
\begin{center}\begin{tikzcd}
\bbG_0\ar[dr,"F^n"]\ar[dd]&&\bbH_0\ar[dr,"f_0"]\ar[rr]\ar[dd]\ar[ll,"\alpha"']&&\bbH\ar[dr,"f"]\ar[dd]\\
&\bbG_0\ar[dd]&&\bbH_0'\ar[dd]\ar[rr]\ar[ll,"\alpha'"', near end]&&\bbH'\ar[dd]\\
X_0\ar[dr,"\sigma^n"]&&Y_0\ar[dr,"="]\ar[ll,"j"', near start]\ar[rr]&&Y\ar[dr,"="]\\
&X_0&&Y_0\ar[ll,"j'"']\ar[rr]&&Y
\end{tikzcd}\end{center}
commutes. Here, $\sigma^n$ is the $n$-fold algebraic Frobenius and $F^n$ is the $n$-fold absolute Frobenius homomorphism.
\end{enumerate}
\item[(9)]Equivalently, $\Spf \Gamma[n]^\vee$ is the formal scheme $\Sub^n_\bbG$ classifying degree $p^n$ subgroups of $\bbG$. The target map $t\colon \Sub^n_\bbG\rightarrow X$ sends a degree $p^n$ subgroup $K\subset\bbH$, where $\bbH$ is a deformation of $\bbG_0$, to the quotient $\bbH/K$, considered as a deformation of $\bbG_0$ via the isomorphism $(\bbH/K)_0\cong \bbG_0/(\bbG_0[F^n])\cong (\sigma^n)^\ast\bbG_0$, where $\bbG_0[F^n]$ is the kernel of the $n$-fold Frobenius and the second equivalence is given by the $n$-fold relative Frobenius.
\end{enumerate}
Now,
\begin{enumerate}
\item[(10)]Write $\omega = \pi_2 E$, so that $\omega = \omega_\bbG$ is the module of invariant differentials on $\bbG$. Then the $\Gamma$-module structure on $\omega$, encoding what is necessary to pass from the ungraded cobialgebroid $\Gamma$ to the even-periodic cobialgebroid $\Gamma(\bbT)$, is given by the quasicoherent sheaf on $\Def$ sending a deformation $\bbH$ to the module of invariant differentials $\omega_\bbH$ on $\bbH$.
\end{enumerate}
Slightly abusing notation, write $\LMod_\Gamma^\heart$ for the category of graded modules over $\Gamma$, equivalent to the category of quasicoherent sheaves of graded modules on $\Def$, and write $\Ring_\Gamma^\heart$ for the category of alternating $\Gamma$-rings, equivalent to the category of quasicoherent sheaves of alternating rings on $\Def$.
\begin{enumerate}
\item[(11)]The restriction $\Ring_\bbT^\heart\rightarrow\Ring_\Gamma^\heart$ is fully faithful when restricted to the full subcategory of $p$-torsion free $\bbT$-rings. The essential image is spanned by those $p$-torsion free $\Gamma$-rings $B$ whose underlying ungraded $\Gamma$-ring $B_0$ satisfies either of the following equivalent congruence criteria:
\begin{enumerate}
\item Let $X_1 = \Spf(E_0/(p))$ and $\bbG_1 = X_1\times_X \bbG$. Consider the map $X_1\rightarrow \Spf \Gamma[1]^\vee \cong \Sub^1_\bbG$ classifying the kernel of the Frobenius on $\bbG$, and choose a lift of this to an $E_0$-linear map $\Gamma[1]^\vee\rightarrow E_0$. Dualize this to obtain an element $Q\in \Gamma[1]$ which is well defined mod $p$. Then $Q x \equiv x^p \pmod{p}$ for all $x\in B_0$.
\item Let $\calF$ be the quasicoherent sheaf of rings associated to $B_0$. Then for every deformation $\bbH$ of $\bbG_0$ to $Y$ with $p=0$ in $\calO_Y$, the diagram
\begin{center}\begin{tikzcd}
\calF_Y(\sigma^\ast\bbH)\ar[dr,"\calF_Y(F)"]\ar[dd,"\simeq"]&\\
&\calF_Y(\bbH)\\
\sigma^\ast\calF_Y(\bbH)\ar[ur,"\sigma"']
\end{tikzcd}\end{center}
commutes. Here, $F\colon \bbH\rightarrow\sigma^\ast\bbH$ is the relative Frobenius on $\bbH$, the left vertical isomorphism arises from pseudonaturality of $\calF$, and $\sigma$ is the algebraic Frobenius on the $\calO_Y$-ring $\calF_Y(\bbH)$.
\end{enumerate}
\end{enumerate}
Write $\Delta = \Delta(\bbT)_{0,0}$ and, slightly abusing notation, write $\LMod_\Delta$ for the category of graded modules over $\Delta$. So $\LMod_\Delta\simeq\LMod_\Gamma$ in the manner described in \cref{prop:epmorita}, and a choice of trivialization of $\omega_\bbG$ gives an isomorphism of algebras $\Delta\cong\Gamma$. Then
\begin{enumerate}
\item[(12)]The algebra $\Delta$ is graded compatibly with $\Gamma$, and both $\Gamma$ and $\Delta$ are Koszul $E_0$-algebras. Moreover, $H^n(\Delta) = 0$ for $n>h$. In particular, every $\Delta$-module which is projective (equivalently, free) over $E_0$ admits a length $h$ projective Koszul resolution.
\qed
\end{enumerate}
\end{theorem}

\begin{ex}\label{ex:ktheory}
The fundamental example is given when $\kappa = \bbF_p$ and $\bbG_0$ is the formal multiplicative group. The associated Lubin-Tate spectrum $E = KU_p$ is the $p$-completion of complex $K$-theory. In this case the full subcategory of $\Ring_\bbT^\heart$ spanned by those objects concentrated in even degrees is equivalent to the category of $\theta$-rings over $\bbZ_p$ (cf.\ \cref{ex:theta}).

The full category $\Ring_{\bbT}^\heart$ may be identified as follows. The $\Gamma = \bbZ_p[\psi]$-module $\omega = \pi_2 KU_p$ may be identified as $\bbZ_p\{\beta\}$ with action $\psi(\beta) = p\beta$. Following \cref{rmk:z2mult}, if $A\in\Ring_{\bbT}^\heart$ and $x,y\in A_{-1}$, then $\psi(xy) = p \psi(x)\psi(y)\in A_0$. In the generic case we may factor out this $p$, and in the end identify $\Ring_{\bbT}^\heart$ as the category of $\bbZ/(2)$-graded alternating rings $A$ over $\bbZ_p$ equipped with a $\theta$-ring structure on $A_0$ and an additive map $\psi\colon A_{-1}\rightarrow A_{-1}$, such that if $x\in A_0$ or $y\in A_0$, then $\psi(xy) = \psi(x)\psi(y)$, and if $x,y\in A_{-1}$, then $\theta(xy) = \psi(x)\psi(y)$.

See \cref{rmk:ht1} for a description of the general $K(1)$-local case.
\tqed
\end{ex}

\begin{ex}\label{ex:ltcurve}
Let $C_0$ be the elliptic curve over a perfect field $\kappa$ of characteristic $p=2$ with affine equation $v^2 + v = u^3$ and identity $(u,v) = (0,0)$. This is a supersingular elliptic curve with formal group $\bbG_0$, whose universal deformation $\bbG$ can be identified as the formal group associated to the elliptic curve $C$ over $R=W(\kappa)[[a]]$ with affine equation $v^2 + a uv + v = u^3$ and identity $(u,v)=(0,0)$; we choose $u$ as our preferred coordinate for this formal group. The structure of power operations for the resulting Lubin-Tate spectrum have been calculated by Rezk \cite{rezk2008power}, and we have recalled the structure of the associated even-periodic cobialgebroid in Examples \ref{ex:ltkoszul}, \ref{ex:ltc}, \ref{ex:ltep}, and \ref{ex:ltring}.

A congruence element of $\Gamma[1]$ allowing us to recover the full category of $\bbT$-rings is given by $Q_0$. Thus if $A$ is a $2$-torsion free $\Gamma$-ring, with $\Gamma$-ring structure on $A_0$ encoded by a map $P\colon A_0\rightarrow A_0[d]/(d^3 = ad+2)$, then $A$ is the underlying $\Gamma$-ring of a $\bbT$-ring, necessarily uniquely, if and only if $P(x)\equiv x^2\pmod{d}$ for all $x\in A_0$.

The operation $Q_0$ generically decomposes as $Q_0(x) = x^2 + 2 \theta(x)$ for some $\theta\in \bbT(E_\ast)_0$, and $\bbT(E_\ast)_0$ is a polynomial ring on certain iterates of $\theta$, $Q_1$, and $Q_2$. The algebra $\Delta = Q(\bbT(E_\ast)_0)$ is generated by $\theta$, $Q_1$, and $Q_2$, subject those relations seen in $\Gamma$ among $Q_1$ and $Q_2$, as well as
\begin{align*}
\theta a &= a^2 \theta - a Q_1 + 3 Q_2\\
Q_1\theta &= Q_2 Q_1 - 2 \theta Q_2,\\
Q_2\theta &= \theta Q_1 + a \theta Q_2 - Q_1 Q_2.
\end{align*}
The composite $\Gamma\rightarrow\bbT(E_\ast)_0\rightarrow\Delta$ is
\[
Q_0\mapsto 2\theta,\qquad Q_1\mapsto Q_1,\qquad Q_2\mapsto Q_2.
\]
The suspension isomorphism $\Delta\rightarrow\Gamma$ is
\[
\theta\mapsto -Q_2,\qquad Q_1\mapsto -Q_0 - a Q_2,\qquad Q_2 \mapsto - Q_1.
\]

If $M$ is a $\Gamma$-module encoded by a coaction $P\colon M\rightarrow \Gamma[1]^\vee\tins{s}{}M$, then $\omega\otimes M = M$ as $R$-modules, with $\Gamma$-module structure encoded by $-dP\colon M\rightarrow\Gamma[1]^\vee\tins{s}{}M$. If $A$ is an augmented $\bbT$-ring, then $Q(A)$ inherits the structure of a $\Gamma$-module, and the Frobenius congruence implies that the image of $P\colon Q(A)\rightarrow\Gamma[1]^\vee\tins{s}{}Q(A)$ is divisible by $d$. If $Q(A)$ is torsion-free, then $\frac{-1}{d}P\colon Q(A)\rightarrow \Gamma[1]^\vee\tins{s}{}Q(A)$ defines a $\Gamma$-module, written $\omega^{-1/2}\otimes M$. To be precise, the underlying graded $R$-module of $\omega^{-1/2}\otimes M$ differs from $M$ by a shift in degrees. With these definitions, $\omega^{-1/2}\otimes M$ is a model for the image of the $\Delta$-module $Q(A)$ under the Morita equivalence $\LMod_\Delta\simeq\LMod_\Gamma$ of \cref{prop:epmorita}. Similar remarks are available for arbitrary Lubin-Tate spectra.
\tqed
\end{ex}

\begin{rmk}\label{ex:l2koszul}
Suppose that $\bbG_0$ is a formal group of height $2$. Then the following description of $H^\ast(\Gamma)$ is given in \cite{rezk2013power}. First, there is a commutative diagram
\begin{center}\begin{tikzcd}
E_0\ar[r,"t"]\ar[d,"t"]\ar[dd,"\Psi"',bend right]&\Gamma[1]^\vee\ar[d,"{\Gamma[1]^\vee\otimes t}"]\\
\Gamma[2]^\vee\ar[r,"c"]\ar[d,"q"]\ar[r]&\Gamma[1]^\vee\tins{s}{t}\Gamma[1]^\vee\ar[d,"f"]\\
E_0\ar[r,"s"]&\Gamma[1]^\vee
\end{tikzcd},\end{center}
the bottom square of which is Cartesian. Here, $s$, $t$, and $c$ are part of the structure of the cobialgebroid $\Gamma$. As $\Gamma[1]^\vee$ classifies rank $p$ subgroups $H$, the tensor product $\Gamma[1]^\vee\tins{s}{t}\Gamma[1]^\vee$ classifies chains $H_0\subset H_1$ where both $H_0$ and $H_1/H_0$ are of rank $p$. It is the remaining maps which are special to height $2$; the map $q$ classifies the $p$-torsion subgroup $\bbG[p]$ and the map $f$ classifies the chain $H\subset \bbG[p]$, where $H$ is the universal rank $p$ subgroup defined over $\Gamma[1]^\vee$.

In particular, $\Psi$ is the automorphism of $E_0$ classifying the deformation $\bbG/\bbG[p]$. For the Lubin-Tate spectrum of \cref{ex:ltcurve}, this square is described explicitly in \cref{ex:ltring}; in this example $\Psi$ is the identity provided $\kappa\subset\bbF_4$.

Quadraticity of $\Gamma$ implies that $H^0(\Gamma) = E_0$ and $H^1(\Gamma) = \Gamma[1]^\vee$. As the bottom square of the above diagram is Cartesian, we may moreover identify
\[
H^2(\Gamma) = \coker(c\colon \Gamma[2]^\vee\rightarrow\Gamma[1]^\vee\tins{s}{t}\Gamma[1]^\vee)\cong \coker(s\colon E_0\rightarrow\Gamma[1]^\vee).
\]
All higher cohomology groups vanish. Multiplication $H^1(\Gamma)\otimes H^1(\Gamma)\rightarrow H^2(\Gamma)$ is the composite
\[
\Gamma[1]^\vee\tins{s}{t}\Gamma[1]^\vee\rightarrow\Gamma[1]^\vee\rightarrow \Gamma[1]^\vee/s(E_0) \cong H^2(\Gamma),
\]
where the first map is $f$. The right $E_0$-module structure on $H^2(\Gamma)$ is through $s$, and the left $E_0$-module structure twists this by $\Psi$, i.e.\ $a\cdot x = x\cdot s(\Psi(a))$ for $a\in E_0$ and $x\in H^2(\Gamma)$.

Koszul complexes computing $\Ext_\Gamma$ are readily obtained from this; see \cref{ex:su3} for an explicit example.
\tqed
\end{rmk}

We would like to apply our understanding of algebraic structures such as $\bbT$ to obstruction-theoretic machinery for computing with $E$-algebras. Here, one runs into the subtlety that $\bbT$ does not perfectly encode the structure of all operations that act on the homotopy groups of $E$-algebras: the map $\bbT(\pi_\ast F)\rightarrow \pi_\ast\widehat{\bbP} F$ is not an isomorphism for $F\in\Mod_E^\free$. The missing piece is that the $K(h)$-local condition on our $E$-algebras enforces a certain completeness condition on their homotopy groups, and this is not seen by $\bbT$.

Write $\calA_\frakm$ for the $0$'th left-derived functor of $\frakm$-adic completion on $\Mod_{E_\ast}^\heart$. This is a localization, and we will call the $\calA_\frakm$-local objects \textit{$\frakm$-complete}, and denote the category of $\calA_\frakm$-local objects by $\Mod_{E_\ast}^{\Cpl(\frakm),\heart}$. Although this is distinct from the classic notion of $\frakm$-adic completeness, we will not use the classic notion, and so minimal confusion should arise. The functor $\calA_\frakm$ has been studied in \cite[Appendix A]{hoveystrickland1999morava} under the name of $L$-completion, in \cite{rezk2018analytic} under the name of analytic completion, and in other places by other names; when $\frakm = (p)$, this is ${\Ext}\hyp p$-completion in the sense of \cite[Section VI.2.1]{bousfieldkan1972monster}, as we have encountered in \cref{ex:extcompl}; in the particular context of theories, these topics were studied in \cite[Section 6]{balderrama2021deformations}. We will recall what we require in \cref{ssec:completions} below.

As implicitly noted in \cite[Section 1.6]{rezk2009congruence}, it is not immediately obvious that one may combine $\frakm$-completeness with the monad $\bbT$ to obtain a well-behaved category of $\frakm$-complete $\bbT$-rings. This issue was originally resolved by Barthel--Frankland \cite{barthelfrankland2015completed}, who show that the composite $\calA_{\frakm}\bbT$ admits the structure of a monad, yielding a completed form of $\bbT$. An alternate approach to completing $\bbT$ was carried out by Brantner in \cite[Definition 3.2.21]{brantner2017lubin}, using more involved $\infty$-categorical techniques.

It turns out that less work is necessary. Write $\CAlg_E^{\loc,\free}$ for the category of ($K(h)$-local $\bbE_\infty$) $E$-algebras which are free on a free $E$-module. The homotopy category $\h\CAlg_E^{\loc,\free}$ is a discrete theory whose category of discrete models is monadic over $\Mod_{E_\ast}^\heart$; write the associated monad as $\widehat{\bbT}$. The category $\Alg_{\widehat{\bbT}}$ is then a candidate for a category of $\frakm$-complete $\bbT$-rings. That it is the correct category turns out to follow from material already in \cite{rezk2009congruence}, together with some basic facts about localizations of monads. This is an instance of a general philosophy we follow, that it can be easier to construct the category of algebras for a monad than it is to construct the monad itself.

\begin{prop}\label{prop:dfcpl}
The forgetful functor $\Model_{\h\CAlg_E^{\loc,\free}}^\heart\rightarrow\Ring_\bbT^\heart$ is fully faithful, with essential image spanned by those $\bbT$-rings whose underlying $E_\ast$-module is $\frakm$-complete. In particular, $\widehat{\bbT}$ is a plethory for the theory of $\frakm$-complete $E_\ast$-modules.
\end{prop}
\begin{proof}
There is by construction a map $\bbT\rightarrow\widehat{\bbT}$ of monads on $\Mod_{E_\ast}^\heart$. As $\widehat{\bbT}$ takes values in $\frakm$-complete modules, as a map of functors this factors as $\bbT\rightarrow\calA_\frakm\bbT\rightarrow\widehat{\bbT}$. By \cite[Lemma 6.1.3, Proposition 6.1.4]{balderrama2021deformations}, it is sufficient to verify that $\calA_\frakm\bbT\rightarrow\widehat{\bbT}$ is an isomorphism of functors. As both source and target preserve geometric realizations, it is sufficient to verify that $\calA_\frakm\bbT(F_\ast)\rightarrow \widehat{\bbT}(F_\ast)$ is an isomorphism when $F_\ast$ is a free $E_\ast$-module. Fix such $F_\ast$, and write $F_\ast = \pi_\ast F$ for a free $E$-module $F$. By the construction of $\bbT$ and $\widehat{\bbT}$, there is a commutative diagram
\begin{center}\begin{tikzcd}
\bbT(F_\ast)\ar[d]\ar[r]&\calA_\frakm\bbT(F_\ast)\ar[r]\ar[d,"\cong"]&\widehat{\bbT}(F_\ast)\ar[d,"\cong"]\\
\pi_\ast\widehat{\bbP}F\ar[r,"\cong"]&\pi_\ast\widehat{\bbP}F\ar[r,"\cong"]&\pi_\ast\widehat{\bbP}F
\end{tikzcd}.\end{center}
Here, the right vertical map is an isomorphism by construction, and the middle vertical map is an isomorphism as $\bbT(F_\ast)$ is free \cite[Proposition 4.17]{rezk2009congruence}. Thus the top right horizontal map is an isomorphism, and this proves the proposition.
\end{proof}

\begin{rmk}\label{rmk:ht1}
The abstract construction of $\widehat{\bbT}$ certainly does not rely on $E$ being a Lubin-Tate spectrum, and with some work the algebraic constructions of \cref{thm:eops} can also be extended to more general $K(h)$-local even-periodic $\bbE_\infty$ ring spectra. To set this up correctly would take us too far afield, so we will not do so here. However, let us note how the algebraic story plays out at height $h=1$.

As discussed in \cite{hopkins2014k1}, the transfer yields an equivalence $L_{K(1)}\Sigma^\infty B\Sigma_p\simeq \bbS_{K(1)}$, using this one can define an operation $\theta\in \pi_0 L_{K(1)}\Sigma^\infty_+ B\Sigma_p$ making $\pi_0$ of an arbitrary $K(1)$-local $\bbE_\infty$ ring spectrum into a $\theta$-ring, and in fact if $R$ is a $K(1)$-local $\bbE_\infty$ ring, then $\pi_0 L_{K(1)} \bbP_R R$ is the free ${\Ext}\hyp p$-complete $\theta$-ring on $\pi_0 R$. If $R$ is even-periodic, then this splitting and identification extends to nonzero degrees. It follows that $\h\CAlg_R^{\loc,\free}$ is a theory of ${\Ext}\hyp p$-complete $\bbZ/(2)$-graded $\theta$-rings equipped with a map from $R_\star$. 

To be precise, the correct notion of a ``$\bbZ/(2)$-graded $\theta$-ring'' must incorporate the $\bbZ_p[\psi]$-module structure on $\omega=\pi_2 R$, in the same manner as it was incorporated in \cref{ex:ktheory}. This plays out as follows. Under the suspension map $R_0^\wedge \bbP_p \bbS\rightarrow R_2^\wedge\bbP_p\Sigma^2 \bbS$, the operation $\theta$ is sent to some additive operation $\frac{1}{p}\psi\colon \pi_2 \rightarrow\pi_2 $. Now the category of models of $\h\CAlg_R^{\loc,\free}$ is equivalent to the category of pairs $(A_0,A_{-1})$ where: first, $A_0$ is an ${\Ext}\hyp p$-complete $\theta$-ring under $R_0$; second, $A_{-1}$ is an ${\Ext}\hyp p$-complete $A_0$-module and $\bbZ[\psi]$-module satisfying $\psi(a_0\cdot a_{-1}) = \psi(a_0)\cdot\psi(a_{-1})$ for $a_0\in A_0$ and $a_{-1}\in A_{-1}$, where $\psi$ is the operation on $A_0$ defined by $\psi(a_0) = a_0^p + p \theta(a_0)$; and third, there is a suitably alternating and associative multiplication $m\colon \omega\otimes A_{-1}\otimes A_{-1}\rightarrow A_0$ satisfying $\theta(m(u\otimes a \otimes a')) = m(\frac{1}{p}\psi(u)\otimes\psi(a)\otimes\psi(a'))$.
\tqed
\end{rmk}

We end this subsection by pointing to where one can find some computations of the structure of $E$-power operations. The height $h=1$ case is as covered in \cref{rmk:ht1}, and explicit computations at heights $h\geq 3$ are not currently feasible, so we are left with height $h=2$, where computations are made possible by the theory of elliptic curves.

The first full explicit computation in this setting is the computation at $p=2$ of Rezk \cite{rezk2008power} recalled in \cref{ex:ltcurve}. Further computations at $p=2$ have been carried out by Schumann \cite{schumann2014k2}, allowing for elliptic curves with any Weierstra{\ss} equation of the form $y^2 + a_1 xy + a_3 y = x^3 + a_2 x^2 + a_4 x$, i.e.\ with $a_6 = 0$. Notably, this work gives a closed-form description of the total power operation on $E^0 BU(1) \cong E_0[[u]]$; for the Lubin-Tate spectrum of \cref{ex:ltcurve}, this is the ring map
\[
P\colon E_0[[u]]\rightarrow E_0[[u]][d]/(d^3 - ad - 2),\qquad P(u) = \frac{u^2-du}{1+d^2 u}.
\]

At $p=3$, for Lubin-Tate spectra $E$ associated to certain elliptic curves, $E$-power operations have been computed by Nendorf \cite{nendorf2012power} and by Zhu \cite{zhu2014power}. The latter also discusses the power operation structure on $L_{K(1)} E$ for the height $h=2$ Lubin-Tate spectrum $E$ in question. Further work of Zhu in \cite{zhu2019semistable} gives a recipe that works for arbitrary primes.

We point also to \cite{rezk2013power}, which contains a wealth of information at heights $h\leq 2$, and in particular a number of computations in the cohomology of $\bbT$-rings at heights $h\leq 2$.

\subsection{Completions}\label{ssec:completions}

Fix notation as in the preceding section. Following \cref{prop:dfcpl}, we are interested in the homotopy theory of certain completed contexts. Some of the general interaction between theories and completions was studied in \cite[Section 6]{balderrama2021deformations}; here we explain how things fit together in the context of $\bbT$-rings.

Let us first recall the definitions in our particular context. Given $M\in\Mod_{E_\ast}$, say that $M$ is \textit{$\frakm$-nilpotent} if $M[x^{-1}]=0$ for all $x\in \frakm$, is \textit{$\frakm$-local} if $\Map(N,M)\simeq \ast$ for all $\frakm$-nilpotent $N$, and is \textit{$\frakm$-complete} if $\Map(N,M)\simeq\ast$ for all $\frakm$-local $N$. The full subcategory $\Mod_{E_\ast}^{\Cpl(\frakm)}\subset\Mod_{E_\ast}$ of $\frakm$-complete modules is a reflective subcategory, and an explicit formula for the reflection $(\bs)_\frakm^\wedge$ is given as follows. Choose generators $u_0,\ldots,u_{h-1}\in \frakm$ and fix $M\in\Mod_{E_\ast}$. Then $M_\frakm^\wedge$ is the total cofiber of the $h$-cube obtained as the external product of the $1$-cubes $T_i-u_i\colon M[[T_i]]\rightarrow M[[T_i]]$. If $M\in\Mod_{E_\ast}^\heart$, then $\calA_\frakm M = \pi_0 (M_\frakm^\wedge)$.

Observe that the preceding definitions can be applied equally well in any linear setting in which there is a natural action by the elements of $\frakm$. In particular, they apply to $\Mod_E$, where $\frakm$-completion coincides with $K(h)$-localization. In general, if $\calM$ is some category over $\Mod_E$ or $\Mod_{E_\ast}$, we will write $\calM^{\Cpl(\frakm)}\subset\calM$ for the full subcategory spanned by those objects whose underlying module is $\frakm$-complete. Thus for instance \cref{prop:dfcpl} tells us that 
\[
\Model_{\h\CAlg_E^{\loc,\free}}^\heart\simeq\Ring_\bbT^{\heart,\Cpl(\frakm)}.
\]
The key fact that allows us to handle completions is that this continues to hold at the level of simplicial rings.

\begin{theorem}\label{lem:derivedcompleterefl}
There is an equivalence $\Model_{\h\CAlg_E^{\loc,\free}}\simeq\Ring_\bbT^{\Cpl(\frakm)}$.
\end{theorem}
\begin{proof}
This follows from \cite[Proposition 6.1.4]{balderrama2021deformations} and \cref{prop:dfcpl}, as $\calA_\frakm \bbT(F_\ast)\simeq\bbT(F_\ast)_\frakm^\wedge$ for $F_\ast\in\Mod_{E_\ast}^\free$ by tameness of $\bbT(F_\ast)$.
\end{proof}

In particular, given $R\in\Ring_\bbT^{\Cpl(\frakm),\heart}$, $S\in\Ring_{R\otimes\bbT}^{\Cpl(\frakm),\heart}$, $M\in \Ab(\Ring_{R\otimes\bbT/S}^{\Cpl(\frakm),\heart})\simeq\LMod_{S\otimes\Delta}^{\Cpl(\frakm),\heart}$, and $A\in\Ring_{R\otimes\bbT/S}^{\Cpl(\frakm),\heart}$, all the following spaces are equivalent:
\begin{align*}
\Map_{R/\Ring_\bbT^{\Cpl(\frakm)}/S}(A,S&\ltimes B^n M)\simeq \Map_{R\otimes\bbT/S}(A,S\ltimes B^n M)\simeq\calH^n_{R\otimes\bbT/S}(A;M)\\
&\simeq \EXT^n_{S\otimes\Delta}(S\otimes_A^\bbL \Omega_{A|R},M)\simeq\EXT^n_{S\otimes\Delta}(S\cotimes_A ^\bbL\bbL\widehat{\Omega}_{A|R},M).
\end{align*}
Because of this, we will generally write things in terms of $\bbT$, although there is a sense in which $\widehat{\bbT}$ is more fundamental in our setting.

The primary subtlety of completions relevant to us is that \cref{lem:derivedcompleterefl} does not extend to all settings. For example, if $R$ is an $E_\ast$-ring, then there is a category $\Mod_R^{\Cpl(\frakm)}$ of $\frakm$-complete $R$-modules, and $\Mod_R^{\Cpl(\frakm)}\simeq\LMod_\calP$ where $\calP\subset\Mod_R$ is the full subcategory spanned by the (derived) $\frakm$-completions of free $R$-modules. But the failure of coproducts to be exact in general \cite[Appendix B]{baker2009lcomplete} can force this theory $\calP$ to be non-discrete, and in particular $\Mod_R^{\Cpl(\frakm)}$ need not be the derived category of $\Mod_R^{\Cpl(\frakm),\heart}$. 

Say that $M$ is \textit{tame} if $(M^{\oplus I})_\frakm^\wedge$ is discrete for any set $I$; it is sufficient to consider the case $I = \omega$. Then most issues with completions vanish so long as we build on tame objects. For example, if $R\in\CAlg_E^\loc$ with $R_\ast$ tame, then $\pi_\ast \widehat{\bbP}_R (R\cotimes F)\simeq \calA_\frakm (R_\ast\otimes\bbT(F_\ast))$ for $F\in\Mod_E^\free$, and there is an equivalence $\Model_{\h\CAlg_R^{\loc,\free}}\simeq \Ring_{R_\ast\otimes\bbT}^{\Cpl(\frakm)}$.

We end by noting the following.

\begin{lemma}\label{lem:compkoszul}
Fix $R\in\Ring_\bbT^{\Cpl(\frakm),\heart}$, $S\in\Ring_{R\otimes\bbT}^\heart$, $M\in\LMod_{S\otimes\Delta}^{\Cpl(\frakm),\heart}$, and $A\in \Ring_{R\otimes\bbT/S}^{\Cpl(\frakm),\heart}$. If $A$ is smooth as an $\frakm$-complete alternating $R$-ring in the sense that $\bbL\widehat{\Omega}_{A_\ast|R_\ast}$ is the completion of a discrete projective $A_\ast$-module, then $H^n_{R\otimes\bbT/S}(A;M) = 0$ for $n > h$.
\end{lemma}
\begin{proof}
More generally, choose $N\in\LMod_{S\otimes\Delta}^{\Cpl(\frakm)}$ whose underlying $S$-module is the completion of a projective $S$-module; we claim that $\Ext_{S\otimes\Delta}^n(N,M) = 0$ for $n > h$. In the case where $S = E_\ast$ and $N$ is the completion of a $\Delta$-module whose underlying $E_\ast$-module is projective, this is a consequence of part (12) of \cref{thm:eops}. In general, consider  the diagram
\begin{center}\begin{tikzcd}
\LMod_{S\otimes\Delta}^{\Cpl(\frakm),\heart}\ar[r]\ar[d]&\Mod_S^{\Cpl(\frakm),\heart}\ar[d]\\
\LMod_{S\otimes\Delta}^\heart\ar[r]\ar[d]&\Mod_S^\heart\ar[d]\\
\LMod_\Delta^\heart\ar[r]&\Mod_{E_\ast}^\heart
\end{tikzcd}.\end{center}
Write $\LMod_{S\otimes\Delta}^{\Cpl(\frakm),\heart}\simeq\LMod_F^\heart$, where $F$ is an algebra over $\Mod_S^{\Cpl(\frakm),\heart}$. Each square in the above is distributive, so by \cref{prop:koszulcomposition} the algebra $F$ is Koszul, with length $h$ Koszul resolutions. Though $N$ may not be discrete if $S$ is not tame, we may nonetheless apply \cref{lem:bares} to identify $\EXT_{S\otimes\Delta}(N,M)$ as the totalization of $\EXT_{S\otimes\Delta}(B(S\otimes\Delta,S\otimes\Delta,N),M)$. As $M$ is $\frakm$-complete and discrete, there is an equivalence $B_{S\otimes\Delta}(N,M)\simeq B_F(\pi_0 N,M)$. As $\pi_0 N$ is a projective object of $\Mod_S^{\Cpl(\frakm),\heart}$, there is a quasiisomorphism $B_F(\pi_0 N,M)\simeq K_F(\pi_0 N,M)$ by Koszulity. As $K_F(\pi_0 N,M)$ is a length $h$ complex, this proves the lemma.
\end{proof}

\subsection{Mapping spaces and highly structured orientations}\label{ssec:emaps}

We can now describe some applications of the preceding theory. Fix notation as in the preceding subsections; in particular $E$ is a Lubin-Tate spectrum of height $h$.

\begin{theorem}\label{thm:emapss}
Fix $R\in\CAlg_E^\loc$, and choose $S\in\CAlg_R^\loc$ such that $R_\ast\rightarrow S_\ast$ is surjective (such as $S=0$ or $S=R$). Fix $A,B\in\CAlg_{R/S}^\loc$, and choose a map $\phi\colon A_\ast\rightarrow B_\ast$ in $\Ring_{R_\ast\otimes\bbT/S_\ast}$. Let $\CAlg_{R/S}^\phi(A,B)$ be the space of lifts of $\phi$ to a map in $\CAlg_{R/S}$. Then there is a decomposition
\[
\CAlg_{R/S}^\phi(A,B)\simeq\lim_{n\rightarrow\infty}\CAlg_{R/S}^{\phi,\leq n}(A,B),
\]
with layers fitting into fiber sequences
\[
\CAlg_{R/S}^{\phi,\leq n}(A,B)\rightarrow\CAlg_{R/S}^{\phi,\leq n-1}(A,B)\rightarrow\calH^{n+1}_{R_\ast\otimes\bbT/B_\ast}(A_\ast;\pi_\ast\Omega^n F),
\]
where $F = \Fib(B\rightarrow S)$. In particular,
\begin{enumerate}
\item There are successively defined obstructions in $H^{n+1}_{R_\ast\otimes\bbT/B_\ast}(A_\ast;\pi_\ast\Omega^n F)$ for $n\geq 1$ to exhibiting a point of $\CAlg_{R/S}^\phi(A,B)$;
\item Once a point of $\CAlg_{R/S}^\phi(A,B)$ is chosen, there is a fringed spectral sequence of signature
\[
E_1^{p,q} = H^{p-q}_{R_\ast\otimes\bbT/B_\ast}(A_\ast;\pi_\ast\Omega^p F)\Rightarrow \pi_q(\CAlg_{R/S}(A,B),f),\quad d_r^{p,q}\colon E_r^{p,q}\rightarrow E_r^{p+r,q-1}.
\]
\end{enumerate}
Specializing further, if $A_\ast$ is smooth as an $\frakm$-complete alternating $R_\ast$-ring, then
\begin{enumerate}[resume]
\item If $h=1$, then $\CAlg_{R/S}^\phi(A,B)$ is nonempty. Moreover, 
\[
\pi_0\CAlg_{R/S}^\phi(A,B) \cong H^1_{R_\ast\otimes\bbT/B_\ast}(A_\ast;\pi_\ast\Omega F),
\]
and if we choose $f\in\CAlg_{R/S}^\phi(A,B)$, then there are short exact sequences
\begin{align*}
0\rightarrow H^1_{R_\ast\otimes\bbT/B_\ast}(A_\ast;\pi_\ast\Omega^{n+1}F)&\rightarrow \pi_n(\CAlg_{R/S}(A,B),f)\\
&\rightarrow H^0_{R_\ast\otimes\bbT/B_\ast}(A_\ast;\pi_\ast\Omega^n F)\rightarrow 0
\end{align*}
for $n\geq 1$.
\item If $h=1$ and each of the rings in question is concentrated in even degrees, then $\CAlg_{R/S}^\phi(A,B)$ is connected. Moreover,
\[
\pi_{2n-\epsilon}\CAlg_{R/S}^{\phi}(A,B)\cong H^\epsilon_{R_\ast\otimes\bbT/B_\ast}(A_\ast;\pi_\ast\Omega^{2n}F)
\]
for $n\geq 1$ and $\epsilon\in\{0,1\}$.
\item If $h=2$ and each of the rings in question is concentrated in even degrees, then $\CAlg_{R/S}^\phi(A,B)$ is nonempty. Moreover, $\pi_0 \CAlg_{R/S}^\phi(A,B) \cong H^2_{R_\ast\otimes\bbT/B_\ast}(A_\ast;\pi_\ast\Omega^2 F)$, and if we choose $f\in\CAlg_{R/S}^\phi(A,B)$, then there are short exact sequences
\begin{align*}
0\rightarrow H^2_{R_\ast\otimes\bbT/B_\ast}(A_\ast;\pi_\ast\Omega^{2(n+1)}F)&\rightarrow\pi_{2n}(\CAlg_{R/S}(A,B),f)\\
&\rightarrow H^0_{R\ast\otimes\bbT/B_\ast}(A_\ast;\pi_\ast\Omega^{2n}F)\rightarrow 0
\end{align*}
and isomorphisms
\[
\pi_{2n-1}(\CAlg_{R/S}(A,B),f)\cong H^1_{R_\ast\otimes\bbT/B_\ast}(A_\ast;\pi_\ast\Omega^{2n}F)
\]
for $n\geq 1$.
\end{enumerate}
\end{theorem}
\begin{proof}
The obstruction theory is an application of \cite[Theorem 5.3.1]{balderrama2021deformations}, just as in the proof of \cref{thm:pmaps}.
The final statements then follow using \cref{lem:compkoszul}.
\end{proof}

\begin{rmk}
Following \cref{rmk:ht1}, the preceding theorem also applies when $E$ is instead taken to be an arbitrary $K(1)$-local even-periodic $\bbE_\infty$ ring spectrum.
\tqed
\end{rmk}

Our main application of \cref{thm:emapss} is to the theory of $\bbE_\infty$ orientations. We first recall some history. Power operations for Lubin-Tate spectra were first studied by Ando \cite{ando1995isogenies} precisely in the context of producing highly structured complex orientations. In particular, there it is shown that the Honda formal group law refines to a unique $\bbH_\infty$ orientation of its associated Lubin-Tate spectrum. The characterization of $\bbH_\infty$ orientations is described in a more general setting in Ando-Hopkins-Strickland \cite{andohopkinsstrickland2004sigma}, which in addition transitions to explicitly considering $MUP$ orientations, where $MUP$ is the Thom spectrum of the tautological bundle over $\bbZ\times BU$. In brief, homotopy ring maps $MUP\rightarrow E$ correspond to coordinates on the formal group $\bbG$, and the conditions necessary for this coordinate to correspond to a map of $\bbH_\infty$ rings can be described algebraically; we will call these coordinates \textit{norm-coherent}, and will very briefly review the characterization in the proof of \cref{thm:orientations}. Work of Zhu \cite{zhu2020norm} extends the existence and uniqueness of norm-coherent coordinates to an arbitrary Lubin-Tate spectrum so long as $\bbF_p\subset\kappa$ is algebraic; in our language, this work says that the first map in
\[
\Ring_\bbT(E_0^\wedge MUP,E_0)\rightarrow \Ring_{E_0}(E_0^\wedge MUP,\kappa)\cong\Coord(\bbG_0)
\]
is an isomorphism. Recall that at height $h=1$, the category of even $\bbT$-rings is exactly the category of $\theta$-rings sliced under $E_0$. Here it is classical that $E_0 = W(\kappa)$ is in fact the cofree $\theta$-ring on the $E_0$-ring $\kappa$, and so the above isomorphism is immediate; moreover this requires only that $\kappa$ is perfect. An unpublished theorem of Rezk extends this to arbitrary heights, showing that $E_0$ is always the cofree $\bbT$-ring on the $E_0$-ring $\kappa$.

Given the preceding, we can safely say that $\bbH_\infty$ orientations are well-understood. By contrast, significantly less is known about $\bbE_\infty$ orientations. The exception to this is $\bbE_\infty$ orientations by $MU$ at height $h=1$; the case of $p$-adic $K$-theory has been studied by Walker \cite{walker2009orientations}, and the more general $K(1)$-local case by M\"ollers \cite{mollers2012k1}, using methods similar to those employed in \cite{andohopkinsrezk2006multiplicative}; in short, in the $K(1)$-local context, every $\bbH_\infty$ orientation refines uniquely to an $\bbE_\infty$ orientation. In Hopkins-Lawson \cite{hopkinslawson2018strictly}, a general obstruction theory for $\bbE_\infty$ orientations by $MU$ is constructed that recovers the known $h=1$ story. Even less is known about $\bbE_\infty$ orientations by $MUP$. The only work in this direction we are aware of is \cite{hahnyuan2019exotic}, which demonstrates their existence when $h=1$ and $\kappa = \bbF_2$. Our contribution to this story is the following.
\begin{theorem}\label{thm:orientations}
\hphantom{blank}
\begin{enumerate}
\item Let $R$ be a $K(1)$-local even-periodic $\bbE_\infty$ ring spectrum. Then every norm-coherent coordinate on the formal group associated to $R$ refines uniquely to an $\bbE_\infty$ orientation $MUP\rightarrow R$.
\item The multiplicative formal group law $x+y-xy$ refines uniquely to an $\bbE_\infty$ orientation $MUP\rightarrow KU$.
\item Let $E$ be a Lubin-Tate spectrum at height $h=2$. Then every norm-coherent coordinate on $\bbG$ refines to an $\bbE_\infty$ orientation $MUP\rightarrow E$.
\end{enumerate}
\end{theorem}
\begin{proof}
Claims (1) and (3) are immediate consequences of \cref{thm:emapss}, as $E_0 MUP$ is smooth over $E_0$ (in fact $E_0MUP \cong E_0[t_0^{\pm 1},t_1,\ldots]$; the standard references are \cite{adams1974stable} and \cite[Example 2.51]{andohopkinsstrickland2001elliptic}). Claim (2) follows directly from (1), the arithmetic fracture square, and the fact that $x+y-xy$ is a norm-coherent coordinate of the multiplicative formal group at all primes; for completeness we give the details. First, arithmetic fracture gives a Cartesian square of the form
\begin{center}\begin{tikzcd}
\CAlg(MUP,KU)\ar[r]\ar[d]&\prod_p\CAlg_{KU_p}(KU_p\cotimes MUP,KU_p)\ar[d]\\
\CAlg(MUP_\bbQ,KU_\bbQ)\ar[r]&\CAlg(MUP_\bbQ,(\prod_p KU_p)_\bbQ)
\end{tikzcd}.\end{center}
As $MUP_\bbQ$ is free as a rational $\bbE_\infty$ ring, equivalent to the Eilenberg-MacLane spectrum of the graded Lazard ring, the coordinate $x+y-xy$ gives points of the bottom two spaces, and $\pi_1\CAlg(MUP_\bbQ,(\prod_p KU_p)_\bbQ) = 0$. So it is sufficient to verify that for each prime $p$, the homotopy orientation associated to the formal group law $x+y-xy$ refines uniquely to a map $KU_p\cotimes MUP\rightarrow KU_p$ of $KU_p$-algebras. By \cref{thm:emapss}, it is sufficient to verify that the coordinate associated to the formal group law $x+y-xy$ is norm-coherent at all primes.

The description of norm-coherent coordinates given in \cite[Section 4]{andohopkinsstrickland2004sigma}, in the case of a Lubin-Tate spectrum $E$, can be summarized as follows. A norm-coherent coordinate $x$ on $\bbG$ is a coordinate such that for every formal scheme $Y$, map $f\colon Y\rightarrow X$, and finite subgroup $K\subset f^\ast\bbG$, we have $N_\pi \mu^\ast f^\ast(x) = q^\ast\alpha^\ast(x)$, where $\pi\colon K\times_Y f^\ast \bbG\rightarrow f^\ast\bbG$ is the projection, $N_\pi$ is the associated norm map, $\mu\colon K\times_Y f^\ast\bbG\rightarrow f^\ast\bbG$ is the multiplication, and $\alpha\colon (f^\ast \bbG)/K\rightarrow\bbG$ identifies $(f^\ast\bbG)/K$ as a deformation of $\bbG_0$. To be precise, \cite{andohopkinsstrickland2004sigma} works in the context of level structures rather than finite subgroups, but the translation follows from the fact that $\coprod_{|A|=p^n}\Level(A,\bbG)\rightarrow\Sub^n_\bbG$ is surjective \cite[Theorem 12.4]{strickland1997finite}. In addition, it is sufficient to restrict to the case where $K$ is a subgroup of rank $p$.

When $h=1$, the kernel of formal multiplication by $p$ is the unique subgroup of rank $p$. The $p$-series associated to the formal group law $x+y-xy$ is given by $[p](x) = 1-(1-x)^p$, so the above condition translates to asking that multiplication by $x+y-xy$ on the free $\bbZ_p[[x]]$-module $\bbZ_p[[x,y]]/(1-(1-y)^p) = \bbZ\{1,y,\ldots,y^{p-1}\}$ has determinant $1-(1-x)^p$. This itself can be checked by direct calculation, proving the theorem.
\end{proof}

\begin{rmk}
Once an $\bbE_\infty$ orientation $MUP\rightarrow E$ is chosen, the theory of Thom spectra and orientations furnishes an equivalence $\CAlg(MUP,E)\simeq \Map_{\Sp}(ku,\operatorname{gl}_1(E))$. Using this, Senger has shown that the height $h=2$ orientations guaranteed to exist by \cref{thm:orientations}(3) are in fact unique \cite[Remark 3.6]{senger2022obstruction}. By \cref{thm:emapss}(5), this is equivalent to the assertion that $\Ext^2_\Delta(\widehat{Q}(E_0^\wedge MUP),\omega) = 0$.
\tqed
\end{rmk}

\begin{rmk}\label{rmk:ht1orspace}
At height $h=1$, in \cite[Corollary 3.13]{mollers2012k1} it is shown that the choice of an orientation $MU\rightarrow E$ gives a weak equivalence $\CAlg(MU,E)\simeq\Map(KU_p,E)$. This is reflected in the algebra of power operations by the following: there is an isomorphism
\[
\widehat{Q}(E_0^\wedge BU)\cong\Delta\cotimes E_0^\wedge KU_p.
\]
The inclusion $E_0^\wedge KU_p\rightarrow\widehat{Q}(E_0^\wedge BU)$ which extends by $\Delta$-linearity to this isomorphism may be obtained from the map $KU_p\rightarrow L_{K(1)}\Sigma^\infty_+ BU$, itself obtained by applying the Bousfield-Kuhn functor \cite{kuhn2008guide} to the unit $BU\rightarrow \Omega^\infty\Sigma^\infty_+\Omega^\infty BU$. This map is special to height $1$, and indeed $\widehat{Q}(E_0^\wedge BU)$ is no longer projective over $\Delta$ at higher heights (\cref{rmk:noproj}).
\tqed
\end{rmk}

We now give a few more examples illustrating \cref{thm:emapss}.

\begin{ex}
In Chatham-Hahn-Yuan \cite{chathamhahnyuan2019wilson}, an interesting family of $\bbE_\infty$ ring spectra $R_{h-1}$ at chromatic height $h$ are constructed, and left open is the question of whether there exists an $\bbE_\infty$ map $R_{h-1}\rightarrow E$, where $E$ is a Lubin-Tate spectrum of height $h$. Combining \cite[Theorem 7.6]{chathamhahnyuan2019wilson} with the preceding, we learn that there are $\bbE_\infty$ maps $R_1\rightarrow E$ whenever $E$ is a Lubin-Tate spectrum of height $2$ associated to a supersingular elliptic curve.
\tqed
\end{ex}

\begin{ex}\label{ex:gmku}
Given an arbitrary $\bbE_\infty$ ring spectrum $R$, we may define
\[
\bbA^1(R) = \CAlg(\Sigma^\infty_+ \bbN,R),\qquad \bbG_m(R) = \CAlg(\Sigma^\infty_+ \bbZ,R).
\]
We considered the case where $R$ is an $\bbE_\infty$ algebra over $\bbF_p$ in \cref{ex:affineline}. By construction, $\bbA^1(R)$ carries the structure of a strictly commutative monoid. The subspace $\bbG_m(R)\subset\bbA^1(R)$ is a collection of path components, and can be regarded as a $\bbZ$-module, i.e.\ deloops to a connective $H\bbZ$-module. In particular $\pi_n \bbG_m(R)$ is indepenent of the choice of basepoint for all $n\geq 1$; the same is not generally true for $\bbA^1(R)$, as seen in the following example.

We can describe $\bbA^1(KU)$ as a simple example illustrating the use of \cref{thm:emapss}. Write $\widehat{\bbZ} = \prod_p \bbZ_p$ for the profinite integers. Then $\pi_0 \bbA^1(KU) = \{-1,0,1\}$,
\[
\pi_n\bbG_m (KU) = \begin{cases}
\widehat{\bbZ},&n=1;\\
\bbZ,&n=2;\\
\widehat{\bbZ}/\bbZ,&n>2\text{ odd};
\end{cases}
\]
and there are short exact sequences
\[
0\rightarrow \widehat{\bbZ}/\bbZ\rightarrow \pi_{2n+1}(\bbA^1 (KU),0)\rightarrow  \prod_{p}\bbZ/(p^n)\rightarrow 0,
\]
which are necessarily split as $\widehat{\bbZ}/\bbZ$ is injective. Here, the above product ranges over primes $p$, and all unspecified groups are zero. 

This is computed as follows. Arithmetic fracture gives a Cartesian square
\begin{center}\begin{tikzcd}
\bbA^1 (KU)\ar[r]\ar[d]&\prod_p  \bbA^1 \left(KU_p \right)\ar[d]\\
\bbA^1\left( KU_\bbQ\right)\ar[r]&\bbA^1\left(\left(\prod_p  KU_p \right)_\bbQ\right)
\end{tikzcd}.\end{center}
If $R$ is a rational $\bbE_\infty$ ring, then $\bbA^1 (R) = \Omega^\infty R$, and this determines the bottom two spaces. The claimed structure of $\bbA^1 (KU)$ will follow easily by inspecting this square as soon as we understand $\bbA^1(KU_p)$. To that end we claim that $\pi_0 \bbA^1 (KU_p )\cong\bbF_p\subset \bbZ_p = \pi_0 KU_p $ is given by the image of the Teichm\"uller character, and that
\[
\pi_1 \bbG_m (KU_p ) \cong \bbZ_p \cong \pi_2 \bbG_m (KU_p ),\qquad \pi_{2n+1} (\bbA^1 (KU_p) ,0) \cong \bbZ/(p^n),
\]
all other groups being zero. 

One approach to computing this proceeds by showing that $\bbA^1 (KU_p )$ fits into a fiber sequence
\[
\bbA^1 \left(KU_p \right)\rightarrow \Omega^\infty KU_p \rightarrow\Omega^\infty KU_p ,
\]
with second map corresponding to the $KU_p $-cohomology operation $\theta^p$. As we wish to illustrate the use of \cref{thm:emapss}, we will proceed in a different way, although the two approaches are not really any different.

The goal is to compute $\pi_\ast\bbA^1(KU_p ) = \pi_\ast \CAlg_{KU_p }(KU_p \cotimes \Sigma^\infty_+ \bbN,KU_p )$. As $\theta$-rings,
\begin{gather*}
\pi_0 KU_p  \cong \bbZ_p,\qquad \psi(\lambda) = \lambda;\\
\pi_0 KU_p \otimes\Sigma^\infty_+ \bbN \cong \bbZ_p[t],\qquad \psi(t) = t^p;
\end{gather*}
so \cref{thm:emapss} gives an isomorphism
\[
\pi_0\bbA^1(KU_p ) \cong \Ring_\bbT(\bbZ_p[t],\bbZ_p) \cong \{\lambda\in\bbZ_p : \lambda^p = \lambda\}.
\]
This is the image of the Teichm\"uller character as claimed.

To get at $\pi_\ast \bbA^1(KU_p)$, given some element $\phi$ of the above, we must compute the groups
\[
H^\ast_{\bbT/\bbZ_p}(\bbZ_p[t],\pi_\ast \Omega^{2n} KU_p) \cong \Ext_{\Delta}^\ast(Q(\bbZ_p[t]),\omega^n)
\]
for $n\geq 1$. Here, $ \omega^n = \bbZ_p\{\beta^n\}$ has $\Gamma = \bbZ_p[\psi]$-module structure $\psi(\beta^n) = p^n \beta^n$, and thus $\Delta = \bbZ_p[\theta]$-module structure $\theta(\beta^n) = p^{n-1} \beta^n$.

Consider first the case where $\phi$ is in a path component corresponding to an element of $\bbF_p^\times$. These are all equivalent, so it is sufficient to consider the map
\[
\phi\colon \bbZ_p[t]\rightarrow\bbZ_p,\qquad \phi(t) = 1.
\]
With this augmentation, $Q(\bbZ_p[t]) = \bbZ_p\{s\}$ where $s$ is the class of $t-1$, and this has $\Gamma$-module structure $\psi(s) = ps$, and thus $\Delta$-module structure $\theta(s) = s$. The Koszul complex for $\Ext_\Delta(\bbZ_p\{s\},\omega^n)$ takes the form
\[
p^{n-1}-1\colon \bbZ_p\rightarrow\bbZ_p.
\]
As $p^{n-1}-1$ is a unit unless $n=1$, the only nonzero groups are
\[
\Ext^0_\Delta(\bbZ_p\{s\},\omega) = \bbZ_p = \Ext^1_\Delta(\bbZ_p\{s\},\omega).
\]
\cref{thm:emapss} then implies that $\pi_1 \bbG_m(KU_p ) = \bbZ_p  = \pi_2 \bbG_m(KU_p )$ as claimed. 

Next consider the map
\[
\phi\colon \bbZ_p[t]\rightarrow \bbZ_p,\qquad \phi(t) = 0.
\]
With this augmentation, $Q(\bbZ_p[t]) = \bbZ_p\{t\}$ with $\Delta$-module structure $\theta(t)=0$, and the Koszul complex for $\Ext_\Delta(\bbZ_p\{t\},\omega^n)$ takes the form
\[
p^{n-1}\colon \bbZ_p\rightarrow\bbZ_p.
\]
The resulting nonzero groups are
\[
\Ext^1_\Delta(\bbZ_p\{t\},\omega^n) = \bbZ/(p^{n-1}).
\]
\cref{thm:emapss} implies that $\pi_{2n+1} \bbA^1(KU_p ) = \bbZ/(p^n)$ as claimed.
\tqed
\end{ex}

\begin{ex}
  Let $G$ be a finite group. It is known that $G$ may be recovered from its symmetric monoidal category of complex representations, but not from its complex representation ring $RU(G)$, even when considered as a lambda ring; see \cite{meirszymik2021adams} for a discussion. Now suppose that $G$ is a finite $p$-group, and let $E$ be the Lubin-Tate spectrum associated to the formal multiplicative group over an algebraically closed field $\kappa$ of characteristic $p$; informally, $E = KU_p$, only with coefficients extended along $\bbZ_p\rightarrow W(\kappa)$. Then there is an equivalence
  \[
    BG\simeq\CAlg_E(E^{BG_+},E),
  \]
as can be see from the easily verified height $1$ case of \cite[Conjecture 5.4.14]{hopkinslurie2013ambidexterity}.
By the $p$-adic Atiyah-Segal completion theorem \cite[Proposition III.1.1]{atiyahtall1969group}, $E^1 BG = 0$ and $E^0 BG = W(\kappa)\otimes RU(G)$. \cref{thm:emapss} then gives a curious filtration of $BG$, interpolating between the data encoded by the $\theta$-ring $RU(G)$ and the group $G$ itself; to phrase it in terms of a fringed spectral sequence, this is
\[
E_1^{2p,q} = \Ext_{\bbZ_p[\theta]}^{2p-q}(\bbL Q(R(G));\omega^p)\Rightarrow \pi_q BG,\qquad d_{2r}^{2p,q}\colon E_{2r}^{2p,q}\rightarrow E_{2r}^{2(p+r),q-1},
\]
where $\omega^p = \pi_{2p} E$.
\tqed
\end{ex}

\begin{rmk}
Power operations also play a role in other computations with $E$-algebras. We illustrate this with an example. For $R,S\in\CAlg_E^{\loc}$, there is a K\"unneth spectral sequence
\[
E^1_{p,q} = \pi_{q-p}\left(R_\ast\cotimes_{E_\ast}^\bbL (\omega^{p/2}\otimes S_{\ast})\right)\Rightarrow \pi_q(R\cotimes_E S),\qquad d^r_{p,q} \colon E^r_{p,q}\rightarrow E^r_{p+r,q-1}.
\]
Observe that $E^1_{p,q}$ carries the structure of a $\Gamma$-module for $p \geq 0$. Write $Z^r_{p,q}$ and $B^r_{p,q}$ for the $r$-cycles and $r$-boundaries of this spectral sequence, so that
\[
0 = B^0_{p,q}\subset B^1_{p,q}\subset\cdots\subset Z^1_{p,q}\subset Z^0_{p,q} = E^1_{p,q},\qquad E^{r+1}_{p,q} = Z^r_{p,q}/B^r_{p,q}.
\]
It follows from the discussion in \cite[Section 5.4]{balderrama2021deformations} (especially \cite[Propositions 5.4.9, 5.4.11]{balderrama2021deformations}) that $Z^r_{p,q}\subset E^1_{p,q}$ is a sub $\Gamma$-module for all $p \geq 0$ and $B^r_{p,q}\subset Z^r_{p,q}$ is a sub $\Gamma$-module for all $p \geq r$, and that $d^r_{p,q}\colon Z^r_{p,q}\rightarrow E^r_{p+r,q-1}$ is a map of $\Gamma$-modules for all $p \geq 0$.
\tqed
\end{rmk}

\subsection{Topological Andr\'e-Quillen cohomology}\label{ssec:etaq}

We now consider the $K(h)$-local topological Andr\'e-Quillen homology and cohomology of $E$-algebras, where $E$ is a Lubin-Tate spectrum of height $h$ as in the preceding subsections. This is of particular interest due to work of Behrens-Rezk \cite{behrensrezk2017spectral} \cite{behrensrezk2012bousfield}, which constructs for a space $X$ a natural comparison map $\Phi_h X\rightarrow\TAQ_{\bbS_{K(h)}}(\bbS_{K(h)}^{X_+})$, showing it to be an isomorphism in some nice cases. Here, $\Phi_h$ is the $K(h)$-local Bousfield-Kuhn functor \cite{kuhn2008guide}. This gives rise to a comparison map $E\cotimes \Phi_h X\rightarrow\TAQ_E(E^{X_+})$, again an isomorphism in some nice cases. This gives an approach to computing $E_\ast \Phi_h X$, which in turn gives an approach to computing $\pi_\ast \Phi_h X$ by descent along $\bbS_{K(h)}\rightarrow E$.

We must introduce some notation. Given a $K(h)$-local $\bbE_\infty$ ring spectrum $R$, and $M\in\Mod_R^\loc$, write $\wTAQ^R(A;M) = L_{K(h)}\TAQ^R(A;M)$ for the $K(h)$-local Andr\'e-Quillen homology of $A\in\CAlg_R^{\loc,\aug}$ with coefficients in $M$. This construction can be characterized as the unique functor
\[
\wTAQ^R(\bs;M)\colon \CAlg_R^{\loc,\aug}\rightarrow\Mod_R^\loc
\]
which preserves geometric realizations and satisfies
\[
\wTAQ^R(\widehat{\bbP}_R N; M)= M\cotimes_R N
\]
for $N\in\Mod_R^\loc$. See \cite[Section 3]{kuhn2003localization} for further discussion; note that we write $\wTAQ^R(A;M)$ for what might there be written as $L_{K(h)}(M \wedge \operatorname{taq}^{K(h)}_R(A))$. In addition, write $\TAQ_R(A;M) = \Mod_R(\wTAQ^R(A,R),M)$. When $M = R$, we omit it from the notation.

On the algebraic side, given $R\in\Ring_\bbT^{\Cpl(\frakm),\heart}$ and $M\in\Mod_R^{\Cpl(\frakm),\heart}$, define
\[
\wAQ^{R}(\bs;M)\colon \Ring_{R\otimes\bbT}^{\Cpl(\frakm),\aug}\rightarrow\Mod_R^{\Cpl(\frakm)}
\]
to be the unique functor preserving geometric realizations and satisfying
\[
\wAQ^{R}(R\cotimes^\bbL_{E_\ast}\widehat{\bbT}(P)) = M\cotimes_{E_\ast}^\bbL P
\]
for $P\in\Mod_{E_\ast}^{\Cpl(\frakm),\free}$. In addition, set
\[
\AQ_{R}(A;M) = \EXT_R(\wAQ^{R}(A;R),M).
\]
Observe that
\[
\wAQ^{R}(A;M) \simeq M\cotimes_R^\bbL \epsilon_! \bbL\widehat{Q}_R(A)\simeq \ol{M}\cotimes_{R\otimes\Delta}^\bbL \bbL\widehat{Q}_R(A),
\]
where $\epsilon\colon R\otimes\Delta\rightarrow R$ is the augmentation and $\epsilon_!$ should be interpreted in the derived sense, and that
\[
\AQ_{R}(A;M) \simeq \calH_{R\otimes\bbT/R}(A;\ol{M}).
\]

The following theorem generalizes \cite[Proposition 4.7]{behrensrezk2012bousfield}.

\begin{theorem}\label{thm:eaqss}
Fix $R\in\CAlg_E^\loc$ and $M\in\Mod_R^\loc$, and choose $A\in\CAlg_R^{\loc,\aug}$. Then there is a conditionally convergent spectral sequence of signature
\[
E_1^{p,q} = \AQ_{R_\ast}^{p+q}(A_\ast;\omega^{-p/2}\otimes M_\ast)\Rightarrow \TAQ^q_R(A;M),\qquad d_r^{p,q}\colon E_r^{p,q}\rightarrow E_r^{p+r,q+1}.
\]
If $R_\ast$ is tame, then there is a spectral sequence of signature
\[
E^1_{p,q} = \wAQ_{p+q}^{R_\ast}(A_\ast;\omega^{p/2}\otimes M_\ast)\Rightarrow\wTAQ^R_q(A;M),\qquad d^r_{p,q}\colon E^r_{p,q}\rightarrow E^r_{p-r,q-1},
\]
which is convergent if each $\wAQ^{R_\ast}(A_\ast;\omega^{p/2}\otimes M_\ast)$ is truncated.
\end{theorem}
\begin{proof}
The first spectral sequence can be obtained by patching together the filtrations of $\CAlg_R^\aug(A;R\ltimes\Sigma^n M)$ for various $n$ given by \cref{thm:emapss}. For the second, tameness of $R$ guarantees that $\Model_{\h\CAlg_R^{\loc,\aug,\free}}\simeq\Ring_{R_\ast\otimes\bbT/R_\ast}^{\Cpl(\frakm)}$ and $\LMod_{\h\Mod_R^{\loc,\free}}\simeq\Mod_{R_\ast}^{\Cpl(\frakm)}$, and the spectral sequence can be obtained by applying \cite[Theorem 4.2.2]{balderrama2021deformations}.
\end{proof}

We expect it could be possible to remove the tameness assumption in \cref{thm:eaqss}, but we have no reason to do so. The most important case is $R = E = M$, and we will write $\nul = \ol{E}_\ast \in\LMod_\Delta$ and $\AQ_{E_\ast} = \AQ$.

\begin{rmk}\label{rmk:omegashift}
The action of $\Delta$ on $\omega^{p/2}\otimes \ol{M}_\ast$ as it appears in \cref{thm:eaqss} is through the augmentation on $\Delta$; the presence of $\omega^{p/2}$ only serves to shift degrees. In particular, once a trivialization of $\omega$ is chosen, $\omega^{p/2} \otimes \ol{M}_\ast = s^{-p}\ol{M}_\ast$ depends only on the congruence class of $p$ mod $2$. However, these powers of $\omega$ also serve to track additional structure that is present in examples of interest, such as actions by the Morava stabilizer group.
\tqed
\end{rmk}

\begin{ex}\label{rmk:noproj}
In \cref{rmk:ht1orspace}, we asserted that $\widehat{Q}(E_0^\wedge BU)$ is not projective over $\Delta$ at heights $h\geq 2$. Indeed, if it were then the spectral sequence of \cref{thm:eaqss} would collapse, and in particular $\wTAQ_E(E\cotimes\Sigma^\infty_+ BU)$ would be some nonzero $E$-module. But $\wTAQ_E(E\cotimes \Sigma^\infty_+ BU)\simeq E\cotimes KU_{\geq 1}\simeq 0$ as $L_{K(h)} KU_{\geq 1}\simeq 0$ for $h\geq 2$.
\tqed
\end{ex}

In \cite{rezk2013power} and \cite{zhu2018morava}, $E_\ast^\wedge \Phi_h S^{2k+1}$ is computed for $h\leq 2$; the computation proceeds by computing the cohomology groups $\AQ^n(E^\ast S^{2k+1};\omega^{m/2}\otimes\nul)$. In particular, it is shown in these $h\leq 2$ cases that this group vanishes for $n\neq h$. In general, let us say that $E$ satisfies the \textit{weak algebraicity condition} if $\AQ^n(E^\ast S^{2k+1};\omega^{m/2}\otimes \nul) = 0$ for $n\neq h$.

Work of Bousfield \cite{bousfield1998ktheory} \cite{bousfield2007adic} describes the $v_1$-periodic homotopy groups of nice spaces in terms of their $p$-adic $K$-theory. One obstruction to extending this to higher heights using $\TAQ$ is in determining when $\Phi_h X \simeq \TAQ_{\bbS_{K(h)}}(\bbS_{K(h)}^{X_+})$; we will not consider this issue here. Provided one takes this as known, Bousfield's description of $\pi_\ast\Phi_1 X$ for nice spaces $X$ at primes $p\geq 3$ can be reinterpreted as a description of $KU_{p,\ast}^\wedge\Phi_1 X$ that can be obtained from \cref{thm:eaqss}, combined with the standard fiber sequence $\Phi_1 X \rightarrow KU_p\cotimes \Phi_1 X \rightarrow KU_p\cotimes \Phi_1 X$. We view the following observation as the first part of a higher height analogue.

\begin{prop}\label{prop:taqalg}
Suppose $E$ satisfies the weak algebraicity condition defined above. Let $X$ be a simply connected space such that $H^\ast(X)$ is a finitely generated exterior algebra on odd-dimensional classes. Then 
\[
\TAQ_E^q(E^{X_+}) \cong\Ext^h_\Delta(Q(E^\ast X);\omega^{(q-h)/2}\otimes\nul).
\]
\end{prop}
\begin{proof}
Write $H^\ast X\simeq \Lambda(t_1,\ldots,t_n)$ with $|t_i| = m_i$. The Atiyah-Hirzebruch spectral sequence collapses to give $E^\ast X\simeq \Lambda_{E_\ast}(x_1,\ldots,x_n)$. More precisely, there is a cell structure on $X$ with the following property. Write $X_{\leq n}$ for the $n$-skeleton of $X$. Then the cofibering $X_{\leq n-1}\rightarrow X_{\leq n}\rightarrow \bigvee S^n$ induces a short exact sequence on $H^\ast$ and $E^\ast$, and the restriction of $x_i$ to $X_{\leq m_i}$ is the image of a generator of some $E^\ast S^{m_i}$. Now $\bbL Q(E^\ast X) = Q(E^\ast X) = E_\ast\{x_1,\ldots,x_n\}$ as $E_\ast$-modules, each $E_\ast\{x_i,\ldots,x_n\}\subset Q(E^\ast X)$ is a sub-$\Delta$-module, and $E_\ast\{x_i,x_{i+1},\ldots,x_n\}/E_\ast\{x_{i+1},\ldots,x_n\}\cong\omega^{m_i/2}$. This gives a finite filtration of the $\Delta$-module $Q(E^\ast X)$ with filtration quotients given by various $\omega^{m_i/2}$, and the associated spectral sequence for $\Ext_\Delta^\ast(Q(E^\ast X);\omega^{p/2}\otimes\nul)$ collapses by the weak algebraicity condition, implying that it is concentrated in degree $h$. This implies that the spectral sequence of \cref{thm:eaqss} collapses on a single line into the claimed isomorphism.
\end{proof}

\begin{ex}\label{ex:su3}
At heights $h\leq 2$, at least for small primes, the algebraic input to \cref{thm:emapss} and \cref{thm:eaqss} is computationally accessible. As our examples so far have been of a general nature, let us illustrate this with an explicit computation of $\TAQ^\ast_E(E^{SU(4)_+})$ for $E$ the Lubin-Tate spectrum of height $2$ considered in \cref{ex:ltcurve}. There is nothing special to $SU(4)$; the method of computation may be applied more generally, and we have just chosen the simplest interesting example beyond the case of an odd-dimensional sphere already computed by Zhu \cite{zhu2018morava}. 

Consider $SU(n)$ in general. This space satisfies the conditions of \cref{prop:taqalg}, so to compute $\TAQ^\ast_E(E^{SU(n)_+})$ it is sufficient to compute $\Ext^2_\Delta(Q(E^\ast SU(n)),\omega^{1/2}\otimes\nul)$. This is carried out in two steps; first one must understand the $\Delta$-module $Q(E^\ast SU(n))$, and then one must carry out the $\Ext$ calculation. The first step might be regarded as homotopical input to the calculation, in the sense that it is not an instance of the general algebra of $E$-power operations. The second step is then purely algebraic. Following our discussion in \cref{ex:ltcurve}, the image of $Q(E^\ast SU(n))$ under the Morita equivalence $\omega^{-1/2}\otimes\bs\colon \LMod_\Delta\rightarrow\LMod_\Gamma$ is determined by the $\Gamma$-module structure of $Q(E^\ast SU(n))$. Thus the first step reduces to understanding $Q(E^\ast SU(n))$ as a $\Gamma$-module, and the second step transforms into computing $\Ext^2_\Gamma(\omega^{-1/2}\otimes Q(E^\ast SU(n)),\nul)$.

There is a suspension map $\Sigma\Sigma^\infty SU(n)\rightarrow\Sigma^\infty BSU(n)$, and this induces an isomorphism $Q(E^\ast SU(n))\cong Q(E^\ast BSU(n))$ of $\Gamma$-modules. As $E^0 BSU(n) = E_0[[c_2,c_3,\ldots]]/(c_{n+1},\ldots)$, and as $E^0 BSU = E_0[[c_2,c_3,\ldots]]$ sits inside $ E^0 BU=E_0[[c_1,c_2,\ldots]]$ in the obvious way, one therefore reduces to computing the action of $\Gamma$ on $E^0 BU$ modulo indecomposables.

Write $E^0 BU(1) = E_0[[u]]$. The summation maps $BU(1)^{\times m}\rightarrow BU(m)\rightarrow BU$ induce maps $E^0 BU \rightarrow E_0[[u_1,\ldots,u_m]]$ sending $c_j$ to the $j$'th elementary symmetric polynomial in $u_1,\ldots,u_m$, and in the limiting case this identifies $E^0 BU$ as the ring of symmetric power series in the variables $u_i$.  Recall from \cref{ex:ltring} the coalgebraic interpretation of $\Gamma$-modules. As mentioned at the end of \cref{ssec:epow}, Schumann \cite{schumann2014k2} has computed the coaction $P$ on $E^0 BU(1)$ in closed form to be $P(u) = \frac{u^2-du}{1+d^2 u}$. Given a symmetric function $s$ and function $f$, write $s\wr f$ for the symmetric function $(s\wr f)(u_1,\ldots) = s(f(u_1),\ldots)$. Then putting everything together, the coaction on $E^0 BU$ is the map
\[
P'\colon E_0[[c_1,c_2,\ldots]]\rightarrow E_0[[c_1,c_2,\ldots]][d]/(d^3=ad+2),\qquad P'(c_m) = c_m\wr P.
\]

By Newton's identities, if $P(u)^m = \sum_j \alpha_j u^j$ with $\alpha_j\in E_0[d]/(d^3=ad+2)$, then $P'$ reduces modulo indecomposables to $P'(c_m) \equiv \frac{(-1)^m}{m} \sum_j (-1)^j j\alpha_jc_j$.
Finally, the coaction on $\omega^{-1/2}\otimes Q(E^\ast SU(n)) = E_0\{c_2,\ldots,c_n\}$ is given by $P'' = \frac{-1}{d}P'$.

Specialize now to $n=4$. Write $M = \omega^{-1/2}\otimes Q(E^\ast SU(4)) = E_0\{c_2,c_3,c_4\}$. As
\begin{align*}
P(u)^2 &= d^2 u^2 - 2(ad^2+3d)u^3 + (1+4d^3+3d^6)u^4+O(u^5)\\
P(u)^3 &= -d^3 u^3 + 3(d^2+d^5)u^4+ O(u^5)\\
P(u)^4 &= d^4 u^4 + O(u^5),
\end{align*}
it follows that
\begin{align*}
P'(c_2)&=d^2c_2+3(d+d^4)c_3+2(1+4d^3+3d^6)c_4 \\
P'(c_3)&=-d^3c_3-4(d^2+d^5)c_4 \\
P'(c_4) &= d^4 c_4.
\end{align*}
Using $\frac{2}{-d}=a-d^2$, we divide by $-d$ to obtain
\begin{align*}
P''(c_2) &= -dc_2 - 3(1+d^3)c_3 - (d^2-a)(1+4d^3+3d^6)c_4 \\
P''(c_3) &= d^2 c_3 + 4(d+d^4)c_4 \\
P''(c_4) &= -d^3 c_4.
\end{align*}
Somewhat abusively, write $M^\vee = \{c_2,c_3,c_4\}$, where $c_n$ is dual to $c_n$. Then the adjoint map $P^\vee\colon M^\vee\rightarrow M^\vee\tins{}{t}\Gamma[1]^\vee$ is given by
\begin{align*}
P^\vee(c_2) &= -c_2d \\
P^\vee(c_3) &=  c_3d^2  -3c_2(1+d^3)\\
P^\vee(c_4) &= -c_4d^3  + 4c_3(d+d^4) - c_2(d^2-a)(1+4d^3+3d^6).
\end{align*}
By \cref{thm:homogkoszuldifferential} and \cref{ex:l2koszul}, the Koszul differential $\delta\colon K^1(M;\nul)\rightarrow K^2_\Gamma(M;\nul)$ is the composite
\begin{center}\begin{tikzcd}
{M^\vee\otimes_{t}\Gamma[1]^\vee}\ar[r,"{P^\vee\otimes\Gamma[1]^\vee}"]\ar[d,"\delta",dashed]&{M^\vee\otimes_{t}\Gamma[1]^\vee\tins{s}{t}\Gamma[1]^\vee}\ar[d,"M^\vee\otimes f"]\\
{M^\vee\otimes_{t}\Gamma[1]^\vee/s(E_0)}&{M^\vee\otimes_{t}\Gamma[1]^\vee}\ar[l]
\end{tikzcd}.\end{center}
For example, as $P^\vee(c_2) = -c_2d$, it follows that $(P^\vee\otimes\Gamma[1]^\vee)(c_2d^2) = -c_2d'd^2$; this is sent to $-c_2(a-d^2)d^2 = 2c_2d$ under $M^\vee\otimes f$, and thus $\delta(c_2d^2) = 2  c_2 d$. 

In general, we compute $\delta$ to be given by
\begingroup\allowdisplaybreaks
\begin{align*}
\delta(c_2) &= c_2d^2 \\
\delta(c_2d) &= 0\\
\delta(c_2d^2) &= c_2(2d) \\
\delta(c_3) &= c_3(2d-ad^2) + c_2(-6ad+3a^2d^2) \\
\delta(c_3 d) &= c_3(2d^2) + c_2(9d - 6ad^2) \\
\delta(c_3 d^2) & = c_2(9d^2)\\
\delta(c_4) &= c_4(-2a d + a^2 d^2) + c_3(8a^2d + (12 - 4 a^3)d^2) \\
&+ c_2((-66  - 40 a^2  + 88 a^3  + 6 a^5  - 6 a^6 )d 
+ (113 a  + 32 a^3  - 56 a^4  - 3 a^6  + 3 a^7 )d^2) \\
\delta(c_4 d) &= c_4(4 d - 2 a d^2) + c_3(-16 a d + 8 a^2 d^2) \\
&+ c_2((33 a  - 128 a^2  - 12 a^4  + 12 a^5)d 
+ (-66  - 40 a^2  + 88 a^3  + 6 a^5  - 6 a^6 )d^2) \\
\delta(c_4d^2) &= c_4(4 d^2) + c_3(24 d - 16 a d^2) \\ 
&+ c_2((160 a  + 24 a^3  - 24 a^4 )d 
+ (33 a  - 128 a^2  - 12 a^4  + 12 a^5 )d^2).
\end{align*}\endgroup
Here one can observe the manner in which $K_\Gamma(M;\nul)$ is built from $K_\Gamma(\omega^n;\nul)$ for $1\leq n \leq 3$; compare  \cite[Example 4]{zhu2018morava}. In the end, $\TAQ^0(E^{SU(4)_+})=0$, and if we write $w = c_3 d$, $x = c_3 d^2$, $y = c_4 d$, $z = c_4 d^2$, then $\TAQ^1(E^{SU(4)_+})$ is isomorphic to $E_0\{w,x,y,z\}$ modulo
\begin{align*}
4x &= 0 \\
2w &= ax \\
4z &= 0 \\
2ay &= a^2 z \\
4y &= 2a(x+z).
\end{align*}
By \cite[Proposition 8.8]{behrensrezk2017spectral}, this also describes $E_\ast^\wedge\Phi_2 SU(4)$.
\tqed
\end{ex}

\begingroup
\raggedright
\bibliography{refs}
\bibliographystyle{alpha}
\endgroup
\end{document}